\documentclass[final]{siamart1116}

\usepackage{graphicx,epstopdf}
\usepackage[caption=false]{subfig}

\usepackage{lipsum}
\usepackage{amsfonts}
\usepackage{graphicx}
\usepackage{epstopdf}
\usepackage{amssymb}
\usepackage{algorithmic}
\ifpdf
  \DeclareGraphicsExtensions{.eps,.pdf,.png,.jpg}
\else
  \DeclareGraphicsExtensions{.eps}
\fi

\newcommand{\black}{\color{black}}

\newcommand{\overbar}[1]{\mkern 1.5mu\overline{\mkern-1.5mu#1\mkern-1.5mu}\mkern 1.5mu}

\DeclareMathOperator{\divv}{div}

\DeclareMathOperator{\id}{id }

\newcommand{\UU}{\mathrm{U}}

\newcommand{\R}{\mathbf R}

\usepackage{accents}

\newcommand{\ac}{\accentset{\circ}}
\newcommand{\DD}{\mathsf{\Omega}}
\newcommand{\eps}{\epsilon}
\newcommand{\y}{f}
\newcommand{\Sb}{\textbf{S}}
\newcommand{\Tb}{\textbf{T}}

\newcommand{\beqn}{\begin{eqnarray*}}
	\newcommand{\eeqn}{\end{eqnarray*}}

\newcommand{\ben}{\begin{equation}}
	\newcommand{\een}{\end{equation}}

\newcommand{\beq}{\begin{eqnarray}}
	\newcommand{\eeq}{\end{eqnarray}}

\newcommand{\benn}{\begin{equation*}}
	\newcommand{\eenn}{\end{equation*}}

\numberwithin{theorem}{section}

\newtheorem{remark}[theorem]{{\textbf Remark}}
\newtheorem{assumption}[theorem]{{\textbf Assumption}}

\newcommand{\TheTitle}{Optimal Actuator Design Based on Shape Calculus}
\newcommand{\TheAuthors}{Kalise, D., Kunisch, K. and Sturm, K.}

\headers{\TheTitle}{\TheAuthors}

\title{{\TheTitle}\thanks{D.K. and K.K. were partially funded by the ERC Advanced Grant OCLOC.}}

\author{
	Dante Kalise\thanks{Department of Mathematics, Imperial College London, South Kensington Campus,
		London SW7 2AZ, United Kingdom (\email{dkaliseb@ic.ac.uk}),}
	\and
	Karl Kunisch\thanks{Johann Radon Institute for Computational and Applied Mathematics, Austrian Academy of Sciences \& Institute of Analysis and Scientific Computing,  University of Graz, Heinrichstr. 36, 8010 Graz, Austria (\email{karl.kunisch@uni-graz.at})}
	\and
	Kevin Sturm \thanks{Institute for Analysis and Scientific Computing, TU Wien, Wiedner Hauptstr. 8-10, 1040 Wien, Austria
		(\email{kevin.sturm@tuwien.ac.at}).}
}

\usepackage{amsopn}

\ifpdf
\hypersetup{
  pdftitle={\TheTitle},
  pdfauthor={\TheAuthors}
}
\fi

\newcommand{\cJ}{\mathcal{J}}

\begin{document}
	
	\maketitle
	
\begin{abstract}
An approach to optimal actuator design based on shape and topology optimisation techniques is presented. For linear diffusion equations, two scenarios are considered. For the first one, best actuators are determined depending on a given initial condition. In the second scenario, optimal actuators are determined based on all initial conditions not exceeding a chosen norm. Shape and topological sensitivities of these cost functionals are determined. A numerical algorithm for optimal actuator design based on the sensitivities and a level-set method is presented. Numerical results support the proposed methodology.
\end{abstract}

	
	
	\section{Introduction}
	
	In engineering, an actuator is a device transforming an external signal into a relevant form of energy for the system in which it is embedded. Actuators can be mechanical, electrical, hydraulic, or magnetic, and are fundamental in the control loop, as they materialise the control action within the physical system. Driven by the need to improve the performance of a control setting, actuator/sensor positioning and design is an important task in modern control engineering which also constitutes a challenging mathematical topic. Optimal actuator positioning and design departs from the standard control design problem where the actuator configuration is known a priori, and addresses a higher hierarchy problem, namely, the optimisation of the \textsl{control to state map}.
	
There is no unique framework which is followed to address optimal actuator problems. However, concepts which immediately suggest themselves -at least for linear dynamics- and which have been addressed in the literature, build on choosing actuator design in such a manner that stabilization or controllability are optimized by an appropriate choice of the controller. This can involve Riccati equations from linear-quadratic regulator theory, and appropriately chosen parameterizations of the set of admissible actuators. The present work partially relates to this stream as we optimise the actuator design based on the performance of the resulting control loop. Within this framework,  we follow a distinctly different approach by casting the optimal actuator design problem as  shape and  topology optimisation problems.   The class of admissible actuators are characteristic functions of measurable sets and their shape is determined by techniques from shape calculus and optimal control. The class of cost functionals which we consider within this work are quadratic ones and account for the stabilization of the closed-loop dynamics. We present the concepts here for the linear heat equation, but the techniques  can be extended to more general classes of functionals and stabilizable dynamical systems. We believe that the concepts of shape and topology optimisation constitute an important tool for solving  actuator positioning problems, and to our knowledge this can be the first step towards this direction. More concretely, our contributions in this paper are:

\begin{itemize}
	
	\item[i)] We study an optimal actuator design problem for linear diffusion equations. In our setting, actuators are parametrised as indicator functions over a subdomain, and are evaluated according to the resulting closed-loop performance for a given initial condition, or among a set of admissible initial conditions not exceeding a certain norm.
	
	\item[ii)] By borrowing a leaf from shape calculus, we derive shape and topological sensitivities for the optimal actuator design problem.
	
	\item[iii)] Based on the formulas obtained in ii),  we construct a gradient-based and a level-set method for the numerical realisation of optimal actuators.
	
	\item[iv)] We present a numerical validation of the proposed computational methodology. Most notably, our numerical experiments indicate that throughout the proposed framework we obtain non-trivial, multi-component actuators, which would be otherwise difficult to forecast based on tuning, heuristics, or experts' knowledge.
	
\end{itemize}

Let us, very briefly comment on the related literature. Most of these endeavors focus on control problems related to ordinary differential equations.  We quote the  two  surveys papers \cite{F03,WJ01} and \cite{M10}. From these publications already it becomes clear that the notion by which optimality is measured  is an important topic in its own right. The literature on optimal actuator positioning for distributed parameter systems is less rich but it also dates back for several decades already. From among the earlier contributions we quote \cite{EP88} where the topic is investigated in a semigroup setting for linear systems, \cite{B72} for a class of linear infinite dimensional filtering problems, and \cite{FD00} where the optimal actuator problem is investigated for hyperbolic problems related to active noise suppression. In the works \cite{M11,M13,M15} the optimal actuator problem is formulated in terms of parameter-dependent linear quadratic regulator problems where the parameters characterize the position of  actuators, with predetermined shape,  for example.  By choosing the actuator position in \cite{HH05} the authors optimise the decay rate in the one-dimensional wave equation. Our research may be most closely related to the recent contribution \cite{PRTRZU17}, where the optimal actuator design is driven by exact controllability considerations, leading to actuators which are chosen on the basis of minimal energy controls steering the system to zero within a specified time uniformly, for a bounded set of initial conditions. Finally, let us mention that the optimal actuator problem is in some sense dual to optimal sensor location problems \cite{R16}, which is of paramount importance.

\paragraph{Structure of the paper}
	The paper is organised as follows.

In Section~\ref{sec:three},   the optimal control problems, with respect to which optimal actuators are sought later, are introduced. While the first formulation depends on a single initial condition for the system dynamics,  in the second formulation the optimal actuator  mitigates the worst closed-loop performance among all the possible initial conditions.
	
	In Sections~\ref{sec:four} and \ref{sec:five} we derive the shape and topological sensitivities associated to the aforedescribed optimal actuator design problems.
	
	Section~\ref{sec:six} is devoted to describing a numerical approach which constructs the optimal actuator based on the shape and topological derivatives computed in Sections \ref{sec:four} and \ref{sec:five}. It involves the numerical realisation of the sensitivities and iterative gradient-based and level-set approaches.
	
	Finally in  Section~\ref{sec:seven} we report on  computations involving numerical tests for our model problem in dimensions one and two.

\subsection{Notation}
Let $\DD\subset \R^d$, $d=1,2,3$ be either a bounded domain with $C^{1,1}$ boundary $\partial \Omega$ or a convex domain, and let $T>0$ be a fixed time. The space-time cylinder is denoted by $\Omega_T := \Omega\times (0,T]$.    Further  by $H^1(\DD)$ denotes the  Sobolev space of square integrable functions on $\DD$ with square integrable  weak  derivative. The space $H^1_0(\DD)$ comprises all functions in $H^1(\DD)$ that have trace zero on $\partial \DD$ and $H^{-1}(\DD)$ stands for the dual of $H^1_0(\DD)$. The space $\ac C^{0,1}(\overbar\DD , \R^d)$ comprises all Lipschitz
continuous functions on $\overbar{\DD}$ vanishing on $\partial \DD $. It is
a closed subspace of $C^{0,1}(\overbar\DD ,\R^d)$, the space of Lipschitz continuous mappings
defined on $\overbar\DD $.
Similarly we denote by $\ac C^k(\overbar\DD ,\R^d)$ all $k$-times differentiable
functions on $\overbar\DD $ vanishing on $\partial \DD $. We use the notation $\partial f$ for the
Jacobian of a function $f$. Further
$B_\eps(x)$ stands for the open ball centered at $x\in \R^d$ with radius $\eps >0$. Its closure is denoted $\overbar{B}_\eps(x) := \overbar{B_\eps(x)}$.
By $\mathfrak Y(\DD)$ we denote the set of all measurable subsets $\omega\subset \DD$. We say that a sequence $(\omega_n)$ in $\mathfrak Y(\DD)$ converges to an element $\omega\in \mathfrak Y(\DD)$ if
$\chi_{\omega_n}\to \chi_\omega $ in $L_1(\DD)$ as $n\to \infty$, where $\chi_{\omega}$ denotes the characteristic function of $\omega$. In this case we write $\omega_n \to \omega$.  Notice that $\chi_{\omega_n}\to \chi_\omega $ in $L_1(\DD)$ as $n\to \infty$ if and only if $\chi_{\omega_n}\to \chi_\omega $ in $L_p(\DD)$ as $n\to \infty$ for all $p\in [1,\infty)$.
For two sets $A,B\subset \R^d$ we write $A\Subset B$ is $\overbar{A}$ is compact and $\overbar{A}\subset B$.


\section{Problem formulation and first properties}\label{sec:three}

\subsection{Problem formulation}
Our goal is to study an optimal actor positioning and design problem for a controlled linear parabolic equation. Let $\mathcal U$ be a  closed and convex subset of $L_2(\DD)$ with $0\in \mathcal U$. For each $\omega\in \mathfrak Y(\DD)$ the set   $\chi_\omega \mathcal U$ is a convex subset of $L_2(\DD)$. The elements of the space $\mathfrak Y(\DD)$ are referred to as \emph{actuators}. The choices $\mathcal U= L_2(\DD)$ and $\mathcal U = \R$, considered as  the space of constant functions on $\Omega$, will play a special role.
Further, $\UU := L_2(0,T;\mathcal U)$
denotes the space of time-dependent controls, which is equipped with the topology induced by the $L_2(0,T;L_2(\DD))-$norm. We denote  by  $K$ a nonempty, weakly closed subset of $H^1_0(\Omega)$. It will serve as the set of admissible initial conditions for the stable formulation of our optimal actuator positioning problem.

With these preliminaries we consider for every triplet $(\omega,u,\y)\in \mathfrak Y(\DD)\times \UU \times H^1_0(\Omega)$  the following linear parabolic equation: find $y:\overbar \Omega\times [0,T]\to \R$ satisfying
\begin{subequations}
	\label{eq:state_system}
	\begin{align}
		\partial_t y   - \Delta y =   \chi_\omega u  & \qquad \text{ in } \DD\times (0,T], \label{eq:statea} \\
		y=0 & \qquad\text{ on } \partial \DD\times (0,T], \label{eq:stateb}\\
		y(0)=f & \qquad\text{ on } \DD. \label{eq:statec}
	\end{align}
\end{subequations}

\noindent In the following, we discuss the well-posedness of the system dynamics \ref{eq:state_system} and the associated linear-quadratic optimal control problem, to finally state the optimal actuator design problem.

 \paragraph{Well-posedness of the linear parabolic problem}
It is a classical result \cite[p. 356, Theorem 3]{EV98}  that system \eqref{eq:state_system} admits a unique weak solution $y=y^{u,\y,\omega}$ in $W(0,T)$, where
\[
W(0,T) :=  \{ y\in L_2(0,T;H^1_0(\DD)):\; \partial_t y \in L_2(0,T;H^{-1}(\DD)) \},
\]
which satisfies by definition,
\ben\label{eq:weak_form}
\langle \partial_t y, \varphi\rangle_{H^{-1}, H^1_0}  + \int_{\DD}\nabla y \cdot \nabla \varphi \; dx =  \int_\DD  \chi_\omega u  \varphi \; dx
\een
for all  $\varphi\in H^1_0(\DD)$ for a.e.  $t\in (0,T]$, and $ y(0)=\y$. For the shape calculus of Section 4 we require that $f \in H_0^1 (\Omega)$.
In this case the state variable enjoys additional regularity properties. In fact, in \cite[p. 360, Theorem 5]{EV98} it is shown that for  $\y\in H^1_0(\DD)$ the weak solution $y^{\omega,u,\y}$ satisfies
\ben\label{eq:reg+y}
y^{u,f,\omega}\in L_2(0,T,H^2(\DD))\cap L_\infty(0,T;H^1_0(\DD)), \quad \partial_t y^{u,f,\omega}\in L_2(0,T;L_2(\DD))
\een
and there is a constant $c>0$, independent of $\omega, \y$ and $u$, such that
\ben\label{eq:estimate_apriori}
\|y^{u,f,\omega}\|_{L_\infty(H^1)} + \|y^{u,f,\omega}\|_{L_2(H^2)} + \|\partial_t y^{u,f,\omega}\|_{L_2(L_2)} \le  c(\|  \chi_\omega u \|_{L_2(L_2)} + \|f\|_{H^1}).
\een
Thanks to the lemma of Aubin-Lions the space
\ben\label{eq:embedding}
Z(0,T) := \{y\in L_2(0,T;H^2(\Omega)\cap H^1_0(\Omega)):\; \partial_t y \in L_2(0,T;L^2 (\Omega))  \}
\een
is compactly embedded into $L_\infty(0,T;H^1_0(\DD))$.


\paragraph{The linear-quadratic optimal control problem} After having discussed the well-posedness of the linear parabolic problem, we recall a standard linear-quadratic optimal control problem associated to a given actuator $\omega$. Let $\gamma >0$ be given. First we define for every triplet $(\omega,f,u)\in \mathfrak Y(\DD)\times  H^1_0(\Omega)\times \UU$ the cost functional
\ben\label{eq:cost_J}
J(\omega,u,\y) :=  \int_0^T\|y^{u,f,\omega}(t)\|_{L_2(\DD)}^2 + \gamma \| u (t)\|_{L_2(\DD)}^2\; dt.
\een
By taking the infimum in \eqref{eq:cost_J} over all controls $u\in \UU$ we obtain the function $\mathcal J_1$, which is defined for all $(\omega,f)\in \mathfrak Y(\DD)\times  H^1_0(\Omega)$:
\ben\label{eq:cost}
\mathcal J_1(\omega,\y) :=  \inf_{u\in \UU}  J(\omega,u,f).
\een
It is well known, see e.g. \cite{TR05} that the minimisation problem on the right hand side of \eqref{eq:cost}, constrained to the dynamics \eqref{eq:state_system}  admits a unique solution. As a result, the  function $\mathcal J_1(\omega,\y)$ is well-defined.  The minimiser  $\overbar u$ of \eqref{eq:cost} depends on the initial condition $f$ and the set $\omega$, i.e., $\overbar u = \overbar u^{\omega,\y}$.
In order to eliminate the dependence of the optimal actuator $\omega$ on the initial condition $f$ we define a robust  function $\mathcal J_2$ by  taking the supremum in \eqref{eq:cost} over all normalized initial conditions $\y$ in $K$:
\ben\label{eq:final_cost}
\mathcal J_2(\omega) := \sup_{\substack{\y\in K,\\\|\y\|_{H^1_0(\Omega)}\le 1}}\mathcal J_1(\omega,f).
\een
We show later on that the supremum on the right hand side of \eqref{eq:final_cost} is actually attained.
\paragraph{The optimal actuator design problem} We now have all the ingredients to state the optimal actuator design problem we shall study in the present work.
In the subsequent sections we are concerned with the following minimisation problem
\ben\label{eq:min_constrained_problem1}
\begin{split}
	\inf_{\substack{\omega\in  \mathfrak Y(\DD)\\ |\omega|  = c}} \mathcal J_1(\omega,f), \text{ for } f \in K,
\end{split}
\een
where $c\in (0,|\DD|)$ is the measure of the prescribed volume of the actuator $\omega$. That is, for a given initial condition $f$ and a given volume constraint $c$, we design the actuator $\omega$ according to the closed-loop performance of the resulting linear-quadratic control problem \eqref{eq:cost}. Note that no further constraint concerning the actuator topology is considered. Buidling upon this problem, we shall also study the problem
\ben\label{eq:min_constrained_problem}
\begin{split}
	\inf_{\substack{\omega\in  \mathfrak Y(\DD)\\ |\omega|  = c}} \mathcal J_2(\omega),
\end{split}
\een
where the dependence of the optimal actuator on the initial condition of the dynamics is removed by minimising among the set of all the normalised initial condition $f\in K$.

Finally, another problem of interest which can be studied within the present framework is the  \textsl{optimal actuator positioning} problem, where the topology of the actuator is fixed, and only its position is optimised. Given a fixed set $\omega_0\subset \DD$ we study the optimal actuator positioning problem by solving
\ben\label{eq:min_constrained_position_problem1}
\begin{split}
	\inf_{X\in \R^d} \mathcal J_1((\id +X)(\omega_0),f), \text{ for } f \in K,
\end{split}
\een
and
\ben\label{eq:min_constrained_position_problem}
\begin{split}
	\inf_{X\in \R^d} \mathcal J_2((\id +X)(\omega_0)),
\end{split}
\een
where $(\id +X)(\omega_0) = \{x+X:\; x\in \omega_0\}$, i.e., we restrict our optimisation procedure to a set of actuator translations.

Our goal is to characterize shape and topological derivatives for $\mathcal J_1(\omega,f)$ (for fixed $f$) and $\mathcal J_2(\omega)$ in order to develop gradient type algorithms to solve \eqref{eq:min_constrained_problem1} and \eqref{eq:min_constrained_problem}. The results presented in Sections \ref{sec:four} and \ref{sec:five} can  also be utilized to derive optimality conditions for problems \eqref{eq:min_constrained_position_problem1} and \eqref{eq:min_constrained_position_problem}. In addition, we investigate numerically whether the proposed methodology provides results which coincide with physical intuition.

 While the existence of optimal shapes according to  \eqref{eq:min_constrained_problem1} and \eqref{eq:min_constrained_problem} is certainly also an interesting task,  this issue is postponed to future work. We  mention \cite{PRTRZU17} where a problem similar  to ours but with different cost functional is considered.

\subsection{Optimality system for $\mathcal{J}_1$}\label{subsec_opt}
The unique solution $\bar u\in \UU$ of the minimisation problem on the right hand side of \eqref{eq:cost} can be characterised by the first order necessary optimality condition
\ben\label{eq:solution_opt_control}
\partial_u J(\omega,\bar u,\y)(v-\bar u) \ge 0\quad \text{ for all } v\in \UU.
\een
The function $\bar u\in \UU$ satisfies the variational inequality \eqref{eq:solution_opt_control} if and only if there is a multiplier $\overbar p\in W(0,T)$ such that  the triplet
$(\overbar u,\overbar y,\overbar p)\in \UU\times W(0,T)\times W(0,T)$  solves
\begin{subequations}
	\label{eq:optimality_system}
	\begin{align}
		\int_{\Omega_T} \partial_t \overbar y \varphi   + \nabla \overbar y \cdot \nabla \varphi \; dx\;dt & =  \int_{\Omega_T}  \chi_\omega \overbar u \varphi \; dx\;dt \quad \text{ for all } \varphi\in W(0,T),  \label{eq:opt_1}\\
		\int_{\Omega_T} \partial_t \psi  \overbar p + \nabla \psi \cdot \nabla \overbar p \; dx\;dt & =  -\int_{\Omega_T} 2 \overbar y \psi \; dx\;dt \quad \text{ for all } \psi\in W(0,T), \label{eq:opt_2} \\
		\int_{\Omega} (2\gamma \overbar u - \chi_\omega \bar p)& (v-\overbar u)\; dx \ge  0 \quad \text{ for all } v\in \mathcal U,\quad \text{ a.e. } t\in (0,T),  \label{eq:opt_3}
	\end{align}
\end{subequations}
supplemented with the initial and terminal conditions $\overbar y(0)=\y$ and $\overbar p(T)=0$ a.e. in $\DD$.
Two cases are of particular interest to us:

\begin{remark}\label{rem:special_control_space}
	\begin{itemize}
		\item[(a)] If $\mathcal U = L_2(\DD)$, then \eqref{eq:opt_3} is equivalent to $2\gamma \bar u =\chi_\omega \bar p$ a.e. on $\DD\times (0,T)$.
		\item[(b)] If $\mathcal U = \R$, then \eqref{eq:opt_3} is equivalent to
		$ 2\gamma |\DD|\bar u = \int_{\omega} \bar p \; dx$ a.e. on $(0,T)$.
	\end{itemize}
\end{remark}

\subsection{Well-posedness of $\mathcal{J}_2$}
Given $\omega\in \mathfrak Y(\DD)$ and $\y\in K$, we use the notation $\overbar{u}^{\y,\omega}$ to denote the unique minimiser of $J(\omega,\cdot,\y)$ over $\UU$.
\begin{lemma}\label{lem:continuity_f}
	Let $(\y_n)$ be a sequence in $K$ that converges weakly in $H^1_0(\Omega)$ to  $f\in K$, let $(\omega_n)$ be a sequence in $\mathfrak Y(\DD)$ that converges to $\omega\in \mathfrak Y(\DD)$, and let $(u_n)$ be a sequence in $\UU$ that converges weakly to a function $u\in \UU$. Then we have
	\ben\label{eq:convergence_shape_un_yn}
	\begin{split}
		y^{u_n,\y_n,\omega_n}  &\to y^{u,f,\omega} \quad \text{ in } L_2(0,T;H^1_0(\Omega)) \quad \text{ as } n\to \infty, \\
		y^{u_n,\y_n,\omega_n}  & \rightharpoonup y^{u,f,\omega} \quad \text{ in } L_2(0,T;H^2(\Omega)\cap H^1_0(\Omega))\quad \text{ as } n\to \infty.
	\end{split}
	\een
\end{lemma}
\begin{proof}
	The a-priori estimate \eqref{eq:estimate_apriori} and the compact embedding $Z(0,T)\subset \\ L_2(0,T;H^1_0(\DD))$ show that we can extract a subsequence of $(y^{u_n,\y_n,\omega_n})$ that converges weakly to an element $y$ in $L_2(0,T;H^2(\Omega)\cap H^1_0(\Omega))$ and strongly in $L_2(0,T;H^1_0(\Omega))$. Using this to pass to the limit in \eqref{eq:weak_form} with $(u,\y,\omega)$ replaced by $(u_n,\y_n,\omega_n)$ implies by uniqueness that $y = y^{u,f,w}.$
\end{proof}

\begin{lemma}\label{lem:continuity_u}
	Let $(\y_n)$ be a sequence in $H^1_0(\Omega)$ converging weakly to $f\in H^1_0(\Omega)$ and let $(\omega_n)$ be a sequence in $\mathfrak Y(\DD)$ that converges to $\omega\in \mathfrak Y(\DD)$. Then we have
	\ben
	\bar u^{\y_n,\omega_n}  \to \bar u^{f,\omega} \quad \text{ in } L_2(0,T;L_2(\DD)) \text{ as } n\to \infty.
	\een
\end{lemma}
\begin{proof}
	Using estimate \eqref{eq:estimate_apriori} we see that for all $u\in \UU$ and $n\ge 0$, we have
	\ben\label{eq:u_f_n}
	\begin{split}
		 &\int_0^T\|y^{ \bar u^{\y_n,\omega_n} ,\y_n,\omega_n}(t)\|_{L_2(\DD)}^2 + \gamma \| \bar u^{\y_n,\omega_n} (t)\|_{L_2(\DD)}^2\; dt \\
		 & \le \int_0^T\|y^{ u ,\y_n,\omega_n}(t)\|_{L_2(\DD)}^2 + \gamma \|  u (t)\|_{L_2(\DD)}^2\; dt \\
		 &\le c(\| \chi_{\omega_n} u \|_{L_2(L_2)}^2 + \|f_n\|_{H^1}^2).
	\end{split}
	\een
	It follows that $(\bar u_n):= (\bar u^{\y_n,\omega_n})$ is bounded in $L_2(0,T;L_2(\DD))$ and hence there is an element $\bar u\in L_2(0,T;L_2(\DD))$ and a subsequence $(\bar u_{n_k})$,  $\bar u_{n_k} \rightharpoonup \bar u$ in $L_2(0,T;L_2(\DD))$ as $k\to \infty$. In addition this  subsequence satisfies $\liminf_{k\to \infty} \|\bar u_{n_k}\|_{L_2(0,T;L_2(\DD))} \ge \|\bar u\|_{L_2(0,T;L_2(\DD))}$. Since $\mathcal U$ is closed we also have $\bar u \in L_2(0,T;\mathcal U)$. Together with Lemma~\ref{lem:continuity_f} we therefore obtain from \eqref{eq:u_f_n} by taking the $\liminf$ on both sides,
	\ben
	\int_0^T\|y^{ \bar u ,\y,\omega}(t)\|_{L_2(\DD)}^2 + \gamma \|  \bar u (t)\|_{L_2(\DD)}^2\; dt \le \int_0^T\|y^{ u ,\y,\omega}(t)\|_{L_2(\DD)}^2 + \gamma \|  u (t)\|_{L_2(\DD)}^2\; dt
	\een
	for all $u\in \UU$. This shows that $\bar u=\bar u^{\y,\omega}$ and since $\bar u^{\y,\omega}$ is the unique minimiser of $J(\omega,\cdot ,y)$ the whole sequence $(\bar u_n)$ converges weakly to $\bar u^{\y,\omega}$. In addition it follows from the strong convergence $y^{ \bar u^{\y_n,\omega_n} ,\y_n,\omega} \to y^{ \bar u^{\y,\omega} ,\y,\omega}$ in $W(0,T)$ and estimate \eqref{eq:u_f_n} that the norm $\|\bar u^{\y_n,\omega_n}\|_{L_2(0,T;L_2(\DD))}$ converges to $\|\bar u^{\y,\omega}\|_{L_2(0,T;L_2(\DD))}$. As norm convergence together with weak convergence imply strong convergence, this shows that $\bar u^{\y_n,\omega_n}$ converges strongly to $\bar u^{\y,\omega}$ in $L_2(0,T;L_2(\DD))$ as was to be shown.
\end{proof}

We now prove that $\omega\mapsto \mathcal J_2(\omega)$ is well-defined on $\mathfrak Y(\DD)$.
\begin{lemma}\label{lem:maximiser}	
For every $\omega\in \mathfrak Y(\DD)$ there exists $ \y \in K$ satisfying $\|\y\|_{H^1_0(\Omega)}\le 1$ and
\ben
\mathcal J_2(\omega) = \mathcal J_1(\omega, \y).
\een
\end{lemma}
\begin{proof}
Let $\omega\in \mathfrak Y(\DD)$ be fixed. In view of  $0\in \mathcal U$ and \eqref{eq:estimate_apriori} and since $K\subset H^1_0(\Omega)\hookrightarrow H^1_0(\DD)$ we obtain for all $\y \in H^1_0(\Omega)$ with $\|\y\|_{H^1_0(\Omega)}\le 1$,
\ben
\mathcal J_1(\omega,\y) = \min_{u\in \UU}J(\omega,u,\y) \le  \int_0^T\|y^{0,f,\omega}(t)\|_{L_2(\DD)}^2\;dt \le c \|\y\|_{H^1_0(\Omega)}^2 \le cr^2.
\een   	
Further we can express $\mathcal J_2$ as follows
\ben
\mathcal J_2(\omega) = \sup_{\substack{\y\in K\\\|\y\|_{H^1_0(\Omega)}\le 1}}\int_0^T\|y^{ \bar u^{\y,\omega} ,\y,\omega}(t)\|_{L_2(\DD)}^2 + \gamma \|  \bar u^{\y,\omega} (t)\|_{L_2(\DD)}^2\;dt.
\een
Let $(\y_n)\subset K$, $\|\y_n\|_{H^1_0(\Omega)} \le 1$ be a maximising sequence, that is,
\ben\label{eq:J2_min_sequence}
\mathcal J_2(\omega) = \lim_{n\to \infty}\int_0^T\|y^{ \bar u^{\omega,\y_n} ,\y_n,\omega}(t)\|_{L_2(\DD)}^2 + \gamma \| \bar u^{\omega,\y_n} (t)\|_{L_2(\DD)}^2\; dt.
\een
The sequence $(\y_n)$ is bounded in $K$ and therefore we find a subsequence $(\y_{n_k})$ converging weakly to an element $\y\in K$. Additionally,  the limit element satisfies  $\|\y\|_{H^1_0(\Omega)} \le \liminf_{k\to \infty}\|\y_{n_k}\|_{H^1_0(\Omega)}\le 1$  and hence $\|\y\|_{H^1_0(\Omega)}\le 1$. Since $(\y_{n_k})$ is also bounded in $H^1_0(\DD)$ we may assume that $(\y_{n_k})$ also converges weakly to  $\y\in H^1_0(\DD)$. Thanks to  Lemmas~\ref{lem:continuity_u} and \ref{lem:continuity_f} we obtain
\ben
\begin{split}
\mathcal J_2(\omega) &= \lim_{k\to \infty}\int_0^T\|y^{ \bar u^{\y_{n_k},\omega} ,\y_{n_k},\omega}(t)\|_{L_2(\DD)}^2 + \gamma \|  \bar u^{\y_{n_k},\omega} (t)\|_{L_2(\DD)}^2\; dt \\
& = \int_0^T\|y^{ \bar u^{\y,\omega} ,\y,\omega}(t)\|_{L_2(\DD)}^2 + \gamma \| \bar u^{\y,\omega}(t)\|_{L_2(\DD)}^2\; dt.
\end{split}
\een
\end{proof}

\begin{remark}
	In view of Lemma~\ref{lem:maximiser} we write from now on
	$ \mathcal J_2(\omega) =$\newline $\max_{\substack{\y\in K,\\\|\y\|_{H^1_0(\Omega)}\le 1}}\mathcal J_1(\omega,f).$
\end{remark}

\section{Shape derivative}\label{sec:four}
In this section we prove the directional differentiability of
$\mathcal J_2$ at arbitrary measurable sets. We employ the averaged adjoint approach \cite{MR3374631} which is tailored to the
derivation of directional derivatives of PDE constrained shape functions. Moreover this approach  allows us later on to also compute the topological derivative of  $\mathcal J_1$ and $\mathcal J_2$ without performing asymptotic analysis which can otherwise  be quite involved \cite{NOSO13}. 

Of course, there are notable alternative approaches, most prominent the material derivative approach,  to prove directional differentiability of shape functions, see e.g. \cite{MR2434064,MR948649}. For an overview of available methods the reader may consult \cite{MR3467385}.

\subsection{Shape derivative}
Given a vector field $X\in \ac C^{0,1}(\overbar\DD ,\R^d)$, we denote by $T_t^X$ the perturbation of the identity
$
T_t^X(x) := x + t X(x)
$
which is bi-Lipschitz for all $t\in [0,\tau_X]$, where
$
\tau_X := 1/(2\|X\|_{C^{0,1}}).
$
We omit the index $X$ and write $T_t$ insteand of $T_t^X$ whenever no confusion is possible. A mapping $J:\mathfrak Y(\DD) \rightarrow \R $  is called \emph{shape function}.

\begin{definition}\label{def1}
	The directional derivative of $J$ at $\omega\in \mathfrak Y(\DD)$ in direction $X\in \ac C^{0,1}(\overbar\DD , \R^d)$ is defined by
	\ben
	DJ(\omega)(X):= \lim_{t \searrow 0}\frac{J(T_t(\omega))-J(\omega)}{t} .
	\een
	We say that $J$ is
	\begin{itemize}
		\item[(i)] directionally differentiable at $
		\omega$ (in $\ac C^{0,1}(\overbar\DD , \R^d)$), if 	 $DJ(\omega)(X)$
		exists for all \\ $X\in C^{0,1}(\overbar\DD , \R^d)$,
		\item[(ii)]  \textit{differentiable} at $\omega$
		(in $\ac C^{0,1}(\overbar\DD ,\R^d)$), if  $DJ(\omega)(X)$ exists for all \\ $X \in \ac C^{0,1}(\overbar\DD ,\R^d)$ and
		$ X     \mapsto DJ(\omega)(X) $
		is linear and continuous.
	\end{itemize}
\end{definition}

The following properties will frequently be used.

\begin{lemma}\label{lem:phit}
	Let $\DD \subseteq \R^d$ be open and bounded and pick a vector field  $X\in \ac C^{0,1}(\overbar\DD, \R^d)$. (Note that $T_t(\Omega)=\Omega$ for all $t$.)
	\begin{itemize}
		\item[(i)] We  have as $t\rightarrow 0^+$,
		\begin{align*}
			\frac{ \partial T_t - I}{t} \rightarrow  \partial X \quad
			&\text{ and }\quad   \frac{
				\partial T_t^{-1} - I}{t} \rightarrow  - \partial X &&
			\text{ strongly in } L_\infty(\overbar\DD , \R^{d\times d})\\
			\frac{ \det(\partial T_t) - 1}{t} \rightarrow & \divv(X) && \text{ strongly in } L_\infty(\overbar\DD ).
		\end{align*}
		\item[(ii)]  	For all
		$\varphi \in L_2(\Omega)$, we have as $t\rightarrow 0^+$,
		\begin{align}\label{eq:Lp}
			\varphi\circ T_t  \rightarrow & \varphi
			&& \text{ strongly in }  L_2(\Omega).
		\end{align}
		\item[(iii)]

		Let $(\varphi_n)$ be a sequence in $H^1(\Omega)$ that converges weakly to $\varphi \in H^1(\Omega)$. Let $(t_n)$ a null-sequence. Then we have as $n \to \infty$,
		\begin{align}\label{eq:Wkp}
			\frac{\varphi_n\circ T_{t_n}  - \varphi_n}{t_n} \rightharpoonup & \nabla \varphi \cdot X
			&& \text{ weakly in }  L_2(\Omega).
		\end{align}

	\end{itemize}
\end{lemma}

\begin{proof}
	Item (i) is obvious. The convergence result \eqref{eq:Lp} is proved in \cite[Lem. 2.1, p.527]{DEZO11} and \eqref{eq:Wkp} can be proved in a similar fashion.

	Item (iii) is less obvious and we give a proof. For every $\eps >0$ and $\psi\in H^1(\Omega)$, there is $N>0$, such that $|(\varphi_n-\varphi,\psi)_{H^1}|\le \eps$ for all $n\ge N_\epsilon$.  By density we find for every $n$ and every null-sequence $(\epsilon_n)$,  $\eps_n >0$ an element $\tilde \varphi_n \in C^1(\overbar\Omega)$, such that
	\ben
	\|\tilde \varphi_n - \varphi_n\|_{H^1} \le \eps_n.
	\een
	It is clear that $\tilde \varphi_n \rightharpoonup \varphi$  weakly in $H^1(\Omega)$ as $n\to \infty$.
	We now write
	\ben\label{eq:varphi_n}
	\begin{split}
		\frac{\varphi_n \circ T_{t_n} - \varphi_n}{t_n} - \nabla \varphi_n \cdot X  = & \frac{(\varphi_n - \tilde \varphi_n) \circ T_{t_n} - (\varphi_n-\tilde \varphi_n)}{t_n} - \nabla (\varphi_n-\tilde \varphi_n) \cdot X\\
		& + \frac{\tilde\varphi_n \circ T_{t_n} - \tilde\varphi_n}{t_n} - \nabla \tilde\varphi_n \cdot X.
	\end{split}
	\een
	Let $x\in \Omega$. Applying the fundamental theorem of calculus to $s\mapsto \tilde \varphi_n(T_{s}(x))$ on $[0,1]$ gives
	\ben
	\frac{\tilde\varphi_n(T_{t_n}(x)) - \tilde\varphi_n(x)}{t_n} =  \int_0^1 \nabla \tilde \varphi_n(x+t_n s X(x)) \cdot X(x)\;ds.
	\een
	We now show that the function $q_n(x) := \int_0^1 \nabla \tilde \varphi_n(x+t_n s X(x)) \cdot X(x)$ converges weakly to $\nabla \varphi\cdot X$ in $L_2(\Omega)$. For this purpose we consider for $\psi\in L_2(\Omega)$,
	\ben
	\int_\Omega q_n \psi \; dx = \int_\Omega \int_0^1 \nabla \tilde \varphi_n(x+t_n s X(x)) \cdot X(x) \psi(x)\;ds \; dx.
	\een
	Interchanging the order of integration and invoking a change of variables (recall $T_t(\Omega)=\Omega$), we get
	\ben
	\int_\Omega q_n \psi \; dx  = \int_0^1 \underbrace{\int_\Omega \det(\partial T_{st_n}^{-1})  \nabla \tilde \varphi_n \cdot \left((X\psi)\circ T_{st_n}^{-1}\right) \; dx}_{:=\eta(t_n, s)} \;ds.
	\een
	Owing to item (ii) and noting that $X\circ T_t^{-1} \to X$ in $L_\infty(\Omega)$ as $t\to 0$, we also have for $s\in [0,1]$ fixed,
	\ben
	\det(\partial T_{st_n}^{-1})   (X\psi)\circ T_{st_n}^{-1}  \to X\psi \quad \text{ in } L_2(\Omega, \R^2) \quad\text{ as } n \to \infty.
	\een
	As a result using the weak convergence of $(\tilde \varphi_n)$ in $H^1(\Omega)$,  we get for $s\in[0,1]$,
	\ben
	\eta(t_n, s)   \to \int_\Omega \nabla \varphi \cdot X \psi\; dx\quad \text{ as } n\to \infty.
	\een
	It is also readily checked using H\"older's inequality that $|\eta(t_n, s)| \le c \|\nabla \tilde \varphi_n\|_{L_2}\|\psi\|_{L_2}$ for a constant $c>0$ independent of $s\in [0,1]$. As a result we may apply Lebegue's dominated convergence theorem to obtain
	\ben
\int_{\Omega} q_n \psi  \; dx = 	\int_0^1 \eta(t_n, s) \; ds \to \int_0^1 \eta(0, s)\; ds =  \int_\Omega \nabla \varphi \cdot X\; dx \quad \text{ as } n\to \infty.
	\een
	This proves that $q_n$ converges weakly to $\nabla \varphi\cdot X$.

	Finally testing \eqref{eq:varphi_n} with $\psi$, integrating over $\Omega$ and estimating gives
	\ben\label{eq:varphi_n2}
	\begin{split}
		\bigg| & \bigg(\frac{\varphi_n \circ T_{t_n}  - \varphi_n}{t_n} - \nabla \varphi_n \cdot X,\psi\bigg)_{L_2}\bigg|  \\
		& \le c\|\psi\|_{L_2} (\eps_n/t_n + \eps_n)
		+ \bigg|\bigg( \frac{\tilde\varphi_n \circ T_{t_n} - \tilde\varphi_n}{t_n} - \nabla \tilde\varphi_n \cdot X, \psi\bigg)_{L_2}\bigg|
	\end{split}
	\een
	with a constant $c>0$ only depending on $X$. Now we choose $\tilde N_\epsilon \ge 1$ so large that
	\ben
	\bigg|\bigg( \frac{\tilde\varphi_n \circ T_{t_n} - \tilde\varphi_n}{t_n} - \nabla \varphi \cdot X, \psi\bigg)_{L_2}\bigg| \le \epsilon\quad  \text{ for all } n\ge \tilde N_\epsilon.
	\een
	Then
	\ben\label{eq:final_estimate}
	\begin{split}
	  & \bigg|\bigg( \frac{\tilde\varphi_n \circ T_{t_n} - \tilde\varphi_n}{t_n} - \nabla \tilde\varphi_n \cdot X, \psi\bigg)_{L_2}\bigg| \\
      & \le \epsilon + |(\nabla(\tilde \varphi_n - \varphi_n)\cdot X,\psi)_{L_2}| + |(\nabla(\varphi_n - \varphi)\cdot X, \psi)_{L_2}|\\
      & \le \epsilon + \epsilon_n + \epsilon \quad \text{ for all } n\ge \max\{N_\epsilon, \tilde N_\epsilon \}.
	\end{split}
	\een
	Choosing $\epsilon_n := \min\{ t_n^2,\epsilon \}$ and combining the previous estimate with \eqref{eq:varphi_n2} shows the right hand side of \eqref{eq:final_estimate} can be bounded by $3\epsilon$. Since $\epsilon >0$ was arbitrary we see that \eqref{eq:Wkp} holds.
\end{proof}

\subsection{First main result: the directional derivative of $\mathcal J_2$}
Given $\omega\in \mathfrak Y(\DD)$ and $r>0$, we define the set of maximisers of  $\mathcal J_1(\omega,\cdot)$ by
\ben\label{eq:frak_X2_r}
\mathfrak X_2(\omega) := \{\bar f\in K:\; \sup_{\substack{f\in K, \\ \|f\|_{H^1_0(\Omega)}\le 1 }}\mathcal J_1(\omega,f) = \mathcal J_1(\omega,\bar f)\}.
\een
The set $\mathfrak X_2(\omega)$ is nonempty as shown in Lemma~\ref{lem:maximiser}. Before stating our first main result we make the following assumption.
\begin{assumption}\label{ass:cal_U}
	For every $X\in \ac C^{0,1}(\overbar \DD, \R^d)$ and $t\in [0,\tau_X]$ we have
	\ben
	u\in \mathcal U \quad \Longleftrightarrow \quad u\circ T_t \in \mathcal U.
	\een
\end{assumption}

\begin{remark}
	Assumption~\ref{ass:cal_U} is satisfied for $\mathcal U$ equal to  ${\tiny }L_2(\DD)$ or $\R$.
\end{remark}

Under the Assumption~\ref{ass:cal_U} we have the following theorem, where
we set $\bar y^{\y,\omega} := y^{\bar u^{\omega,\y},\y,\omega}$ and $\bar p^{\y,\omega} := p^{\bar u^{\omega,\y},\y,\omega}$ for $\omega\in \mathfrak Y(\DD)$ and $\y\in K$.
Furthermore we define for $A\in \R^{d\times d}, B\in \R^{d\times d}, a, b, c\in \R^{d}$
\begin{equation*}
A:B= \sum_{i,j=1}^d a_{ij}b_{ij}, \quad (a\otimes b)c := (b\cdot c) a,
\end{equation*}
where $a_{ij},b_{ij}$ are the entries of the matrices $A,B$, respectively.

\begin{theorem}\label{thm:shape_derivative_J2}
	\begin{itemize}
		\item[(a)]
		The directional derivative of $\mathcal J_2(\cdot)$ at $\omega$ in direction $X\in \ac C^{0,1}(\overbar \DD,\R^d)$ is given by
		\ben\label{eq:shape_main_J2}
		D\mathcal J_2(\omega)(X) = \max_{\y\in \mathfrak X_2(\omega)}\int_{\Omega_T} \Sb_1(\bar y^{\y,\omega} ,\bar p^{\y,\omega},\bar u^{\y,\omega}):\partial X + \Sb_0(\y) \cdot X \; dx\;dt,
		\een
		where the functions $\Sb_1(\y) := \Sb_1(\bar y^{\y,\omega} ,\bar p^{\y,\omega},\bar u^{\y,\omega})$ and $\Sb_0(\y)$ are given by
		\ben\label{eq:S0_S1}
		\begin{split}
			\Sb_1(\y)  =& I( |\bar y^{\y,\omega}|^2 + \gamma  |\bar u^{\y,\omega}|^2 -  \bar y^{\y,\omega} \partial_t \bar p^{\y,\omega}  + \nabla\bar y^{\y,\omega}\cdot \nabla \bar p^{\y,\omega}  - \chi_\omega \bar u^{\y,\omega} \bar p^{\y,\omega})\\
			&   - \nabla \bar y^{\y,\omega}\otimes \nabla \bar p^{\y,\omega} - \nabla \bar p^{\y,\omega}\otimes \nabla \bar y^{\y,\omega}, \\
			\Sb_0(\y)  =& - \frac1T\nabla \y \,\bar p^{\y,\omega}
		\end{split}
		\een
		and the adjoint $\bar p^{\y,\omega}$ satisfies
		\begin{alignat}{2}
			-\partial_t \bar p^{\y,\omega}  - \Delta \bar p^{\y,\omega} =  -2 \bar y^{\y,\omega} &  & &\quad \text{ in } \DD \times (0,T], \\
			\bar p^{\y,\omega}=0 &  & &\quad \text{ on }\partial \DD\times (0,T],\\
			\bar p^{\y,\omega}(T)=0 &  & &\quad \text{ in } \DD.
		\end{alignat}
		\item[(b)] The directional derivative of $\mathcal J_1(\cdot,\y)$ at $\omega$ in direction $X\in \ac C^{0,1}(\overbar \DD,\R^d)$ is given by
		\ben\label{eq:shape_main_J1}
		D\mathcal J_1(\omega,\y)(X) =\int_{\Omega_T} \Sb_1(\y):\partial X + \Sb_0(\y)\cdot X \; dx\;dt,
		\een
		where $\Sb_0(\y)$ and $\Sb_1(\y)$ are defined by  \eqref{eq:S0_S1}. 	
	\end{itemize}
\end{theorem}
\begin{proof}[Proof of item (b)]
	We notice that for $r>0$ we have
	\ben
	\max_{\substack{\y\in K,\\\|\y\|_{{H^1_0(\Omega)}}\le r}}\mathcal J_1(\omega,f) = 	r^2\max_{\substack{\y\in \frac1r K,\\\|\y\|_{{H^1_0(\Omega)}}\le 1}}\mathcal J_1(\omega,f).
	\een
	Therefore we may assume that $\bar \y\in K$ with $\|\bar \y\|_{H^1_0(\Omega)}\le 1$.
	Setting $K:= \{\bar \y \}$, we have for all  $\omega\in \mathfrak Y(\DD)$,
	\ben
	\mathcal J_2(\omega)  = \max_{\substack{\y\in K,\\\|\y\|_{{H^1_0(\Omega)}}\le 1}}\mathcal J_1(\omega,f) = \mathcal J_1(\omega,\bar f)
	\een
	and hence the result follows from item (a) since $\mathfrak X_2(\omega)=\{\bar f\}$ is a singleton. The proof of part (a) will be given in the following subsections.
\end{proof}
We pause here to comment on the regularity requirements imposed on $\y$. As can be seen from the volume expression \eqref{eq:shape_main_J2} we can extend $D\mathcal J_1(\omega, \y)$ to initial conditions $\y $ in $L_2(\DD)$. In fact, the only term that requires weakly differentiable initial conditions is the one involving $\Sb_0$ and it can be rewritten as follows for a.e. $t\in [0,T]$,
\ben
\begin{split}
\int_{\DD} \Sb_0(t)\cdot X\; dx & = - \frac1T \int_{\DD}
\nabla \y \cdot X \bar p^{\y, \omega}(t) \; dx \\
& = \frac1T \int_{\DD} \divv(X) \y \bar p^{\y, \omega}(t) + \y \nabla \bar p^{\y, \omega}(t) \cdot X\; dx,
\end{split}
\een
where we used that $\bar p^{\y, \omega}(t)=0$ on $\partial \DD$. This shows that the shape derivative $D\mathcal J_1(\omega, \y)$ can be extended to initial conditions $\y \in L_2(\DD)$. However, it is not possible to obtain the shape derivative for $\y \in L_2(\DD)$ in general. This will become clear in the proof of Theorem~\ref{thm:shape_derivative_J2}.

The next corollary shows that under certain smoothness assumptions on $\omega$ we can write the integrals \eqref{eq:shape_main_J2} and \eqref{eq:shape_main_J1} as integrals over $\partial \omega$.

\begin{corollary}\label{cor:shape_derivative_J1}
	Let $\y\in K$ and  $X\in \ac C^{0,1}(\overbar \DD,\R^d)$ be given. Assume that $\omega\Subset \Omega$ and $\Omega$ are  $C^2$ domains. Moreover, suppose that either $\mathcal U = L_2(\DD)$ or $\mathcal U = \R$.
	\begin{itemize}
		\item[(a)]  Given $\y \in \mathfrak X_2(\omega)$ define  $\hat{\Sb_1}(\y) := \int_{0}^{T} \Sb_1(\y)(s)\;ds$ and \\ $\hat {\Sb_0}(\y) := \int_{0}^{T} \Sb_0(\y)(s)\;ds$. Then we have
		\ben\label{eq:regularity_S0_S1}
		\begin{split}
			\hat {\Sb_1}(\y)|_{\omega}\in W^1_1(\omega,\R^{d\times d}), & \: \hat {\Sb_1}(\y)|_{\DD\setminus \overbar\omega}\in W^1_1(\DD\setminus \overbar\omega,\R^{d\times d}),
			\hat {\Sb_0}(\y)|_{\omega}\in L_2(\omega,\R^d),
		\end{split}
		\een
		and
		\ben\label{eq:cont_equation}
		-\divv(\hat{\Sb_1}(\y))+ \hat{\Sb_0}(\y) =0 \quad \text{ a.e. in } \omega \cup (\DD \setminus \overbar\omega).
		\een
		Moreover \eqref{eq:shape_main_J2} can be written as
		\ben\label{eq:boundary_formula}
		\begin{split}
		D\mathcal J_2(\omega)(X) & = \max_{\y\in \mathfrak X_2(\omega)}\int_{\partial \omega} [\hat{\Sb_1}(\y)\nu] \cdot X\; ds  \\  &= \max_{\y\in \mathfrak X_2(\omega)}-\int_{\partial \omega} \int_0^T\bar u^{\omega,\y} \bar p^{\omega,\y} (X\cdot \nu)\;dt\; ds
		\end{split}
		\een
		for  $X\in \ac C^1(\overbar \DD,\R^d)$, with $\nu$ the outer normal to $\omega.$ Here $[\hat{\Sb_1}(\y)\nu] :=\\ \hat{\Sb_1}(\y)|_{\omega}\nu - \hat{ \Sb_1}(\y)|_{\DD\setminus\omega}\nu$ denotes the jump of $\hat{\Sb_1}(\y)\nu$ across $\partial \omega$.
		\item[(b)] We have that  \eqref{eq:shape_main_J1} can be written as
		\ben\label{eq:shape_main_J1_bdy}
		D\mathcal J_1(\omega,\y)(X) = -\int_{\partial \omega} \int_0^T\bar u^{\omega,\y} \bar p^{\omega,\y} (X\cdot \nu)\;dt\; ds
		\een
		for $X\in \ac C^1(\overbar \DD, \R^d)$.	 	
	\end{itemize}
\end{corollary}

Before we prove this corollary we need the following auxiliary result.
\begin{lemma}\label{lem:auxillary}
	Suppose that $\DD$ is of class $C^2$. For all $\y \in {H^1_0(\Omega)}$ and $\omega\in \mathfrak Y(\DD)$, we have
	\ben\label{eq:inclusion_bdry}
	\int_0^T\bar y^{\y,\omega}(t) \partial_t \bar p^{\y,\omega}(t)\; dt \in W^1_1(\Omega), \quad \text{ and } \quad \int_0^T \nabla \bar p^{\y,\omega}(t)\cdot \nabla \bar y^{\y,\omega}(t) \; dt \in W^1_1(\Omega).
	\een
\end{lemma}
\begin{proof}
	From the general regularity results \cite[Satz 27.5, pp. 403 and Satz 27.3]{WLOKA82} we have that $\bar p^{\y,\omega}\in L_2(0,T;H^3(\DD))$ and $\partial_t \bar p^{\y,\omega}\in L_2(0,T;H^1(\DD))$, and $\bar y^{\y,\omega}\in L_2(0,T;H^2(\DD))$ and $\partial_t \bar y^{\y,\omega} \in L_2(0,T;L_2(\DD))$.
	
Observe that for almost all $t\in [0,T]$ we have $\partial_t \bar p^{\y,\omega}(t)\in H^1(\DD)$ and $\bar y^{\y,\omega}(t) \in H^2(\Omega) $. So
since $H^1(\DD)\subset L_6(\Omega)$ and $H^2(\DD)\subset C(\overbar\DD)$, where we use that $\Omega\subset \R^d$, $d\le 3$  we also have $\bar y^{\y,\omega}(t)\partial_t \bar p^{\y,\omega}(t)\in L_6(\DD)$ and a.e. $t \in (0,T)$
\ben\label{eq:estimate1}
	\|\bar y^{\y,\omega}(t)\partial_t \bar p^{\y,\omega}(t)\|_{L_1(\DD)} \le C\|\bar y^{\y,\omega}(t)\|_{H^2(\DD)} \|\partial_t \bar p^{\y,\omega}(t)\|_{H^1(\DD)}
\een
	for an constant $C>0$. Moreover by the product rule we have
\ben
	\partial_{x_j}(\bar y^{\y,\omega}(t)\partial_t \bar p^{\y,\omega}(t)) = \underbrace{\partial_{x_j}(\bar y^{\y,\omega}(t))}_{\in H^1(\DD)} \underbrace{\partial_t \bar p^{\y,\omega}(t)}_{\in H^1(\DD)} + \underbrace{\bar y^{\y,\omega}(t)}_{\in H^1(\DD)} \underbrace{(\partial_{x_j}\partial_t \bar p^{\y,\omega}(t))}_{\in L_2(\DD)},
\een
	so that $\partial_{x_j}(\bar y^{\y,\omega}(t)\partial_t \bar p^{\y,\omega}(t))\in L_1(\DD)$ and
\ben\label{eq:estimate2}
	\|\partial_{x_j}(\bar y^{\y,\omega}(t)\partial_t \bar p^{\y,\omega}(t))\|_{L_1(\DD)} \le C \|\bar y^{\y,\omega}(t)\|_{H^1(\DD)} \|\partial_t\bar p^{\y,\omega}(t)\|_{H^1(\DD)}
\een
for some constant $C>0$. So \eqref{eq:estimate1} and \eqref{eq:estimate2} imply that $t\mapsto \|\bar y^{\y,\omega}(t)\partial_t \bar p^{\y,\omega}(t)\|_{W^1_1(\DD)}$ belongs to $L_1(0,T)$.
This shows the left inclusion in \eqref{eq:inclusion_bdry}. As for the right hand side inclusion in \eqref{eq:inclusion_bdry} notice that for almost all $t\in [0,T]$ we have $\bar p^{\y,\omega}(t) \in H^3(\DD)$. Therefore $\nabla \bar p^{\y,\omega}(t)\in H^2(\DD)$ and $\nabla \bar y^{\y,\omega}(t)\in H^1(\DD)$ and thus $\nabla \bar y^{\y,\omega}(t)\cdot \nabla \bar p^{\y,\omega}(t)\in L_6(\DD)$. Similarly we check that $\partial_{x_j}(\nabla \bar y^{\y,\omega}(t)\cdot \nabla \bar p^{\y,\omega}(t))\in L_1(\DD)$ and thus $t\mapsto \|\nabla \bar y^{\y,\omega}(t)\cdot \nabla \bar p^{\y,\omega}(t)\|_{W^1_1(\DD)}\in L_1(0,T)$, which gives the right hand side inclusion in \eqref{eq:inclusion_bdry}.
\end{proof}

\begin{proof}[Proof of Corollary~\ref{cor:shape_derivative_J1}]
	We assume that Theorem~\ref{thm:shape_derivative_J2} holds. As a consequence of Lemma~\ref{lem:auxillary} we obtain \eqref{eq:regularity_S0_S1}.
	Then for all $X \in C^1_c( \DD,\R^d)$ satisfying $X|_{\partial \omega}=0$ we have $T_t(\omega) = (\id+tX)(\omega)=\omega$ for all $t\in [0,\tau_X]$. Hence $D\mathcal J_2(\omega)(X)=0$ for such vector fields which gives
	\ben\label{eq:inequality_cont}
	0=D\mathcal J_2(\omega)(X) \ge \int_\DD \hat \Sb_1(\y):\partial X + \hat\Sb_0(\y) \cdot X \; dx
	\een	
	for all $X \in C^1_c( \DD,\R^d)$ satisfying $X|_{\partial \omega}=0$
	and for all $\y\in \mathfrak X_2(\omega)$. Since for fixed $\y$ the expression in \eqref{eq:inequality_cont} is linear in $X$  this proves
	\ben\label{eq:cont_equation_weak}
	\int_\DD \hat \Sb_1(\y):\partial X + \hat\Sb_0(\y) \cdot X \; dx=0
	\een
	for all $X \in C^1_c( \DD,\R^d)$ satisfying $X|_{\partial \omega}=0$
	and for all $\y\in \mathfrak X_2(\omega)$.  Hence testing of \eqref{eq:cont_equation_weak} with vector fields $X\in C^1_c( \omega,\R^d)$ and $X\in C^1_c( \DD\setminus \overbar \omega,\R^d)$, partial integration  and
	\eqref{eq:regularity_S0_S1} yield the continuity equation \eqref{eq:cont_equation}. As a result, by partial integration (see e.g. \cite{LaruainSturm16}), we get for all $X \in C^1_c( \DD,\R^d)$,
	\ben\label{eq:partial_integration_shape}
	\begin{split}
	    D\mathcal J_2(\omega)(X) & = \max_{\y\in \mathfrak X_2(\omega)}\int_\DD \hat \Sb_1(\y):\partial X + \hat \Sb_0(\y) \cdot X \; dx \\
		& =\max_{\y\in \mathfrak X_2(\omega)}\bigg(\int_{\partial \omega} [\hat \Sb_1(\y)\nu]\cdot X \; ds  + \int_{\omega} \underbrace{(-\divv(\hat \Sb_1(\y) +  \hat\Sb_0(\y))}_{=0}\cdot X \; dx\\
		& \qquad\qquad\qquad + \int_{\DD\setminus \overbar \omega} \underbrace{(-\divv(\hat \Sb_1(\y) +  \hat\Sb_0(\y))}_{=0}\cdot X \; dx\bigg),
	\end{split}
	\een
	which proves the first equality in  \eqref{eq:boundary_formula}.  Now using Lemma~\ref{lem:auxillary} we see that  $\Tb(\y) := \hat{\Sb}_1(\y) + \int_0^T\chi_\omega \bar u^{\y,\omega}(t) \bar p^{\y,\omega}(t)\;dt$ belongs to $W^1_1(\DD,\R^{d\times d})$ and hence $[\Tb(\y)\nu]=0$ on $\partial \omega$. It follows that $[\hat \Sb_1(\y)\nu] = -\int_0^T\chi_\omega \bar u^{\y,\omega}(t) \bar p^{\y,\omega}(t)\;dt $ which finishes the proof of (a). Part (b) is a direct consequence of part (a).	
\end{proof}

The following observation is important for our gradient algorithm that we introduce later on.

\begin{corollary}\label{cor:shape_invariant}
	Let the hypotheses of Theorem~\ref{thm:shape_derivative_J2} be satisfied. Assume that if $v\in \mathcal U$ then $-v\in \mathcal U$. Then we have
	\ben
	D\mathcal J_1(\omega,-\y)(X) = D\mathcal J_1(\omega,\y)(X)
	\een
	for all $ X\in \ac C^{0,1}(\overbar \DD, \R^d)$ and $\y\in {H^1_0(\Omega)}$.
\end{corollary}
\begin{proof}
	Let $\y\in {H^1_0(\Omega)}$ be given. From the optimality system \eqref{eq:optimality_system} and the assumption that $v\in \mathcal U$ implies $-v\in \mathcal U$, we infer that $\overbar u^{-\y,\omega} = -\overbar u^{\y,\omega}$, $\bar y^{-\y,\omega} =- \bar y^{\y,\omega}$ and $\bar p^{-\y,\omega} =- \bar p^{\y,\omega}$. Therefore $\Sb_1(-\y)=\Sb_1(\y)$ and $\Sb_0(-\y) = \Sb_0(\y)$ and the result follows from \eqref{eq:shape_main_J1}.
\end{proof}

The following sections are devoted to the proof of Theorem~\ref{thm:shape_derivative_J2}(a) .
\subsection{Sensitivity analysis of the state equation}
In this paragraph we study the sensitivity of the solution $y$ of \eqref{eq:state_system} with respect to $(\omega,\y,u)$.
\paragraph{Perturbed state equation}
Let $X\in \ac C^{0,1}(\overbar \DD, \R^d)$ be a vector field and define $T_\tau:= \id+ \tau  X$.  Given $u\in  U$, $\y\in {H^1_0(\Omega)}$ and $\omega\in \mathfrak Y(\DD)$,  we consider \eqref{eq:state_system} with $\omega_\tau := T_\tau(\omega)$,
\begin{alignat}{2}\label{eq:state_per}
	\partial_t y^{u,\y,\omega_\tau}  - \Delta y^{u,\y,\omega_\tau} =  \chi_{\omega_\tau} u &  & &\quad \text{ in } \DD \times (0,T],  \\ \label{eq:state_per2}
	y^{u,\y,\omega_\tau}=0 &  & &\quad \text{ on }\partial \DD\times (0,T],\\ \label{eq:state_per3}
	y^{u,\y,\omega_\tau}(0)=\y &  & &\quad \text{ in } \DD.
\end{alignat}
We define the new variable
\ben\label{eq:equivalence}
y^{u,\y,\tau} := (y^{u\circ T^{-1},\y,\omega_\tau}) \circ T_\tau.
\een
Then since $\chi_{\omega_\tau} = \chi_\omega\circ T_\tau^{-1}$ and $\Delta f \circ T_t = \divv(A(t)\nabla(f\circ T_t))$, it follows from \eqref{eq:state_per}-\eqref{eq:state_per3} that
\begin{alignat}{2}\label{eq:state_per_weak1}
	\partial_t y^{u,\y,\tau}   - \frac{1}{\xi(\tau)}\divv(A(\tau)\nabla y^{u,\y,\tau}) =  \chi_{\omega}u &  & &\quad \text{ in } \DD\times (0,T], \\ \label{eq:state_per_weak2}
	y^{u,\y,\tau}=0 &  & &\quad \text{ on }\partial \DD\times (0,T], \\ \label{eq:state_per_weak3}
	y^{u,\y,\tau}(0)=\y\circ T_\tau &  & &\quad \text{ in } \DD,
\end{alignat}
where
\[
A(\tau) := \det(\partial T_\tau) \partial T_\tau^{-1} \partial T_\tau^{-\top}, \qquad \xi(\tau) := |\det(\partial T_\tau)|.
\]
Equations~\eqref{eq:state_per_weak1}-\eqref{eq:state_per_weak3} have to be understood in the variational sense, i.e., $y^{u,f,\tau}\in W(0,T)$ satisfying $y^{u,f,\tau}(0)= \y\circ T_\tau$ and
\begin{alignat}{2}\label{eq:state_per_weak2_}
	\int_{\DD_T} \xi(\tau)\partial_t y^{u,f,\tau}\varphi + A(\tau)\nabla y^{u,f,\tau}\cdot \nabla \varphi\;dx\,dt = \int_{\DD_T} \xi(\tau)\chi_{\omega} u\varphi \;dx\,dt &  & &
\end{alignat}
for all $\varphi \in W(0,T)$. Since $X\in \ac C^{0,1}(\overbar \DD, \R^d)$, we have for fixed $\tau$,
$$A(\tau,\cdot),\partial_\tau A(\tau,\cdot)\in L_\infty(\DD,\R^{d\times d}), \quad \xi(\tau,\cdot),\partial_\tau \xi(\tau,\cdot)\in L_\infty(\DD).$$
Moreover, there are constants $c_1,c_2>0$, such that
\ben
A(\tau,x)\zeta \cdot \zeta \ge c_1|\zeta|^2 \quad \text{ for all } \zeta\in \R^d,\; \text{ for a.e } x\in  \DD, \;  \text{ for all } \tau \in [0,\tau_X]
\een
and
\ben
\xi(\tau,x)\ge c_2 \quad \text{ for a.e } x\in  \DD, \;  \text{ for all } \tau  \in [0,\tau_X].
\een

\paragraph{Apriori estimates and continuity}

\begin{lemma}\label{lem:bounds_y}
	There is a constant $c>0$, such that for all $(u,\y,\omega) \in  \UU\times H^1_0(\DD)\times \mathfrak Y(\DD)$, and  $\tau \in [0,\tau_X]$, we have
\ben\label{eq:estimate_apriori2}
\begin{array}l	
\|y^{u,f,\omega_\tau}\|_{L_\infty(H^1)} + \|y^{u,f,\omega_\tau}\|_{L_2(H^2)} + \|\partial_t y^{u,f,\omega_\tau}\|_{L_2(L_2)} \\[1.5ex]
 \qquad \le  c(\| \chi_{\omega_\tau} u\|_{L_2(L_2)} + \|f\|_{H^1}),
\end{array}
\een
	and
	\ben\label{eq:estimate_apriori3}
	\|y^{u,f,\tau}\|_{L_\infty(H^1)}  + \|\partial_t y^{u,f,\tau}\|_{L_2(L_2)} \le  c(\|\chi_\omega u\|_{L_2(L_2)} + \|f\|_{H^1}).
	\een
\end{lemma}
\begin{proof}
	Estimate \eqref{eq:estimate_apriori2} is a direct consequence of \eqref{eq:estimate_apriori}.
	Let us prove \eqref{eq:estimate_apriori3}. Recalling   $y^{u,f,\tau} = y^{u\circ T^{-1}_\tau,f,\omega_\tau} \circ T_\tau$, a change of variables shows,
	\ben\label{eq:est1}
	\begin{split}
	    &\int_{\Omega_T} |y^{u,f,\tau}|^2+|\nabla y^{u,f,\tau}|^2\; \,dx\;dt  \\
		&=\int_{\Omega_T} \xi^{-1}(\tau)|y^{u\circ T^{-1}_\tau,f,\omega_\tau}|^2 + A^{-1}(\tau)\nabla y^{u\circ T^{-1}_\tau,f,\omega_\tau}\cdot \nabla y^{u\circ T^{-1}_\tau,f,\omega_\tau}\; dx\;dt  \\
		& \le c \int_{\Omega_T}|y^{u\circ T^{-1}_\tau,f,\omega_\tau}|^2+ |\nabla y^{u\circ T^{-1}_\tau,f,\omega_\tau}|^2\; dx\,dt  \\
		&  \stackrel{\eqref{eq:estimate_apriori2}}{\le}	c(\| \chi_{\omega_\tau} u\circ T^{-1}_\tau\|_{L_2(L_2)} + \|f\|_{H^1})  \\
		& \le C(\|\chi_\omega u\|_{L_2(L_2)}) + \|f\|_{H^1}),
	\end{split}
	\een
	and we further have
	\ben\label{eq:est2}
	\begin{split}
		\|\chi_{\omega_\tau} u\circ T^{-1}_\tau\|_{L_2(L_2)}^2 &= \|\sqrt{\xi}\chi_\omega u\|_{L_2(L_2)}^2 \le c\|\chi_\omega u\|^2_{L_2(L_2)}.
	\end{split}
	\een
	Combining \eqref{eq:est1} and \eqref{eq:est2} we obtain $\|y^{u,f,\tau}\|_{L_2(H^1)} \le c(\|\chi_\omega u\|_{L_2(L_2)} + \|f\|_{H^1})$.
	In a similar fashion we can show \eqref{eq:estimate_apriori3}.
\end{proof}

\begin{remark}
	An estimate for the second derivatives of $y^{u,f,\tau}$ of the form
	\ben
	\|y^{u,f,\tau}\|_{L_2(H^2)} \le c(\| u\|_{L_2(L_2)} + \|f\|_{H^1})
	\een
	may be achieved by invoking a change of variables in the term $\|y^{u,f}_\tau\|_{L_2(H^2)}$ in \eqref{eq:estimate_apriori2}.
	This, however, requires the vector field $X$ to be more regular, e.g.,  $\ac C^2(\overbar \DD, \R^d)$, and is not needed below.
\end{remark}

After proving apriori estimates we are ready to derive continuity results for the mapping $(u,f,\tau)\mapsto y^{u,f,\tau}$.

\begin{lemma}\label{lem:estimate_chi}
	For every $(\omega_1,u_1,f_1),(\omega_2,u_2,f_2)\in \mathfrak Y(\DD)\times \UU\times {H^1_0(\Omega)}$, we denote by  $y_1$ and $y_2$ the corresponding solution of \eqref{eq:state_per}-\eqref{eq:state_per3}. Then there is a constant $c>0$, independent of $(\omega_1,u_1,f_1),(\omega_2,u_2,f_2)$, such that
  \ben\label{eq:estimate_energy_char}
  \begin{split}	
	\|y_1-y_2\|_{L_\infty(H^1)} + \|y_1-y_2\|_{L_2(H^2)} + \|\partial_t y_1-\partial_t  y_2\|_{L_2(L_2)} \\
	\le  c(\|\chi_{\omega_1}u_1 - \chi_{\omega_2}u_2\|_{L_2(L_2)}+\|f_1-f_2\|_{H^1}) .
	\end{split}
  \een
\end{lemma}

\begin{proof}
	The difference $\tilde y := y_1 - y_1$ satisfies in a variational sense
	\begin{alignat}{2}
		\partial_t \tilde y   - \Delta   \tilde y = u_1\chi_{\omega_1} - u_2\chi_{\omega_2} &  & &\quad \text{ in } \DD \times (0,T], \\
		\tilde y=0 &  & &\quad \text{ on }\partial \DD\times (0,T],\\
		\tilde y(0)=f_1-f_2 &  & &\quad \text{ on } \DD.
	\end{alignat}
	Hence estimate \eqref{eq:estimate_energy_char} follows from \eqref{eq:estimate_apriori}.
\end{proof}

As an immediate consequence of Lemma~\ref{lem:estimate_chi} we obtain the following result.
\begin{lemma}\label{eq:state_pert}
	Let $\omega\in \mathfrak Y(\DD)$ be given. For all $\tau_n\in(0,\tau_X]$, $u_n,u\in \UU$ and $f_n,f\in H^1(\Omega_0)$ satisfying
	\ben
	u_n \rightharpoonup u \quad\text{ in } L_2(0,T;L_2(\DD)), \quad f_n \rightharpoonup f \quad \text{ in } {H^1_0(\Omega)}, \quad \tau_n\rightarrow 0, \quad \text{ as } n\rightarrow \infty,
	\een
	we have
	\ben\label{eq:convergence_yn}
	\begin{split}
		y^{u_n,f_n,\tau_n} \overset{*}{\rightharpoonup} & y^{u,f, \omega } \quad\text{ in } L_\infty(0,T;H^1_0(\DD)) \quad \text{ as } n\rightarrow \infty,\\
		y^{u_n,f_n,\tau_n} \rightharpoonup & y^{u,f, \omega } \quad\text{ in }  H^1(0,T;L_2(\DD))\quad \text{ as } n\rightarrow \infty.
	\end{split}
	\een
\end{lemma}
\begin{proof}
	Thanks to the apriori estimates of Lemma~\ref{lem:bounds_y} there exists $y\in \\ L_\infty(0,T;H^1_0(\DD)) \cap H^1(0,T;L_2(\DD))$ and a subsequence $(y^{u_{n_k},f_{n_k},\tau_{n_k}}) $ converging \\
weakly-star in $L_\infty(0,T;H^1_0(\DD))$ and weakly in  $H^1(0,T;L_2(\DD))$ to $y$.  Since $H^1(\DD)$ embeds compactly into $L^2(\Omega)$ we may assume, extracting another subsequence, that $f_{n_k}\rightarrow f$ in $L_2(\DD)$ as $k\to \infty$. By definition
$y_k := y^{u_{n_k},f_{n_k},\tau_{n_k}}$ satisfies for $k\ge 0$,
\begin{alignat}{2}\label{eq:state_per_weak}
		\int_{\DD_T} \xi(\tau_{n_k})\partial_t y_k\varphi + A(\tau_{n_k})\nabla y_k\cdot \nabla \varphi\;dx\,dt = \int_{\DD_T} \xi(\tau_{n_k}) \chi_{\omega} u_{n_k}\varphi \;dx\,dt &  & &,
	\end{alignat}
for all $\varphi \in W(0,T)$, and $y_k(0)=\y_{n_k}\circ T_{\tau_{n_k}}$ on $\DD$. Using the weak convergence of $u_{n_k},y_k$ stated before and the strong convergence obtained using Lemma~\ref{lem:phit},
\ben
\xi(\tau_n) \rightarrow 1 \quad \text{ in } L_\infty(\DD), \qquad  A(\tau_n) \rightarrow I \quad \text{ in } L_\infty(\DD,\R^{d\times d}),
\een
we may pass to the limit in \eqref{eq:state_per_weak} to obtain,
	\begin{alignat}{2}\label{eq:state_per_weak4}
		\int_{\DD_T} \partial_t y\varphi + \nabla y\cdot \nabla \varphi\;dx\,dt = \int_{\DD_T} \chi_{\omega} u\varphi \;dx\,dt &  & & \quad \text{ for all } \varphi \in W(0,T).
	\end{alignat}
Using Lemma~\ref{lem:phit} we see
	$f_{n_k}\circ T_{\tau_{n_k}}\rightarrow f$ in $L_2(\DD)$ as $k\rightarrow \infty$, and therefore $y(0)=f$.  Since the previous equation with $y(0)=\y$ admits a unique solution we conclude that  $y=y^{u,f,\omega}$. As a consequence of the uniqueness of the limit, the whole sequence $y^{u_n,f_n,\tau_n}$ converges to
	$y^{u,f,\omega}$. This finishes the proof.
\end{proof}

\subsection{Sensitivity of minimisers and maximisers}
Let us denote for $(\tau,\y)\in [0,\tau_X]\times K$ the  minimiser of  $u\mapsto J(\omega_\tau,u \circ T_\tau^{-1},\y)$,
by $\bar u^{\y_n,\tau_n}$.

\begin{lemma}\label{lem:convergence_un_shape}
	For every null-sequence $(\tau_n)$ in  $[0,\tau_X]$ and every sequence $(\y_n)$ in
	$K$ converging weakly (in ${H^1_0(\Omega)}$) to $\y\in K$, we have
	\ben
	\bar u^{\y_n,\tau_n} \to \bar u^{\y,\omega} \quad \text{ in } L_2(0,T;L_2(\DD)) \quad \text{ as } n\to \infty.
	\een
\end{lemma}
\begin{proof}
	We set $\omega_n := \omega_{\tau_n}$. By definition we have $\bar u^{\y_n,\tau_n} = \bar u^{\y_n,\omega_{\tau_n}}\circ T_{\tau_n}$. From Lemma~\ref{lem:continuity_u} we know that
	$ \bar u^{\y_n,\omega_{\tau_n}} $ converges to $\bar u^{\y_n,\omega}$ in $L_2(0,T;L_2(\DD))$. Therefore according to Lemma~\ref{lem:phit} also $\bar u^{\y_n,\omega_{\tau_n}} \circ T_{\tau_n}$ converges in $L_2(0,T;L_2(\DD))$ to $\bar u^{\y_n,\omega}$. This finishes the proof.
\end{proof}

\begin{lemma}\label{lem:convergence_un_shape2}
	For every  null-sequence $(\tau_n)$ in  $[0,\tau_X]$ and every sequence $(\y_n)$, $\y_n\in \mathfrak{X}_2(\omega_{\tau_n})$, there is a subsequence $(\y_{n_k})$ and $\y\in  \mathfrak{X}_2(\omega)$, such that $\y_{n_k} \rightharpoonup \y$ in ${H^1_0(\Omega)}$ as $k\to \infty$.
\end{lemma}
\begin{proof}
	We proceed similarly as in the proof of Lemma~\ref{lem:convergence_un_shape}. Let $\tau \in [0,\tau_X]$ and $v\in \UU$ be given. We obtain for all $\y \in K$,
\ben\label{eq:inf}
	J(\omega_\tau,u^{\y,\tau}\circ T^{-1}_\tau,\y)  = \inf_{u\in \UU} J(\omega_\tau,u\circ T^{-1}_\tau,\y) \le J(\omega_\tau, v\circ T^{-1}_\tau,\y).
\een
	Let $(\bar\y_n)$ be an arbitrary sequence with $ \bar\y_n \in \mathfrak{X}_2(\omega_{\tau_n})$. Since $\|\bar\y_n\|_{H^1_0(\Omega)}\le 1$ for all $n\ge 0$, there
	is a subsequence $(\bar\y_{n_k})$ and a function $\bar\y\in K$, such that $\bar\y_{n_k}\rightharpoonup \bar\y$ in ${H^1_0(\Omega)}$ as $k\to \infty$ and $\|\bar\y\|_{H^1_0(\Omega)}\le 1$. Thanks to Lemma~\ref{lem:convergence_un_shape} the sequence $(\bar u_k)$ defined by  $\bar u_k := \bar u^{\bar\y_{n_k},\tau_{n_k}}$ converges to $\bar u^{\bar\y,\omega}$ in $L_2(0,T;L_2(\DD))$. Moreover, Lemma~\ref{eq:state_pert} also shows that $y^{\bar u_k,\bar\y_{n_k},\tau_{n_k}} \to y^{\bar u^{\bar\y,\omega},\bar\y,\omega}$ in $L_2(0,T;L_2(\DD))$. By definition for all
	$k\ge 0$ and $ \y\in K$,
	\ben
	\begin{split}
	 & \int_{\DD_T} |y^{\bar u^{\y,\tau_{n_k}},\y,\tau_{n_k}}(t)|^2  + \gamma   | \bar u^{\y,\tau_{n_k}}(t)|^2 \;dx\,dt\\
    & \le  \sup_{\substack{f\in K \\ \|f\|_{H^1_0(\Omega)}\le 1 }}\int_{\DD_T} |y^{\bar u^{\y,\tau_{n_k}},\y,\tau_{n_k}}(t)|^2  + \gamma   | \bar u^{\y,\tau_{n_k}}(t)|^2 \;dx\,dt \\
	& = \int_{\DD_T} |y^{\bar u_k,\bar\y_{n_k},\tau_{n_k}}(t)|^2  + \gamma   |\bar u_k(t)|^2 \;dx\,dt
	\end{split}
	\een
	and therefore passing to the limit $k\to \infty$ yields, for all $\y\in K$,
	\ben
	\begin{split}
		\int_{\DD_T} |y^{\bar u^{\y,\omega},\y,\omega}(t)|^2  + \gamma   | \bar u^{\y,\omega}(t)|^2 \;dx\,dt & \le \int_{\DD_T} |y^{\bar u^{\bar\y,\omega},\bar\y,\omega}(t)|^2  + \gamma   | \bar u^{\bar\y,\omega}(t)|^2 \;dx\,dt.
	\end{split}
	\een
	This shows that $\y\in \mathfrak{X}_2(\omega)$ and finishes the proof.
\end{proof}

\subsection{Averaged adjoint equation and Lagrangian}
For fixed $\tau\in [0,\tau_X]$ the mapping  $\varphi \mapsto T_\tau^{-1}\circ \varphi$ is an isomorphism on  $\UU$,  therefore,
\ben\label{eq:repara}
\min_{u\in \UU} J(\omega_\tau,u,\y)  =  \min_{u\in \UU} J(\omega_\tau,u\circ T_\tau^{-1},\y).
\een
Hence a change of variables shows,
\ben\label{eq:j_1_tau}
\begin{split}
	\inf_{u\in \UU}J(\omega_\tau,u,\y) & = \inf_{u\in \UU}\int_0^T \|y^{u,\y,\omega_\tau}(t)\|_{L_2(\DD)}^2  + \gamma\| u(t)\|_{L_2(\DD)}^2\; dt \\
	& \stackrel{\eqref{eq:repara}}{=} \inf_{u\in \UU}\int_{\Omega_T} \xi(\tau)\left(|y^{u, \y, \tau}(t)|^2  + \gamma   |u(t)|^2 \right)\;dx \; dt.
\end{split}
\een
Introduce for every quadruple $(u,\y,y,p)\in \UU \times K\times W(0,T)\times W(0,T)$ and for every $\tau\in [0,\tau_X]$ the parametrised Lagrangian
\ben
\begin{split}
    \tilde{G}(\tau,u,\y,y,p)  := &  \int_{\DD_T} \xi(\tau)\left(|y|^2  + \gamma  | u|^2  \right)\;dx dt\\
	& + \int_{\DD_T} \xi(\tau)\, \partial_t y \,  p\; dx\,dt  + A(\tau)\nabla y \cdot \nabla p \; dx\,dt\\
	& -  \int_{\DD_T}\xi(\tau)u\chi_\omega p \;dx\,dt + \int_{\DD}\xi(\tau)(y(0)-\y\circ T_\tau)p(0)\;dx.
\end{split}
\een

\begin{definition}
	Given $(u,\y)\in \UU\times  K$, and $\tau\in [0,\tau_X]$, the averaged adjoint state $p^{u,f,\tau}\in W(0,T)$ is the solution of averaged adjoint equation
	\ben\label{eq:averaged}
    \int_0^1\partial_y {\tilde G}(\tau,u, \y ,sy^{u,f,\tau}+(1-s)y^{u,f,\omega},p^{u,f,\tau})(\varphi)\; ds =0\quad \text{ for all 	} \varphi \in W(0,T).
	\een
\end{definition}
\begin{remark}
	The averaged adjoint state $p^{u,f,\tau}$ in our special case only depends on $u$ and $\y$ through
	the state $y^{u,f,\tau}$.
\end{remark}

\noindent It is evident that \eqref{eq:averaged} is equivalent to
\ben
\begin{split}
	\int_{\DD_T} \xi(\tau)\partial_t \varphi    p^{u, f, \tau} + A(\tau)\nabla \varphi \cdot \nabla  p^{u, f,\tau}\; dx\,dt  &+ \int_\DD \xi(\tau) p^{u, f,\tau}(0)\varphi(0)\;dx \\
	& =  -\int_{\DD_T}   \xi(\tau)(y^{u, f,\tau} + y^{u, f,\omega}) \varphi \; dx\,dt
\end{split}
\een
for all $\varphi\in W(0,T)$, or equivalently after partial integration in time
\ben\label{eq:averaged_adjoint}
\begin{split}
	\int_{\DD_T}-\xi(\tau) \varphi    \partial_t p^{u,f,\tau}  +  A(\tau)\nabla \varphi \cdot \nabla  p^{u,f,\tau} \; dx\;dt  & =  -\int_{\DD_T}  \xi(\tau)(y^{u,f,\tau} + y^{u,f,\omega}) \varphi \; dx\;dt
\end{split}
\een
for all $\varphi\in W(0,T)$, and $p^{u,f,\tau}(T) = 0$.
This is a backward in time linear parabolic equation with terminal condition zero.

\subsection{Differentiability of max-min functions}

Before we can pass to the proof of Theorem \ref{thm:shape_derivative_J2} we need to address a Danskin type theorem on the differentiability of max-min functions.

Let $\mathfrak U$ and $\mathfrak V$ be two nonempty sets and let $ G:[0,\tau]\times \mathfrak U\times \mathfrak V \to \R$ be a function, $\tau >0$. Introduce the function $g:[0,\tau] \to \R$,
\ben
g(t) := \sup_{y\in \mathfrak V} \inf_{x\in \mathfrak U}  G(t,x,y)
\een
and let $\ell:[0,\tau]\rightarrow \R$ be any function such that $\ell(t)>0$ for $t\in (0,\tau]$ and $\ell(0)=0$. We are interested in sufficient conditions that guarantee that the limit
\ben
\frac{d}{d\ell} g(0^+) := \lim_{t\searrow 0^+} \frac{g(t)-g(0)}{\ell(0)}
\een
exists. Moreover we define for $t\in [0,\tau]$,
\ben
\mathfrak V(t) := \{y^t\in \mathfrak V: \sup_{y\in \mathfrak V}\inf_{x\in \mathfrak U}  G(t,x,y) = \inf_{x\in \mathfrak U}  G(t,x,y^t) \}.
\een

\begin{lemma}\label{lem:maxmin}
	Let the following hypotheses be satisfied.
	\begin{itemize}
		\item[(A0)] For all $y\in \mathfrak V$ and $t\in [0,\tau]$ the minimisation problem
		\ben
		\inf_{x\in \mathfrak U}  G(t,x,y)
		\een
		admits a unique solution and we denote this solution by $x^{t,y}$.
		\item[(A1)] For all $t$ in $[0,\tau]$  the set $\mathfrak V(t)$  is nonempty.
		\item[(A2)]
			The limits
			\ben\label{eq:max_min_a1}
			\lim_{t\searrow 0} \frac{G(t,x^{t,y},y)-G(0,x^{t,y},y)}{\ell(t)}
			\een
			and
			\ben\label{eq:max_min_a2}
			\lim_{t\searrow 0} \frac{G(t,x^{0,y},y)-G(0,x^{0,y},y)}{\ell(t)}
			\een
			exist for all $y\in \mathfrak U$ and they are equal. We denote the limit by \\
			 $\partial_\ell G(0^+,x^{0,y},y)$.

		\item[(A3)] For all real null-sequences $(t_n)$ in $(0,\tau]$ and all sequences $y^{t_n}$ in $\mathfrak V(t_n)$, there exists a subsequence $(t_{n_k})$ of $(t_n)$,  and $(y^{t_{n_k}})$ of $(y^{t_n})$, and $y^0$ in $\mathfrak V(0)$, such that
		\ben\label{eq:max_min_b1}
		\lim_{k\rightarrow  \infty} \frac{G( t_{n_k} ,x^{t_{n_k},y^{t_{n_k}}}, y^{t_{n_k}} ) -
			G(0,x^{t_{n_k},y^{t_{n_k}}}, y^{t_{n_k}})}{\ell(t_{n_k})} = \partial_\ell G(0^+,x^{0,y^0},y^0)
		\een	
		and
		\ben\label{eq:max_min_b2}
		\lim_{k\rightarrow  \infty} \frac{G( t_{n_k} ,x^{0,y^{t_{n_k}}}, y^{t_{n_k}} ) -
			G(0,x^{0,y^{t_{n_k}}}, y^{t_{n_k}})}{\ell(t_{n_k})} = \partial_\ell G(0^+,x^{0,y^0},y^0).
		\een
	\end{itemize}
	Then we have
	\ben
	\frac{d}{d\ell}g(t)|_{t=0^+} = \max_{y\in  \mathfrak V(0) }\partial_\ell G(0^+,x^{0,y},y).
	\een
\end{lemma}
In this section we  apply the previous results  for $\ell(t)=t,$ and in the following one  for $\ell(t) = |B_t(\eta_0)|$, $\eta_0\in \R^d$. For the sake of completeness we give a proof in the appendix; see \cite{DEMA75}.

\subsection{Proof of Theorem~\ref{thm:shape_derivative_J2}}

The following is a direct consequence of \eqref{eq:averaged_adjoint} and Lemma~\ref{eq:state_pert}.
\begin{lemma}\label{lem:averaged_shape}
	For all sequences $\tau_n\in(0,\tau_X]$, $u_n,u\in \UU$ and $f_n,f\in K$, such that
	\ben
	u_n \rightharpoonup u \quad\text{ in } \UU, \quad f_n \rightharpoonup f \quad \text{ in } H^1_0(\Omega), \quad \tau_n\to 0, \quad \text{ as } n\to \infty,
	\een
	we have
	\ben
	\begin{split}
		p^{u_n,f_n,\tau_n} \to & p^{u,f, \omega } \quad\text{ in } L_2(0,T;H^1_0(\DD)) \quad \text{ as } n\rightarrow \infty,\\
		p^{u_n,f_n,\tau_n} \rightharpoonup & p^{u,f, \omega } \quad\text{ in }  H^1(0,T;L_2(\DD))\quad \text{ as } n\rightarrow \infty,
	\end{split}
	\een
	where $p^{u,f,\omega}\in Z(0,T)$ solves the adjoint equation
	\ben\label{eq:adjoint_weak}
	\begin{split}
		\int_{\DD_T} - \varphi\partial_t p^{u,f,\omega}\; dx\,dt  + \int_{\DD_T}\nabla \varphi \cdot \nabla p^{u,f,\omega} \; dx\,dt = -\int_{\DD_T}2 y^{u,f,\omega}\varphi \; dx\,dt
	\end{split}
	\een
	for all $\varphi \in W(0,T)$, and $p^{u,f,\omega}(T) =0$ a.e. on $\DD$.
\end{lemma}

Now we have gathered all the ingredients to complete the proof of Theorem~\ref{thm:shape_derivative_J2}(a) on page 9.

\noindent
{\bf{Proof of Theorem~\ref{thm:shape_derivative_J2}(a)}}
Using the fundamental theorem of calculus we obtain for all $\tau \in [0,\tau_X]$,
\ben\label{eq:average_consequence}
\begin{split}
    \tilde{G}(\tau,u,\y, & y^{u, f ,\tau},p^{u,f,\tau}) -  \tilde{G}(\tau,u,\y,y^{u, f ,\omega},p^{u,f,\tau}) \\
    & = \int_0^1 \partial_y \tilde{G}(\tau,u,\y,s y^{u, f ,\tau} + (1-s)y^{u, f ,\omega},p^{u,f,\tau})(y^{u, f ,\tau}-y^{u, f ,\omega})\; ds =0,
\end{split}
\een
where in the last step we used the averaged adjoint equation \eqref{eq:averaged_adjoint}. In addition we have
$J(\omega_\tau,u\circ T_\tau^{-1},\y) = \tilde{G}(\tau,u,\y,y^{u, f ,\omega},p^{u,f,\tau}),$ which together with \eqref{eq:average_consequence} gives
\ben
J(\omega_\tau,u\circ T_\tau^{-1},\y) = \tilde{G}(\tau,u,\y,y^{u, f ,\omega},p^{u,f,\tau}).
\een
As a consequence we obtain
\ben\label{eq:J}
\mathcal J_1(\omega_\tau, \y) = \inf_{u\in \UU} \tilde{G}(\tau,u,\y,y^{u, f ,\omega},p^{u,f,\tau}).
\een

We apply Lemma~\ref{lem:maxmin} with $\ell(t) := t$,
\ben
G(\tau,u,f) := \tilde{G}(\tau,u,\y,y^{u, f ,\omega},p^{u,f,\tau}),
\een
$\mathfrak U = \UU$, and $\mathfrak V = \{\y\in K:\; \|\y \|_{H^1_0(\Omega)} \le 1 \}$.

Since the minimization problem  \eqref{eq:j_1_tau}  admits a unique solution,  Assumption (A0) is  satisfied.
A  minor change in the proof of  Lemma \ref{lem:maximiser} to accommodate the reparametrisation of the domain $\omega$ shows that  (A1) is  satisfied as well.

Let $(\tau_n)$ be an arbitrary null-sequence and let $(\y_n)$ be a sequence in $K$ converging weakly in ${H^1_0(\Omega)}$ to $f\in K$,
and let us set $\bar u_n := \bar u^{\y_n,\tau_n}$. Thanks to Lemma~\ref{lem:convergence_un_shape} we have that $\bar u_n$ converges strongly in $L_2(0,T;L_2(\DD))$ to $\bar u^{\y,\omega}$.
 Moreover Lemma~\ref{lem:averaged_shape} implies
\ben\label{eq:convergence_averaged_adjoint}
\begin{split}
	p^{\bar u_n,\y_n,\tau_n} \to & p^{\bar u^{\y,\omega}, \y, \omega }  \quad\text{ in } L_2(0,T;H^1_0(\DD)) \quad \text{ as } n\rightarrow \infty,\\
	p^{\bar u_n,\y_n,\tau_n} \rightharpoonup & p^{\bar u^{\y,\omega}, \y, \omega } \quad\text{ in }  H^1(0,T;L_2(\DD))\quad \text{ as } n\rightarrow \infty.
\end{split}
\een
Using Lemma~\ref{lem:auxillary} we see that
\ben
\frac{ A(\tau_n) - I}{\tau_n} \to \divv(X) - \partial X - \partial X^\top \quad \text{ in } L_\infty(\DD,\R^{d\times d}) \quad \text{ as } n\to \infty,
\een
and
\ben
\frac{\xi(\tau_n)-1}{\tau_n} \to \divv(X) \quad \text{ in } L_\infty(\DD)\quad \text{ as } n\to \infty.
\een
Therefore we get
\ben
\begin{split}
	&\frac{G(  \tau_n, \bar u_n,\y_n)-G(0,\bar u_n,\y_n)}{\tau_n} \\
    =&  \frac{ \tilde{G}(\tau_n,\bar u_n,\y_n,y^{\bar u_n, \y_n ,\omega},p^{\bar u_n,\y_n,\tau_n}) - \tilde{G}(0,\bar u_n,\y_n,y^{\bar u_n, \y_n ,\omega},p^{\bar u_n,\y_n,\tau_n})}{\tau_n}\\
	=&\int_{\DD_T} \frac{\xi(\tau_n)-1}\tau \left(|y^{\bar u_n, \y_n ,\omega}|^2  + \gamma  |\bar u_n|^2  \right)\;dx dt\\
	& + \int_{\DD_T} \frac{\xi(\tau_n)-1}\tau\, \partial_t y^{\bar u_n, \y_n ,\omega} \,  p^{\bar u_n,\y_n,\tau_n}\; dx\,dt \\
	& + \int_{\DD_T} \frac{A(\tau_n)-I}{\tau_n}\nabla y^{\bar u_n, \y_n ,\omega} \cdot \nabla p^{\bar u_n,\y_n,\tau_n} \; dx\,dt\\
	& -  \int_{\DD_T}\frac{\xi(\tau_n)-1}\tau \bar u_n \chi_\omega p^{\bar u_n,\y_n,\tau_n} \;dx\,dt\\
	&  + \int_{\DD} \big(\frac{\xi(\tau_n)-1}{\tau_n}  (y^{\bar u_n, \y_n ,\omega}(0) - \y_n\circ T_{\tau_n}) - \frac{ \y_n\circ T_{\tau_n} -\y_n }{\tau_n} \big)p^{\bar u_n,\y_n,\tau_n}(0)\;dx
\end{split}
\een
and using Lemma~\ref{lem:phit} and \eqref{eq:convergence_averaged_adjoint}, we see that the right hand side tends to
\ben\label{eq:derivative_G_prelim}
\begin{split}
	& \int_{\Omega_T} \divv(X) ( |\bar y^{\y,\omega}|^2 + \gamma |\bar u^{\y, \omega}|^2 + \partial_t \bar y^{\y,\omega} \bar p^{\y,\omega}  + \nabla \bar y^{\y,\omega} \cdot \nabla   \bar p^{\y,\omega} -   \bar u^{\y,\omega} \bar p^{\y,\omega} \chi_\omega )\; dx\; dt\\
	& -\int_{\Omega_T} \partial X\nabla \bar y^{\y,\omega}\cdot \nabla \bar p^{\y,\omega} + \partial X\nabla \bar p^{\y,\omega}\cdot \nabla \bar y^{\y,\omega} + \frac1T \nabla \y \cdot X \bar p^{\y,\omega}(0)\; dx\; dt.
\end{split}
\een
Partial integration in time yields
\ben\label{eq:derivative_G_prelim2}
\int_{\DD_T}  \bar p^{\y,\omega} \partial_t \bar y^{\y,\omega} \divv(X)\; dx \; dt = - \int_{\DD_T}  \partial_t \bar p^{\y,\omega}  \bar y^{\y,\omega} \divv(X)\; dx \; dt - \int_{\DD} \divv(X) \y \bar p^{\y,\omega}(0)\; dx,
\een
where we used $\bar y^{\y,\omega}(0) = \y$ and $\bar p^{\y,\omega}(T)=0$. As a result, inserting \eqref{eq:derivative_G_prelim2} into \eqref{eq:derivative_G_prelim}, we see that \eqref{eq:derivative_G_prelim} can be written as
\ben\label{eq:limit_tensor_proof}
\int_{\Omega_T} \Sb_1(\bar y^{\y,\omega} ,\bar p^{\y,\omega},u^{\y,\omega}):\partial X + \Sb_0\cdot X \; dx\;dt
\een
with $\Sb_1, \Sb_2$ being given by \eqref{eq:S0_S1}. Hence we obtain

\ben\label{eq:kk10}
\lim_{n\to \infty} \frac{G(  \tau_n, \bar u_n,\y_n)-G(0,\bar u_n,\y_n)}{\tau_n} = \int_{\Omega_T} \Sb_1(\bar y^{\y,\omega} ,\bar p^{\y,\omega},u^{\y,\omega}):\partial X + \Sb_0\cdot X \; dx\;dt.
\een

Next let  $\bar u_{n,0} := \bar u^{\y_n,0}$. Then we can show in as similar manner as \eqref{eq:kk10} that
\ben\label{eq:kk11}
\lim_{n\to \infty} \frac{G(  \tau_n, \bar u_{n,0},\y_n)-G(0,\bar u_{n,0},\y_n)}{\tau_n} = \int_{\Omega_T} \Sb_1(\bar y^{\y,\omega} ,\bar p^{\y,\omega},u^{\y,\omega}):\partial X + \Sb_0\cdot X \; dx\;dt.
\een
 Hence choosing $(\y_n)$ to be a constant sequence we see that (A2) is satisfied.

 But also (A3) is satisfied  since according to Lemma~\ref{lem:convergence_un_shape2}  we find for every null-sequence $(\tau_n)$ in  $[0,\tau_X]$ and every sequence $(\y_n)$, $\y_n\in \mathfrak{X}_2(\omega_{\tau_n})$, a subsequence $(\y_{n_k})$ and $\y\in  \mathfrak{X}_2(\omega)$, such that $\y_{n_k} \rightharpoonup \y$ in ${H^1_0(\Omega)}$ as $k\to \infty$.
 Now we use \eqref{eq:kk10}  and \eqref{eq:kk11} with $f_n$ replaced by this choice of $f_{n_k}$, and conclude that (A3) holds.
  Thus all requirements of Lemma~\ref{lem:maxmin} are satisfied and this ends the proof of Theorem~\ref{thm:shape_derivative_J2}(a).

\section{Topological derivative}\label{sec:five}
In this section we will derive the topological derivative of the shape functions $\mathcal J_1$ and $\mathcal J_2$ introduced in \eqref{eq:cost} and \eqref{eq:final_cost}, respectively. The topological derivative, introduced in \cite{SOZO99}, allows to predict the position where small holes in the shape should be inserted in order to achieve a  decrease of  the shape function.

\subsection{Definition of topological derivative}
We begin by introducing the so-called topological derivative. For more details we refer to \cite{NOSO13}.
\begin{definition}[Topological derivative]
	The topological derivative of a shape funcional $J:\mathfrak Y(\DD) \to \R $ at $\omega\in \mathfrak Y(\DD)$ in the point $\eta_0\in \DD\setminus \partial \omega$ is defined by
	\ben
	\mathcal TJ(\omega)(\eta_0) =  \left\{\begin{array}{ll}
		\lim_{\eps\searrow 0}\frac{J(\omega\setminus \bar B_\eps(\eta_0)) - J(\omega)}{|\bar B_\eps(\eta_0)|} & \text{ if } \eta_0 \in \omega, \\
		\lim_{\eps\searrow 0}\frac{J(\omega\cup  B_\eps(\eta_0)) - J(\omega)}{|B_\eps(\eta_0)|} & \text{ if } \eta_0 \in \DD \setminus \overbar \omega
	\end{array}\right..
	\een
\end{definition}

\subsection{Second main result: topological derivative of $\mathcal J_2$}
 Given $\omega\in \mathfrak Y(\DD)$ we set 
\begin{equation}
    \omega_\eps :=\begin{cases}
        \omega\setminus \bar B_\eps (\eta_0)& \text{ for } \eta_0\in\omega,\\
        \omega\cup \bar B_\eps(\eta_0) & \text{ for } \eta\in \DD\setminus\overline\omega.
    \end{cases}
\end{equation}
Moreover, we use the notation $\text{sgn}_\omega(x) = -1$ for $x\in \omega$ and $\text{sgn}_\omega(x)=1$ for $x\in \DD\setminus\overline\omega$. Further denote  by $\bar u^{\y,\omega_\eps}$ the minimiser of the right hand side of \eqref{eq:cost} with $\omega=\omega_\eps$.

For $\omega\in \mathfrak Y(\DD)$ and $\y\in K$, we set $\bar y^{\y,\omega} := y^{\bar u^{\omega,\y},\y,\omega}$ and  $\bar p^{\y,\omega} := p^{\bar u^{\omega,\y},\y,\omega}$. The main result that we are going to establish reads as follows.
\begin{theorem}\label{thm:diff_G_top}
    Let $\omega\in \mathfrak Y(\DD)$ be open. 
    \begin{itemize}
        \item[(a)] For  $ \mathcal U = \mathbf{R}$, the topological derivative of $\omega\mapsto \mathcal J_2(\omega)$ at $\omega$ in $\eta_0$  is given by
	\ben\label{eq:topo_main_J2}
	\begin{displaystyle}
		\mathcal T\mathcal J_2(\omega)(\eta_0) = \max_{\y\in \mathfrak X_2(\omega)}
            - \text{sgn}_\omega(\eta_0)	\int_0^T \overbar u^{\y,\omega}(s) \bar p^{\y,\omega}(\eta_0,s)\;ds 	\end{displaystyle}
	\een
    for $\eta_0\in \omega\cup (\DD\setminus\overline\omega)$, 
\item[(b)]	For  $ \mathcal U = L_2(\Omega)$, the topological derivative of $\omega\mapsto \mathcal J_2(\omega)$ at $\omega$ in $\eta_0$  is given by
	\ben\label{eq:topo_main_J2_}
	\begin{displaystyle}
		\mathcal T\mathcal J_2(\omega)(\eta_0) = \frac{1}{4\gamma}\max_{\y\in \mathfrak X_2(\omega)}
            - 		\text{sgn}_\omega(\eta_0)	\int_0^T (\bar p^{\y,\omega}(\eta_0,s))^2\;ds 	\end{displaystyle}
	\een
    for $\eta_0\in \omega\cup (\DD\setminus\overline\omega)$, 
\end{itemize}

    The adjoint $\bar p^{\y,\omega}$ belongs to $C([0,T]\times B_\delta(\eta_0))$ and satisfies
	\begin{alignat}{2}\label{eq:adjoint}
	-\partial_t \bar p^{\y,\omega}  - \Delta \bar p^{\y,\omega} =  -2 \bar y^{\y,\omega} &  & &\quad \text{ in } \DD \times (0,T], \\
	\bar p^{\y,\omega}=0 &  & &\quad \text{ on }\partial \DD\times (0,T],\\
	\bar p^{\y,\omega}(T)=0 &  & &\quad \text{ in } \DD.
	\end{alignat}
\end{theorem}

\begin{corollary}\label{cor:topo_invariant}
	Let the hypotheses of Theorem~\ref{thm:diff_G_top} be satisfied. Assume that if $v\in \mathcal U$ then $-v\in \mathcal U$. Then we have
	\ben
	\mathcal T\mathcal J_1(\omega,-\y)(\eta_0) = \mathcal T\mathcal J_1(\omega,\y)(\eta_0)
	\een
	for all $\eta_0\in \DD\setminus \partial \omega$ and $ \y\in V$.
\end{corollary}
\begin{proof}
	Let $\y\in V$ be given.  From the optimality system \eqref{eq:optimality_system} and the assumption that $v\in \mathcal U$ implies $-v\in \mathcal U$, we infer that $\overbar u^{-\y,\omega} = -\overbar u^{\y,\omega}$, $\bar y^{-\y,\omega} =- \bar y^{\y,\omega}$ and $\bar p^{-\y,\omega} =- \bar p^{\y,\omega}$. Now the result follows.
\end{proof}

\subsection{Averaged adjoint equation and Lagrangian}

Throughout this section we fix an open set $\omega\in \mathfrak Y(\DD)$ and pick $\eta_0 \in \Omega\setminus\partial\omega$ and define $\omega_\eps := \omega\setminus \overbar B_\eps(\eta_0)$, $\eps >0$.

For every quadruple $(u,\y,y,p)\in \UU \times K\times W(0,T)\times W(0,T)$ and every  $\eps \ge 0$ we define the parametrised Lagrangian,
\ben\label{eq:lagrangian_topo}
\begin{split}
	{\tilde G}(\eps,u,\y,y,p)  := & \int_{\DD_T}y^2  +  \gamma u^2  \;dx \; dt
	+ \int_{\DD_T} \partial_t y p + \nabla y \cdot \nabla p \; dx\,dt\\
	& -  \int_{\DD_T} \chi_{\omega_\eps} u p \;dx\,dt + \int_{\DD}(y(0)-f)p(0)\;dx.
\end{split}
\een

We denote by $y^{u,f,\eps}\in W(0,T)$ the solution of the state equation \eqref{eq:state_system} with $\chi=\chi_{\omega_\eps}$ in \eqref{eq:statea}. Then, similarly to  \eqref{eq:averaged}, we introduce the averaged adjoint: find $q^{u,f,\eps}\in W(0,T)$, such that
\ben
\int_0^1\partial_y {\tilde G}(\eps,u,\y,\sigma y^{u,f,\eps}+(1-\sigma)y^{u,f,\omega},q^{u,f,\eps})(\varphi)\; d\sigma =0\quad \text{ for all 	} \varphi \in W(0,T)
\een
or equivalently after partial integration in time, $q^{u,f,\eps}(T) = 0$ and
\ben\label{eq:averaged_weak}
\int_{\Omega_T} -\varphi \partial_t q^{u,f,\eps}  +  \nabla \varphi \cdot \nabla  q^{u,f,\eps} \; dx\;dt   =  -\int_{\Omega_T}   (y^{u,f,\eps} + y^{u,f,\omega}) \varphi \; dx\;dt
\een
for all $\varphi\in W(0,T)$. For $\eps =0$ we obtain the usual adjoint for which we set $p^{u,f,\omega}:= q^{u,f,0}$.

\subsection{Proof of Theorem~\ref{thm:diff_G_top}}

\begin{lemma}\label{lem:averaged_topo}
	Let $\delta>0$ be such that $\bar B_\delta(\eta_0) \Subset \DD$.
	For all sequences $\eps_n\in(0,1]$, $u_n,u\in \UU$ and $f_n,f\in K$, such that
	\ben
	u_n \rightharpoonup u \quad\text{ in } \UU, \quad f_n \rightharpoonup f \quad \text{ in } V, \quad \eps_n\to 0, \quad \text{ as } n\to \infty,
	\een
	we have
	\ben
	\begin{split}
		q^{u_n,f_n,\eps_n} \to & p^{u,f, \omega } \quad\text{ in } L_2(0,T;H^1_0(\DD)) \quad \text{ as } n\rightarrow \infty,\\
		q^{u_n,f_n,\eps_n} \rightharpoonup & p^{u,f, \omega } \quad\text{ in }  H^1(0,T;L_2(\DD))\quad \text{ as } n\rightarrow \infty.
	\end{split}
	\een
	Moreover there is a subsequence $(	q^{u_{n_k},f_{n_k},\eps_{n_k}})$, such that
	\ben\label{eq:compactness_p}
	q^{u_{n_k},f_{n_k},\eps_{n_k}} \to  p^{u,f, \omega } \quad\text{ in }  C([0,T]\times \bar B_\delta(\eta_0))\quad \text{ as } n\rightarrow \infty.
	\een
\end{lemma}
\begin{proof}
	The first two statements follow by similar arguments as used in Lemma~\ref{lem:averaged_shape}. To prove the third we have by interior regularity of parabolic equations   that
	\ben
	q^{u,f,\eps} \in \tilde Z(0,T) :=  L_2(0,T;H^4(B_\delta(\eta_0)))\cap H^1(0,T;H^1_0(B_\delta(\eta_0)))\cap H^2(0,T;L_2(B_\delta(\eta_0)))
	\een
	and we have the apriori bound
	\ben
	\begin{array}l
		\sum_{k=0}^2\|\big(\frac{d}{dt}\big )^kq^{u,f,\eps}\|_{L_2(0,T;H^{4-2k}(B_\delta(\eta_0)))}\\[1.5ex] \qquad \le c(\|y^{u,f,\eps} + y^{u,f}\|_{L_2(H^2)} + \|\frac{d}{dt}(y^{u,f,\eps} + y^{u,f})\|_{L_2(L_2)}),
	\end{array}
	\een
	see e.g.  \cite[p.365-367, Thm.6]{EV98}.
	Hence \eqref{eq:compactness_p} follows since the space $\tilde Z(0,T)$ embeds compactly into $C([0,T]\times \bar B_\delta(\eta_0))$ .
\end{proof}

\begin{lemma}\label{lem:aux_L2}
Let $U:= L_2(0,T;L_2(\Omega))$. For $f\in K$ and $u\in \UU$ introduce $J(\eps,u,f) := \int_{\Omega_T} (y^{u,f,\eps})^2 + \gamma u^2 \;dx \;dt$. For $\eps \ge 0$ small, let $\hat \chi_{\omega_\eps} := \frac{1}{2\gamma}\chi_{\omega_\eps}$ and denote by $\bar p_\eps $ the adjoint associated to the optimal control $\bar u_\eps  :=\bar u^{f,\omega_\eps }$, that means that $\bar p_\eps$ and $\bar u_\eps $ solve
 \begin{equation}
     -\partial_t \bar p_\eps - \Delta \bar p_\eps = -2y^{\bar u_\eps,f,\eps} \quad \text{ in } \Omega\times (0,T]
 \end{equation}
 with the terminal condition $\bar p_\eps(T) =0$ in $\Omega$  and 
 \begin{equation}
     \min_{u\in L_2(0,T;L_2(\Omega))} J(\eps,u,f) = J(\eps,\bar u_\eps,f). 
 \end{equation}  

 For $\eta_0\in \Omega\setminus\partial \omega$, we have
 \begin{equation}\label{eq:first} 
     \lim_{\eps\searrow 0} \frac{J(0,\hat \chi_{\omega_\eps} \bar p_\eps,f) - J(0, \hat \chi_\omega \bar p_\eps, f)}{|B_\eps(0)|} =  \frac{1}{4\gamma} \int_0^T (\bar p^{f,\omega}(\eta_0,t))^2\;dt.
\end{equation}
and 
\begin{equation}\label{eq:second}
\lim_{\eps\searrow 0} \frac{J(\eps,\hat\chi_{\omega_\eps} \bar p ,f) - J(\eps, \hat \chi_{\omega}\bar p, f)}{|B_\eps(\eta_0)|} = -\frac{1}{4\gamma} \int_0^T (\bar p^{f,\omega}(\eta_0,t))^2 \;dt. 
\end{equation}
\end{lemma}
\begin{proof}
    Let $\bar p:= \bar p^{f,\omega}$. We first prove \eqref{eq:first} for $\eta_0 \in \Omega\setminus \overline \omega$. For this we note that $\chi_{\omega} (\hat \chi_{\omega_\eps}\bar p_\eps) = \chi_{\omega} (\hat \chi_{\omega} \bar p_\eps)$ a.e. in $\Omega$. Therefore $y^{\hat \chi_{\omega_\eps}\bar p_\eps , f,\omega} = y^{\hat \chi_{\omega}\bar p_\eps, f,\omega}$. Thus 
\begin{equation}
    \begin{split}
        J(0,\hat \chi_{\omega_\eps} \bar p_\eps,f) - J(0, \hat \chi_\omega \bar p_\eps, f)  & =  \int_0^T\int_\Omega \gamma (\hat \chi_{\omega_\eps} \bar p_\eps)^2-\gamma (\hat \chi_{\omega}\bar p_\eps)^2\;dx\;dt   \\
                                                                                            & =  \frac{1}{4\gamma} \int_0^T\int_{B_\eps(\eta_0)} \bar p_\eps^2\;dx\;dt
\end{split}
\end{equation}
Dividing by $|B_\eps(\eta_0)|$, passing to the limit $\eps\searrow 0$ and using $\bar p_\eps \to \bar p$ in $C(0,T;B_\delta(\eta_0))$ for $\delta >0$ small,  yields \eqref{eq:first}. In a similar fashion we will prove \eqref{eq:second} for $\eta_0\in \omega$. For this we note that $\chi_{\omega_\eps} (\hat \chi_{\omega_\eps}\bar p )  = \chi_{\omega_\eps} (\hat \chi_{\omega}  \bar p)$ a.e. in $\Omega$. Therefore $y^{\hat \chi_{\omega_\eps}\bar p,f,\eps}= y^{\hat \chi_{\omega}\bar p,f,\eps}$. It follows  
\begin{equation}
    \begin{split}
        J(\eps,\hat \chi_{\omega_\eps} \bar p,f) - J(\eps, \hat \chi_\omega \bar p, f) & =  \int_0^T\int_\Omega \gamma (\hat \chi_{\omega_\eps} \bar p)^2-\gamma (\hat \chi_{\omega}\bar p)^2\;dx\;dt   \\
                                                                                       & =  -\frac{1}{4\gamma} \int_0^T\int_{B_\eps(\eta_0)} \bar p^2\;dx\;dt.
\end{split}
\end{equation}
Dividing by $|B_\eps(\eta_0)|$ and passing to the limit $\eps\searrow 0$, using that $\bar p \in C(0,T;B_\delta(\eta_0))$ yields \eqref{eq:second}.

We assume now $\eta_0\in \omega$ and prove \eqref{eq:first}. In the following we choose $\bar p_\eps:= \bar p^{f,\omega_\eps}$. In order to prove \eqref{eq:first} we first recall that for all $\eps \ge 0$ small:
\begin{equation}
    J(0,\hat \chi_{\omega_\eps}\bar p_\eps,f) = \int_{\Omega_T} y_\eps^2 + \frac{1}{4\gamma} \chi_{\omega_\eps}\bar p_\eps^2 \;dx\;dt,
\end{equation}
where $y_\eps$ solves
\begin{equation}\label{eq:aux_problem}
    \partial y_\eps - \Delta y_\eps = \frac{1}{2\gamma}\chi_{\omega_\eps} \bar p_\eps \quad \text{ in } \Omega\times (0,T]
\end{equation}
with initial condition $y_\eps(0) = f$ in $\Omega$. Expressed differently we have $y^{\hat \chi_{\omega_\eps}, f,  \bar p_\eps}$. Note also that $\chi_\omega \chi_{\omega_\eps} = \chi_{\omega_\eps}$, since $\eta_0\in \omega$.  Notice that in the differential quotient on the left hand side of \eqref{eq:first} the given function $\bar p_\eps$ appears on both sides. In order to verify \eqref{eq:first}, we use the averaged adjoint approach. For this purpose we introduce the $\bar p_\eps$-dependent Lagrangian:
\begin{equation}
\begin{split}
    \mathcal L(\eps,y,p;\bar p_\eps)  := & \int_{\DD_T}y^2  +  \frac{1}{4\gamma} \chi_{\omega_\eps} \bar p_\eps^2  \;dx \; dt
	+ \int_{\DD_T} \partial_t y p + \nabla y \cdot \nabla p \; dx\,dt\\
                                & -  \int_{\DD_T}  \frac{1}{2\gamma} \chi_{\omega_\eps} p \bar p_\eps \;dx\,dt + \int_{\DD}(y(0)-f)p(0)\;dx
\end{split}
\end{equation}
for all $p\in Z(0,T)$. 
With this definition of the Lagrangian we have by definition $J(0,\bar p_\eps \hat \chi_{\omega_\eps},f) = \mathcal L(\eps,y_\eps, p; \bar p_\eps)$ for all $ p \in Z(0,T)$. Let $\bar q_\eps $ be the averaged adjoint variable for the Lagrangian associated with $y_\eps$, that is, 
\begin{equation}\label{eq:weak_q_eps}
    \int_0^1 \partial_y \mathcal L(\eps,sy_\eps + (1-s)y_0,\bar q_\eps ;\bar p_\eps)(\psi) = 0 \text{ for all } \psi. 
\end{equation}
or explicitly 
\begin{equation}
   -\partial_t \bar q_\eps  - \Delta \bar q_\eps  = -(y_\eps + y_0) \quad \text{ in } \Omega\times (0,T]
\end{equation}
with the terminal condition $\bar q_\eps (T)=0$. In particular, as in Lemma~\ref{lem:averaged_topo} , one may show that $q_\eps \to \bar p$ in $C(0,T;B_\delta(\eta_0))$ for $\delta >0$ small. This will be used in the last step of this proof. Note that by the mean value theorem we have 
\begin{equation}
    \mathcal L(\eps,y_\eps,\bar q_\eps ;\bar p_\eps) - \mathcal L(\eps,y_0,\bar q_\eps ;\bar p_\eps) = \int_0^1 \partial_y \mathcal L(\eps,sy_\eps + (1-s)y_0,\bar q_\eps ; \bar p_\eps)(y_\eps - y_0)\;ds.
\end{equation}
Choosing in \eqref{eq:weak_q_eps} the test function $\psi := y_\eps-y_0$ implies $\mathcal L(\eps,y_\eps,\bar q_\eps ;\bar p_\eps) = \mathcal L(\eps,y_0, \bar q_\eps ;\bar p_\eps)$. Using that the Lagrangian $p\mapsto \mathcal L(\eps,y_\eps,p;\bar p_\eps)$ is constant on $Z(0,T)$, this implies in turn that $\mathcal L(\eps,y_\eps,p;\bar p_\eps) = \mathcal L(\eps,y_0, \bar q_\eps ;\bar p_\eps)$ for all $p\in Z(0,T)$. Therefore one readily checks
\begin{equation}\label{eq:difference_averaged}
    \begin{split}
        J(0,\hat\chi_{\omega_\eps} \bar p_\eps,f) - J(0, \hat \chi_\omega \bar p_\eps, f) & = \mathcal L(\eps,y_\eps,0;\bar p_\eps)-\mathcal L(0,y_0,0;\bar p_\eps) \\
                                                                                     & =  \mathcal L(\eps,y_0,\bar q_\eps ;\bar p_\eps)-\mathcal L(0,y_0,\bar q_\eps;\bar p_\eps)\\
                                                                                                          &= \int_0^T\int_{B_\eps(0)} -\frac{1}{4\gamma}\bar p_\eps^2 + \frac{1}{2\gamma} \bar p_\eps \bar q_\eps \;dx\;dt.
    \end{split}
\end{equation}
Since according to Lemma~\ref{lem:averaged_topo} the function $\bar q_\eps $ converges in $C(0,T;B_\delta(\eta_0))$ to $\bar p$ for $\delta>0$ small, we obtain 
by dividing \eqref{eq:difference_averaged} by $|B_\eps(0)|$ and letting $\eps$ go to zero that 
\begin{equation}
    \lim_{\eps\searrow 0} \frac{J(0,\hat \chi_{\omega_\eps} \bar p_\eps,f) - J(0, \hat \chi_\omega \bar p_\eps, f)}{|B_\eps(0)|} =  \int_0^T \frac{1}{4\gamma} (\bar p(\eta_0,t))^2 \;dt.
\end{equation}
This shows that \eqref{eq:first} holds for $\eta_0\in \omega$ also.  For $\eta_0\in \Omega\setminus\overline\omega$ the equality \eqref{eq:second} is proved similarly. 

\end{proof}

\noindent
\textbf{Proof of Theorem~\ref{thm:diff_G_top}}
(Step 1) We fix $f\in K$ and derive the topological derivative of $\omega\mapsto \mathcal J_1(\omega,f)$. We introduce $J(\eps,u,f) := \int_{\Omega_T} (y^{u,f,\eps})^2 + \gamma u^2 \;dx \;dt$. 
Recall the definition of ${\tilde G}$ in \eqref{eq:lagrangian_topo} and  testing  the averaged adjoint equation \eqref{eq:averaged_weak}  with $\varphi = y^{u,f,\eps} - y^{u,f\omega}$, we have for all $(\eps,u,f)\in [0,1]\times \UU \times K$
\ben\label{eq:cost_lagrange}
J(\eps,u,\y) = {\tilde G}(\eps,u,\y,y^{u, f ,\omega},q^{u,f,\eps}),
\een
which reduces for $\eps = 0$ to 
\begin{equation}\label{eq:cost_lagrange2}
    J(0,u,\y) = {\tilde G}(0,u,\y,y^{u, f ,\omega},p^{u,f,\omega}).
\end{equation}

The unique solution of the control problem associated to $f \in K$ and $\omega_\eps$, $\eps>0$, is denoted by $\bar u_\eps := \bar u^{f,\omega_\eps}$. 
The associated unique  adjoint and averaged adjoint variable are denoted $\bar p_\eps = p^{\bar u_\eps,f,\omega_\eps}$ and $\bar q_\eps = q^{\bar u_\eps,f,\omega_\eps}$, respectively. Note $\bar p_0 = \bar q_0 = p^{\bar u,f,\omega}$.

In the following we use the abbreviation $\hat \chi_{\omega_\eps} = (2\gamma)^{-1} \chi_{\omega_\eps}$. By \eqref{eq:cost_lagrange} and \eqref{eq:cost_lagrange2}, we have
\begin{equation}\label{eq:lagrange_equality}
    J(\eps,\bar u_\eps,f)-J(0,\bar u_\eps,f) = -\text{sgn}_{\omega}(\eta_0)\int_0^T\int_{B_\eps(\eta_0)} \bar u_\eps \bar q_\eps\;dx\;dt.
\end{equation}
 This implies with $\bar u_\eps = \hat \chi_{\omega_\eps}\bar p_\eps$:
\begin{equation}\label{eq:lagrange_equality2}
    J(\eps, \bar u_\eps, f) - J(0, \bar u, f) = J(0,\bar u_\eps,f) - J(0,\bar u, f) -\text{sgn}_{\omega}(\eta_0)\frac{1}{2\gamma}\int_0^T\int_{B_\eps(\eta_0)} \chi_{\omega_\eps} \bar p_\eps \bar q_\eps\;dx\;dt.
\end{equation}
Notice that the last term on the right hand side is zero if $\eta_0\in \omega$. 
By optimality of $ \bar u$ and since $\hat \chi_\omega \bar p_\eps\in L_2(0,T;L_2(\Omega))$, we have $J(0, \bar u, f)\le J(0,\hat\chi_{\omega} \bar p_\eps, f)$ and consequently
\begin{equation}\label{eq:limit_proof1}
    \begin{split}
        J(\eps, \bar u_\eps, f) - J(0, \bar u, f)  \ge &  J(0,\hat\chi_{\omega_\eps} \bar p_\eps ,f) - J(0, \hat \chi_{\omega}\bar p_\eps, f) \\
                                                    & -\text{sgn}_{\omega}(\eta_0)\frac{1}{2\gamma}\int_0^T\int_{B_\eps(\eta_0)} \chi_{\omega_\eps} \bar p_\eps \bar q_\eps\;dx\;dt.
\end{split}
\end{equation}

On the other hand, $\bar u = \hat \chi_\omega \bar p$ and by optimality of $\bar u_\eps$ we have $J(\eps,\bar u_\eps,f) \le  J(\eps,\hat \chi_{\omega_\eps}\bar p, f)$, thus utilizing again \eqref{eq:cost_lagrange}, \eqref{eq:cost_lagrange2} gives:
\begin{align}
    \begin{split}
        J(\eps , \bar u_\eps, f) - J(0, \bar u, f) & =  J(\eps,\bar u_\eps,f) - J(\eps,\bar u, f) 
                                                + J(\eps,\bar u,f) - J(0,\bar u, f)\\
                                                   & = J(\eps,\bar u_\eps,f) - J(\eps,\bar u,f) - \text{sgn}_{\omega}(\eta_0)\int_0^T\int_{B_\eps(\eta_0)} \bar u \bar q_\eps\;dx\;dt\\
                                                   &\le  J(\eps, \hat \chi_{\omega_\eps} \bar p,f) - J(\eps, \bar u,f) - \text{sgn}_{\omega}(\eta_0) \int_0^T\int_{B_\eps(\eta_0)} \bar u \bar q_\eps\;dx\;dt
\end{split}
\end{align}
and consequently replacing  $\bar u = \hat \chi_\omega \bar p$
\begin{equation}\label{eq:limit_proof2}
    \begin{split}
    J(\eps , \bar u_\eps, f) - J(0, \bar u, f) \le& \; J(\eps, \hat \chi_{\omega_\eps} \bar p,f) - J(\eps,\hat\chi_{\omega} \bar p,f)\\
                                                  & - \text{sgn}_{\omega}(\eta_0) \frac{1}{2\gamma}\int_0^T\int_{B_\eps(\eta_0)} \chi_\omega \bar p \bar q_\eps\;dx\;dt.
\end{split}
\end{equation}
Notice that the last term on the right hand side is zero if $\eta_0\in \Omega\setminus\overline \omega$.

We are now dividing \eqref{eq:limit_proof1} and \eqref{eq:limit_proof2} by $|B_\eps(\eta_0)|$ and want to pass to the limit $\eps\searrow0 $.  Since the family $(\bar u_\eps)$ is bounded, due to the form of the cost functional, for every null sequence $(\eps_n)$, there is a subsequence $(\eps_{n_k})$ and $\bar u\in L_2(0,T;L_2(\Omega))$, such that $u_{\eps_{n_k}}\rightharpoonup \bar u$. Since $\bar u$ is the unique solution to the control problem, we conclude that the whole family $(\bar u_\eps)$ converges to $\bar u$. As in Lemma~\ref{lem:averaged_topo} we can argue that $\bar q_\eps \to \bar p_0$ in $C([0,T];B_\delta(\eta_0))$ as well as 
$\bar p_\eps  = p^{\bar u_\eps, f,\omega_\eps} \to \bar p_0$ and $ p^{\chi_{\omega_\eps}\bar p, f, \omega_\eps} \to \bar p_0$. Therefore dividing \eqref{eq:limit_proof1} and \eqref{eq:limit_proof2} by $|B_\eps(\eta_0)|$, using Lemma~\ref{lem:aux_L2} and passing to the limit $\eps\searrow 0$ we find for $\eta_0\in \omega$
\begin{equation}\label{eq:limit__1}
    \limsup_{\eps\searrow 0}\frac{J(\eps , \bar u_\eps, f) - J(0, \bar u, f)}{|B_\eps(\eta_0)|}\le \frac{1}{4\gamma}\int_0^T  (p^{\bar u, f,\omega}(\eta_0,t))^2\; dt
\end{equation}
and 
\begin{equation}\label{eq:limit__2}
    \liminf_{\eps\searrow 0}  \frac{J(\eps,\bar u_\eps,f) - J(0,\bar u, f)}{|B_\eps(\eta_0)|} \ge  \frac{1}{4\gamma}\int_0^T  (p^{\bar u, f,\omega}(\eta_0,t))^2\; dt
.
\end{equation}
For $\eta_0\in \Omega\setminus\overline\omega$, the inequalities \eqref{eq:limit__1} and \eqref{eq:limit__2} hold with $(4\gamma)^{-1}$ replaced by $-(4\gamma)^{-1}$.  This shows that the topological derivative of  $\omega\mapsto \mathcal J_1(\omega,f)$ exists at every point $\eta_0\in \Omega\setminus\partial \omega$ and proves the formula.

To conclude the proof of the case $L_2(\Omega)$ we invoke the Lemma~\ref{lem:danskin2} of the appendix with (here we use the notation $\eps$ instead of $t$)
\begin{equation}
    \mathcal G(\eps,f):= \mathcal J(\omega_\eps,f), \quad \ell(\eps):= |B_\eps(\eta_0)|. 
\end{equation}

In view of the previous part of the proof we have for all $f$
\begin{equation}
    \lim_{t\searrow 0} \frac{\mathcal G(t,f)-\mathcal G(0,f) }{\ell(t)} =:\partial_\ell \mathcal G(0^+,\bar f)
\end{equation}
exists and is equal to the topological derivative of $\mathcal J_1(\omega,f)$ at $\eta_0$. Moreover, for every null-sequence $(\eps_n)$ and every $(f_n)\in K$, we find a subsequences $(\eps_{n_k})$ and $(f_{n_k})$ such that for some $\bar f\in K$, $\|\bar f\|_{H^1(\Omega)}\le1$ and
\begin{equation}
    f_{n_k}\rightharpoonup \bar f \quad \text{ in } H^1(\Omega). 
\end{equation}
It is readily checked that also $\bar q^{u_{n_k},\bar f_{n_k},\omega_{\eps_{n_k}}}$ and $\bar p^{u_{n_k},\bar f_{n_k},\omega_{\eps_{n_k}}}$ converge strongly to $p^{\bar u,\bar f, \omega}$ in this situation. Therefore the previous part of the proof yields (that means the proof leading to \eqref{eq:limit__1} and \eqref{eq:limit__2})
\begin{equation}
            \lim_{k\rightarrow  \infty} \frac{\mathcal G( t_{n_k} ,\bar{f}_{t_{n_k}}) -
			\mathcal G(0,\bar f_{t_{n_k}})}{\ell(t_{n_k})} = \partial_\ell \mathcal G(0^+,\bar f).
\end{equation}
Therefore Assumption~(A2) of Lemma~\ref{lem:danskin2} is also satisfied. Introduce 
\begin{equation}
g_1(\eps):= \max_{\substack{f\in K\\\|f\|_{H^1(\Omega)}=1}} \mathcal J_1(\omega_\eps,f)
\end{equation}
and 
\begin{equation}
    \mathfrak V_1(\eps):= \{f\in K: \max_{\substack{f\in K\\\|f\|_{H^1(\Omega)\le1}}} \mathcal J_1(\omega_\eps,f) = \mathcal J_1(\omega_\eps,\bar f)\}.
\end{equation}
We conclude from Lemma~\ref{lem:danskin2}:
\begin{equation}
\frac{d}{d\ell} g_1|_{t=0} = \lim_{\eps\searrow 0}  \frac{g_1(\eps) - g_1(0)}{|B_\eps(0)|} = \max_{f\in \mathfrak V_1(0)} \mathcal T\mathcal J_1(\omega,f).
\end{equation}
 This concludes the proof that the topological derivative of $\mathcal J_2$ exists for all $\eta_0\in \Omega\setminus\partial\omega$ and proves the formula \eqref{eq:topo_main_J2_}.

\underline{Case $U = \mathbf R$:} Let $\eta_0\in \Omega\setminus\partial\omega$. In order to prove Theorem~\ref{thm:diff_G_top} for $U = \mathbf R$, we apply Lemma~\ref{lem:maxmin} with
\ben
G(\eps, u, f) := J(\eps,u,f) = {\tilde G}(\eps,u,f,y^{u, f ,\omega},q^{u, f,\eps}),
\een
$\mathfrak U := \UU$, $\mathfrak V := \{ f\in K:\; \|f\|_V\le 1 \}$ and
$\ell(\eps) = |B_\eps(\eta_0)|$.  Since the minimisation problem in \eqref{eq:cost} is uniquely solvable and by Lemma~\ref{lem:maximiser} Assumptions~(A0),(A1) are satisfied.  As for Assumption~(A2) we note that for all $f\in K$, we have 
\begin{equation}\label{eq:ass_A2}
    G(\eps,\bar u^{f,\omega_\eps}, f)- G(0,\bar u^{f,\omega_\eps}, f)  =-\text{sgn}_{\omega}(\eta_0)\int_0^T\int_{B_\eps(\eta_0)} \bar u^{f,\omega_\eps} \bar q^{\bar u, f,\omega_\eps}\;dx\;dt,
\end{equation}
where $\bar q^{\bar u, f,\omega_\eps}$ denotes the averaged adjoint variable.  As the optimal control $\bar u^{f,\omega_\eps}$ is given by  
\begin{equation}\label{eq:control_ R}
    \bar u^{f,\omega_\eps}(t) :=   \bar u_\eps(t) = \frac{1}{2|\Omega|\gamma} \int_{\omega_\eps} \bar p^{\bar u_\eps, f, \omega_\eps}(t,x)\;dx
\end{equation}
it is readily checked that $\bar u^{f,\omega_\eps}$ converges to $\bar u^{f,\omega}$ in $C([0,T];B_\delta(\eta_0))$ for $\delta >0$ small. Moreover also $\bar q^{\bar u, f,\omega_\eps}$ converges to $\bar p^{\bar u,f, \omega}$ in $C([0,T];B_\delta(\eta_0))$. Therefore dividing \eqref{eq:ass_A2} by $|B_\eps(\eta_0)|$ and passing to $\eps \searrow 0$ gives 
\begin{equation}
    \lim_{\eps\searrow 0}    \frac{G(\eps,\bar u^{\eps,f}, f)- G(0,\bar u^{\eps,f}, f)}{|B_\eps(\eta_0)|} = - \text{sgn}_\omega(\eta_0) \int_0^T \bar u^{\omega,f}(t) \bar p^{\bar u,f,\omega}(t,\eta_0)\;dt.
\end{equation}
In the same fashion one can check that 
\begin{equation}
    \lim_{\eps\searrow 0}    \frac{G(\eps,\bar u^{\omega,f}, f)- G(0,\bar u^{\omega,f}, f)}{|B_\eps(\eta_0)|} = - \text{sgn}_\omega(\eta_0) \int_0^T \bar u^{\omega,f}(t) \bar p^{\bar u,f,\omega}(t,\eta_0)\;dt.
\end{equation}
It follows that
\begin{equation}\label{eq:formula_partialG_R}
\partial_{\ell} G(0^+,\bar u^{\omega,f},f) =  - \text{sgn}_\omega(\eta_0) \int_0^T \bar u^{f,\omega}(t) \bar p^{\bar u,f,\omega}(t,\eta_0)\;dt.
\end{equation}
Let us check Assumptions~(A3). Assumption~(A4) is checked similarly and left to the reader. Let a real null-sequences $(\eps_n)$ in $(0,\tau]$ and let a sequence $f_{\eps_n}$ in $\mathfrak V(t_n)$ be given. Then there is a subsequence $(\eps_{n_k})$, $f\in K$ with $\|f\|_{H^1(\Omega)}\le 1$, such that $f_{n_k}\rightharpoonup f$ weakly in $H^1(\Omega)$ and also $\bar u^{f_{n_k},\omega_{\eps_{n_k}}}\rightharpoonup \bar u^{f,\omega}$ weakly in $L_2([0,T];\mathbf R)$. By Lemma~\ref{lem:averaged_topo} we find, possible choosing another subsequence, that 
\begin{align}\label{eq:adjoint_av_adjoint}
    \begin{split}
    \bar p^{\bar u_{\eps_{n_k}, f_{n_k},\omega_{\eps_{n_k}}}} & \to \bar p^{\bar u, f, \omega}  \quad \text{ in } C([0,T];B_\delta(\eta_0)), \\
    \bar q^{\bar u_{\eps_{n_k}, f_{n_k},\omega_{\eps_{n_k}}}} & \to \bar p^{\bar u, f, \omega}  \quad \text{ in } C([0,T];B_\delta(\eta_0)),
\end{split}
\end{align}
where for simplicity we use the notation $\bar u_{\eps_n} := \bar u^{\omega_{\eps_n},f_{\eps_n}}$.
Moreover, we check as in \eqref{eq:lagrange_equality}:
\begin{equation}
G(\eps_{n_k}, \bar u_{\eps_{n_k}}, f_{n_k}) - G(0, \bar u_{\eps_{n_k}}, f_{n_k}) = -\text{sgn}_{\omega}(\eta_0)\int_0^T\int_{B_\eps(\eta_0)} \bar u^{f_{n_k},\omega_{\eps_{n_k}}} \bar q^{\bar u_{\eps_{n_k},f_{n_k}},\omega_{\eps_{n_k}}}\;dx\;dt.
\end{equation}
In view of \eqref{eq:control_ R} and \eqref{eq:adjoint_av_adjoint} we have that $\bar u_{\eps_{n_k}}\to \bar u^{\omega,f}$ in $C([0,T];B_\delta(\eta_0))$. Thus
 \begin{equation}
     \lim_{k\to\infty}     \frac{G(\eps_{n_k}, \bar u_{\eps_{n_k}}, f_{n_k}) - G(0, \bar u_{\eps_{n_k}}, f_{n_k})}{|B_{\eps_{n_k}}(\eta_0)|} = - \text{sgn}_\omega(\eta_0) \int_0^T \bar u^{\omega,f}(t) \bar p^{\bar u,f,\omega}(t,\eta_0)\;dt.
 \end{equation} 
Now by Lemma~\ref{lem:maxmin} the topological derivative follows from \eqref{eq:formula_partialG_R}.

\black

\section{Numerical approximation of the optimal shape problem}\label{sec:six}
In this section we discuss the formulation of numerical methods for optimal positioning and design which are based on the formulae introduced in previous sections. We begin by introducing the discretisation of the system dynamics and the associated linear-quadratic optimal control problem. Then, the optimal actuator design problem is addressed by approximating the shape and topological derivatives, which are embedded into a gradient-based approach and a level-set method, respectively.

\subsection{Discretisation and Riccati equation}

Let $T>0$.  We choose the spaces $K=H^1_0(\DD)$ and $\mathcal U=\R$, so that the control space $\UU$ is equal to $L_2(0,T;\R)$. The cost functional reads
\begin{align}
\mathcal{J}_1(\omega,\y) = \underset{u\in\UU}{\inf}J(\omega,u,\y)=\int\limits_0^T\|y(t)\|^2_{L^2(\Omega)}+\gamma |u(t)|^2\,dt + \alpha (|\omega|-c)^2, \quad \alpha >0,\label{eq:penal}
\end{align}
where $y$ is the solution of the state equation
\begin{alignat}{2}\label{eq:state_numerics1}
\partial_t y(x,t)   =\sigma\Delta y(x,t) +\chi_\omega(x) u(t)&  & &\quad (x,t)\in  \Omega\times (0,T], \\\label{eq:state_numerics2}
y(x,t)=0 &  & &\quad (x,t)\in  \partial \Omega\times (0,T],\\\label{eq:state_numerics3}
y(0,x)=\y &  & &\quad x\in \Omega\,,
\end{alignat}
and $\Omega$ is a polygonal domain. The cost $J$ in  \eqref{eq:penal} includes the additional term $\alpha (|\omega|-c)^2$ which accounts for the volume constraint $|\omega|=c$ in a penalty fashion. This slightly modifies the topological derivative formula, as it will be shown later. We derive a discretised version of the dynamics \eqref{eq:state_numerics1}-\eqref{eq:state_numerics3} via the method of lines. For this, we introduce a family of finite-dimensional approximating subspaces $V_h\subset {H^1_0(\Omega)}$, where h stands for a discretisaton parameter typically corresponding to gridsize  in finite elements/differences, but which can also be related to a spectral approximation of the dynamics. For each $f_h\in V_h$, we consider a finite-dimensional nodal/modal expansion of the form
\begin{equation}
f_h = \sum_{j=1}^{N} f_{j}\phi_{j}\,, \quad f_j\in\R\,,\phi_j\in V_h\,,
\end{equation}
where $\{\phi_i\}_{i=1}^{N}$ is a basis of $V_h$.  We denote the vector of coefficients associated to the expansion by $\underline f_h:=(f_1,\ldots,f_N)^\top$. In the method of lines, we approximate  the solution $y$ of \eqref{eq:state_numerics1}-\eqref{eq:state_numerics3} by a function $y_h$ in $C^1([0,T];V_h(\DD))$ of the type $$y_h(x,t) = \sum_{j=1}^{N} y_{j}(t)\phi_{j}(x)\,,$$
for which we follow a standard Galerkin ansatz. Inserting  $y_h$ in the weak formulation \eqref{eq:weak_form} and testing with $\varphi = \phi_{k}$, $k=1,\ldots, N$ leads to the following system of ordinary equations,
\ben\label{eq:discretisation}
\dot{\underline{y}}_h(t)=A_h\underline{y}_h(t)+B_hu_h(t)\quad t\in (0,T], \quad \underline{y}_h(0) = \underline{f}_h,
\een
where $M_h,K_h\in\R^{N\times N}$ and $B_h, \underline{f}_h \in \R^{N}$ are given by
\ben
\begin{split}
A_h=-M^{-1}_hS_h\,,\quad B_h = M^{-1}\hat{B}_h \,,\quad \underline{f}_h := M^{-1}_h\hat{\underline{f}}_h\,,
\end{split}
\een
with
\ben
\begin{split}
(M_h)_{ij}=(\phi_{i},\phi_{j})_{ L_2}\,,\quad (S_h)_{ij}=\sigma(\nabla \phi_{i},\nabla\phi_{j})_{L_2}\,,\\
(\hat{B}_h)_{j}=( \chi_\omega,\phi_{j})_{ L_2}\,,
\quad  (\hat f_h)_j :=(f,\phi_{j})_{L_2}, \quad  i,j=1,\ldots,N\,.
\end{split}
\een
Note that $\underline{y}_h = \underline{y}_h^{u_h,\underline{\y}_h,\omega} $ depends on $\y_h$, $u_h$, and $\omega$.  Given a discrete initial condition $\y_h\in V_h(\DD)$, the discrete costs are defined by
\ben\label{eq:min_problem_discrete}
\mathcal J_{1,h}(\omega,\y_h) := \underset{u_h\in\UU}{\inf}J_h(\omega,u,\y_h)=\underset{u_h\in\UU}{\inf} \int\limits_0^T (\underline{y}_h)^\top M_h\underline{y}_h+ \gamma | u_h(t)|^2\,dt+ \alpha (|\omega|-c)^2,
\een
and
\ben
\mathcal J_{2,h}(\omega) = \sup_{\substack{f_h\in V_h \\ \|f_h\|_{H^1}\le 1 }}\mathcal J_{1,h}(\omega,f_h).
\een
The  solution of the linear-quadratic optimal control problem in \eqref{eq:min_problem_discrete} is given by
\[\bar{u}^{\omega,\y_h}(t)=-\gamma^{-1}B_h^\top \Pi_h(t)\underline{y}_h\,,\]
where $\Pi_h\in R^{N\times N}$ satisfies the differential matrix Riccati equation
\[-\frac{d}{dt}\Pi_h=A_h\Pi_h + \Pi_hA_h - \Pi_hB_h\gamma^{-1} B_h^\top\Pi_h + M_h \quad\text{ in } [0,T), \quad\Pi_h(T)=0\,.
\]
The coefficient vector of the discrete adjoint state $\bar p_h^{\y_h,\omega}(t)$ at time $t$ can be recovered directly by
$\underline{\bar{p}}_h^{\y_h,\omega}(t)=2\Pi_h(t)\underline{y}_h(t)$. Let us define the discrete analog of \eqref{eq:frak_X2_r},

\ben
\mathfrak X_{2,h}(\omega) := \{\bar f_h\in V_h:\; \sup_{\substack{f_h\in V_h \\ \|f_h\|_{H^1}\le 1 }}\mathcal J_{1,h}(\omega,f_h) = \mathcal J_{1,h}(\omega,\bar f_h)\}.
\een
Since we have the relation
\ben
\mathcal J_{1,h}(\omega,\y_h) = (\Pi_h(0)\underline{\y}_h, \underline{\y}_h)_{L_2} + \alpha (|\omega|-c)^2,
\een
the maximisers $\y_h\in \mathfrak X_{2,h}(\omega)$ can be computed by solving the generalised Eigenvalue problem: find $(\lambda_h,\y_h)\in \R\times V_h$ such that
\ben\label{eq:generalised_EV_problem}
(\Pi_h(0) - \lambda_h S_h)\underline{\y}_h =  0.
\een
The biggest $\lambda_h = \lambda_h^{max}$ is then precisely the value $\mathcal J_{2,h}(\omega)$  and the normalised Eigenvectors for this Eigenvalue are the elements in $\mathfrak X_{2,h}(\omega)$:
\ben
\mathfrak X_{2,h}(\omega) = \{ \y_h:  \underline{\y}_h \in \text{ker}((\Pi_h(0) - \lambda_h^{max} K_h)) \;\text{ and }  \; \|\underline{\y}_h\|=1  \}.
\een

\begin{remark}\label{rem:topo_singleton}
	It is readily checked that if $\y_h \in \mathfrak X_{2,h}(\omega)$, then also $-\y_h\in \mathfrak X_{2,h}(\omega)$.  So if the Eigenspace for the largest eigenvalue is one-dimensional we have $\mathfrak X_{2,h}(\omega) = \{\y_h, -\y_h\}$. However, we know according to Corollary~\ref{cor:shape_invariant} (now in a  discrete setting) that
	\ben
	\mathcal T\mathcal J_{1,h}(\omega,\y_h)(\eta_0) = \mathcal T\mathcal J_{1,h}(\omega,-\y_h)(\eta_0)
	\een
	for all $\eta_0\in \DD\setminus \partial \omega$ and $\y_h\in V_h$. Hence we can evaluate the topological derivative $\mathcal T\mathcal J_{2,h}(\omega)$ by picking either
	$\y_h$ or $-\y_h$. A similar argumentation holds for the shape derivative.
\end{remark}

\subsection{Optimal actuator positioning: Shape derivative}
Here we precise the gradient algorithm based upon a numerical realisation of the shape derivative.
We consider \eqref{eq:state_numerics1}-\eqref{eq:state_numerics3} with its discretisation \eqref{eq:discretisation}. Given a simply connected
actuator $\omega_0\subset \DD$ we employ the shape derivative
of $\mathcal J_1$ to find the optimal position. Let $\y_h\in V_h$. According to Corollary~\ref{cor:shape_derivative_J1} the derivative of $\mathcal J_{1,h}$ in the case $\mathcal U = \R$ is given by
\ben\label{eq:shape_DJ1_numerics}
D\mathcal J_{1,h}(\omega,\y_h)(X) = -\int_{\partial \omega} \bar u_h^{\y_h,\omega}(t) \int_0^T \bar p_h^{\y_h,\omega}(s,t) (X(s)\cdot \nu(s))\; ds\;dt
\een
for $X\in \ac C^1(\overbar \DD, \R^d)$. We assume that $\omega\Subset \DD$. We define the vector $b\in \R^d$ with the components
\ben
b_i := \int_{\partial \omega} \bar u_h^{\y_h,\omega}(t) \int_0^T \bar p_h^{\y_h,\omega}(s,t) (e_i\cdot \nu(s))\; ds\;dt,
\een
where $e_i$ denotes the canonical basis of $\R^d$. From this we can
construct an admissible descent direction by choosing any $\tilde X\in \ac C^1(\overbar \DD, \R^d)$ with $\tilde X|_{\partial \omega} = b$. Then it is obvious that $D\mathcal J_{1,h}(\omega,\y_h)(\tilde X)\le 0$. Let us use the the notation $b=-\nabla \mathcal J_{1,h}(\omega,\y_h)$.
We write $ (\id+t\nabla \mathcal J_{1,h}(\omega,\y_h))(\omega)$ to denote the moved actuator $\omega$ via the vector $b$. Note that only the position, but not the shape of $\omega$ changes by this operation. We refer to this procedure as Algorithm \ref{alg:shape} below.

\begin{algorithm}[H]
\begin{algorithmic}
	\STATE{ \textbf{Input:} $\omega_0\in \mathfrak Y(\DD)$, $\y_h\in V_h$, $b_0 := -\nabla \mathcal J_{1,h}(\omega_0,\y_h)$, $n=0$, $\beta_0>0$, and $\eps >0$. }
	\WHILE{ $|b_n| \ge \epsilon $ }
	\IF{ $\mathcal J_{1,h}((\id+\beta_n b_n)(\omega_n),\y_h) < \mathcal J_{1,h}(\omega_n,\y_h)$ }
	\STATE{	$\beta_{n+1} \gets \beta_n$ }
	\STATE{	$\omega_{n+1} \gets (\id+\beta_n b_n)(\omega_n)$ }
	\STATE{	$b_{n+1} \gets -\nabla \mathcal J_{1,h}(\omega_{n+1},\y_h) $ }
	\STATE{	$n \gets n+1$  }
	\ELSE
	 \STATE decrease $\beta_n$
	\ENDIF
	\ENDWHILE
	\RETURN{ optimal actuator positioning $\omega_{opt}$  }
	\caption{Shape derivative-based gradient algorithm for actuator positioning}\label{alg:shape}
\end{algorithmic}
\end{algorithm}

\subsection{Optimal actuator design: Topological derivative}
As for the shape derivative, we now introduce a numerical approximation of the topological derivative formula which is embedded into a level-set method to generate an algorithm for optimal actuator design, i.e. including both shaping and position.
According to Theorem~\ref{thm:diff_G_top} the discrete topological derivative of $\mathcal J_{1,h}$ is given by
\ben\label{eq:top_discrete}
\begin{displaystyle}
	\mathcal T{\mathcal{J}}_{1,h}(\omega,\y_h)(\eta_0) = \left.
	\begin{cases}
		\int_0^T \overbar u^{\y_h,\omega}_h(t)\bar p_h^{\y_h,\omega}(\eta_0,t)\;dt- 2\alpha(|\omega|-c) & \text{ if } \eta_0 \in \omega, \\
		-\int_0^T\overbar u^{\y_h,\omega}_h(t) \bar p_h^{\y_h,\omega}(\eta_0,t)\;dt + 2\alpha(|\omega|-c) &  \text{ if }  \eta_0 \in  \DD\setminus \overbar \omega,
	\end{cases}
	\right.
\end{displaystyle}
\een

The level-set method is well-established in the context of shape optimisation and shape derivatives \cite{AJT04}. Here we use a level-set method for topological sensitivities as proposed in \cite{AmstutzHeiko06}. We recall that compared to the the formulation based on shape sensitivities, the topological approach has the advantage that multi-component actuators can be obtained via splitting and merging.

For a given actuator $\omega\subset \Omega$, we begin by defining the function
\[
g_h^{\y_h,\omega}(\zeta) =
-\int_0^T\overbar u^{\y_h,\omega}_h(t)\overbar p^{\y_h,\omega}_h(\zeta,t)\;dt + 2\alpha(|\omega|-c), \quad \zeta\in \overbar \DD\,
\]

which is continuous since the adjoint is continuous in space. Note that $\overbar p^{\y_h,\omega}$ and $\overbar u^{\y_h,\omega}$ depend on the actuator $\omega$. For other types of state equations where the shape variable enters into the differential operator (e.g. transmission problems \cite{AM11}) this may not be the case and thus it particular of our setting.
The necessary optimality condition for the cost function $\mathcal J_{1,h}(\omega,\y_h)$ using the topological derivative are formulated as
\ben\label{eq:necessary_g}
\begin{split}
	g_h^{\y_h,\omega}(x) & \le 0 \quad \text{ for all } x\in  \omega,\\
	g_h^{\y_h,\omega}(x) & \ge 0 \quad \text{ for all } x\in \Omega\setminus \overbar \omega.
\end{split}
\een
Since $g_h^{\y_h,\omega}$ is continuous this means that $	g_h^{\y_h,\omega}$ vanishes on $\partial \omega$ and hence
\ben\label{eq:topo_stationary_point}
\int_0^T\overbar u^{\y_h,\omega}_h(t)\overbar p^{\y_h,\omega}_h(\zeta,t)\;dt = 2\alpha(|\omega|-c)\,, \quad \text{ for all } \zeta \in \partial \omega.
\een
An (actuator) shape $\omega$ that satisfies \eqref{eq:necessary_g} is referred to as
stationary (actuator) shape. It follows from \eqref{eq:top_discrete} and  \eqref{eq:necessary_g}, that $g_h^{\y_h,\omega}$ vanishes on the actuator boundary $\partial \omega$ of a stationary shape $\omega$.


We now describe the actuator $\omega$ via  an arbitrary level-set function $\psi_h \in V_h$, such that $\omega=\{x\in\Omega: \psi_h(x)<0\}$ is achieved via an update of an initial guess $\psi_h^0$
\ben
\psi^{n+1}_h = (1-\beta_n)\psi^n_h + \beta_n \frac{g_h^{\y_h,\omega_n}}{\|g_h^{\y_h,\omega_n}\|},\quad \omega_n := \{ x\in\Omega: \psi^n_h(x)<0\},
\een
where $\beta_n$ is the step size of the method. The idea behind this update scheme is the following: if $\psi^n_h(x)<0$ and $g_h^{\y_h,\omega_n}(x)>0$, then we add a positive value to the level-set function, which means that we aim at removing actuator material. Similarly, if $\psi^n_h(x)>0$ and $g_h^{\y_h,\omega_n}(x)<0$, then we create actuator material. In all the other cases the sign of the level-sets remains unchanged. We present our version of the level-set algorithm in \cite{AmstutzHeiko06}, which we refer to as Algorithm \ref{alg:topo}.

\begin{algorithm}[H]
	\begin{algorithmic}
		\STATE{\textbf{Input}: $\psi^0_h\in V_h(\DD)$, $\omega_0:=\{x\in \overbar{\DD},\psi^0_h(x)<0\}$, $\beta_0>0$, $\y_h\in V_h$, and $\eps >0$. }
		
		\WHILE{ $\|\omega_{n+1}-\omega_{n}\|\ge  \eps$ }
		\IF{ $\mathcal J_{1,h}(\{ \psi_h^{n+1}<0\},\y_h) < \mathcal J_{1,h}(\{\psi^n_h<0\},\y_h)$}
		\STATE{ $\psi^{n+1}_h \gets (1-\beta_n) \psi^n_h +  \beta_n \frac{g_h^{\y_h,\omega_n}}{\|g_h^{\y_h,\omega_n}\|}$\;}
		\STATE{	$\beta_{n+1} \gets \beta_n$}
		\STATE{	$\omega_{n+1} \gets \{\psi^{n+1}_h<0\}$}
		\STATE{	$n\gets n+1$ }
		\ELSE
		\STATE{decrease $\beta_n$\;}
		\ENDIF
		\ENDWHILE
		\RETURN{ optimal actuator $\omega_{opt}$ }
		\caption{Level set algorithm for optimal actuator design}\label{alg:topo}
	\end{algorithmic}
\end{algorithm}

Algorithm \ref{alg:topo} is embedded inside a continuation approach over the quadratic penalty parameter $\alpha$ in \eqref{eq:min_problem_discrete},  leading to actuators which approximate the size constraint in a sensible way, as opposed to a single solve with a large value of $\alpha$.

Finally, for the functional $\mathcal J_2(\omega)$ we may employ similar algorithms for shape and topological derivatives. We update the initial condition $f_h\in \mathfrak X_{2,h}(\omega) $ at each iteration whenever the actuator $\omega$ is modified.

\section{Numerical tests}\label{sec:seven}
We present a series of one and two-dimensional numerical tests exploring the different capabilities of the developed approach.

\paragraph{Test parameters and setup} We establish some common settings for the experiments. For the 1D tests, we consider a piecewise linear finite element discretisation with 200 elements over $\Omega=(0,1)$, with $\gamma=10^{-3}$, $\sigma=0.01$, $c=0.2$, and $\epsilon=10^{-7}$. For the 2D tests, we resort to a Galerkin ansatz where the basis set is composed by the eigenfunctions of the Laplacian with Dirichlet boundary conditions over $\Omega=(0,1)^2$. We utilize the first 100 eigenfunctions. This idea has been previously considered in the context of optimal actuator positioning in \cite{M11}, and its advantage resides in the lower computational burden associated to the Riccati solve. The actuator size constraint is set to $c=0.04$.
An important implementation aspect relates to the numerical approximation of the linear-quadratic optimal control problem for a given actuator. For the sake of simplicity, we consider the infinite horizon version of the costs $\cJ_1$ and $\cJ_2$. In this way, the optimal control problems are solved via an Algebraic Riccati Equation approach. The additional calculations associated to $\cJ_2$ and the set $\mathfrak X_2(\omega)$ are reduced to a generalized eigenvalue problem involving the Riccati operator $\Pi_h$. The shape and topological derivative formulae involving the finite horizon integral of $u$ and $p$ are approximated with a sufficiently large time horizon, in this case $T=1000$.
\paragraph{Actuator size constraint} While in the abstract setting the actuator size constraint determines the admissible set of configurations, its numerical realisation follows a penalty approach, i.e. $\cJ_1(\omega,f)$ is as in \eqref{eq:penal},
$$\cJ_1(\omega,f)=\cJ_1^{LQ}(\omega,f)+\cJ_1^{\alpha}(\omega)\,,$$
where $\cJ_1^{LQ}(\omega,f)$ is the original linear-quadratic (LQ) performance measure, and $\cJ_1^{\alpha}(\omega)=\alpha(|\omega|-c)^2$ is a quadratic penalization from the reference size. The cost $\cJ_2$ is treated analogously. In order to enforce the size constraint as much as possible and to avoid suboptimal configurations, the quadratic penalty is embedded within a homotopy/continuation loop. For a low initial value of $\alpha$, we perform a full solve of Algorithm \ref{alg:topo}, which is then used to initialized a subsequent solve with an increased value of $\alpha$. As it will be discussed in the numerical tests, for sufficiently large values of $\alpha$ and under a gradual increase of the penalty, results are accurate within the discretisation order.
\paragraph{Algorithm \ref{alg:topo} and level-set method} The main aspect of Algorithm \ref{alg:topo} is the level-set update of the function $\psi^{n+1}_h$ which dictates the new actuator shape. In order to avoid the algorithm to stop around suboptimal solutions, we proceed to reinitialize the level-set function every 50 iterations. This is a well-documented practice for the level-set method, and in particular in the context of shape/topology optimisation \cite{AJT04,AmstutzHeiko06}. Our reinitialization consists of reinitialising  $\psi^{n+1}_h$  to be the signed distance function of the current actuator. The signed distance function is efficiently computed via the associated Eikonal equation, for which we implement the accelerated semi-Lagrangian method proposed in \cite{AFK15}, with an overall CPU time which is negligible with respect to the rest of the algorithm.
\paragraph{Practical aspects} All the numerical tests have been performed on an Intel Core i7-7500U with 8GB RAM, and implemented in MATLAB. The solution of the LQ control problem is obtained via the \texttt{ARE} command, the optimal trajectories are integrated with a fourth-order Runge-Kutta method in time. While a single LQ solve does not take more than a few seconds in the 2D case, the level-set method embedded in a continuation loop can scale up to approximately 30 mins. for a full 2D optimal shape solve.

\subsection{Optimal actuator positioning through shape derivatives} In the first two tests we study the optimal positioning  problem \eqref{eq:min_constrained_position_problem1} of a single-component actuator of fixed width $0.2$ via the gradient-based approach presented in Algorithm \ref{alg:shape}. Tests are carried out for a given initial condition $y_0(x)$, i.e. the $\cJ_1$ setting.
\paragraph{Test 1} We start by considering $y_0(x)=\sin(\pi x)$, so the test is fully symmetric, and we expect the optimal position to be centered in the middle of the domain, i.e. at $x=0.5$. Results are illustrated in Figure \ref{fig:test1shape}, where it can be observed that as the actuator moves from its initial position towards the center, the cost $\cJ_1$ decays until reaching a stationary value. Results are consistent with the result obtained by inspection (Figure \ref{fig:test1shape} left), where the location of the center of the actuator has been moved throughout the entire domain.
\begin{figure}[!h]
	\centering
	\includegraphics[width=0.29\textwidth]{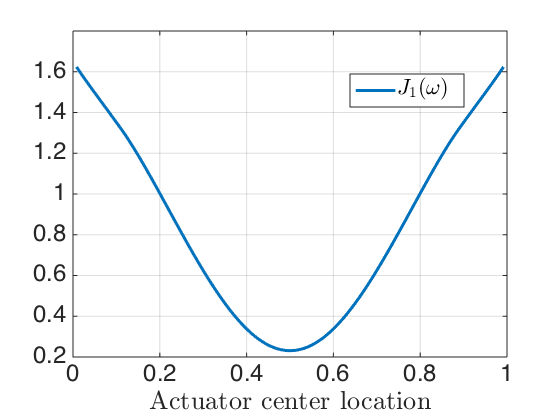}
	\includegraphics[width=0.66\textwidth]{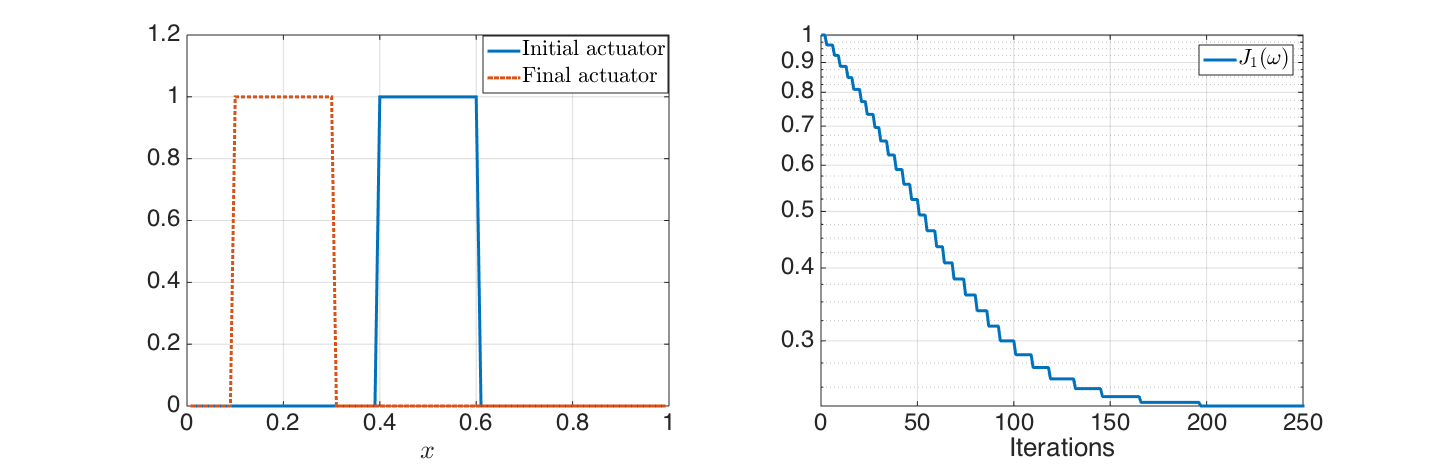}\\
	\caption{Test 1. Left: different single-component actuators with  different centers have been spanned over the domain, locating the minimum value of $\cJ_1$ for the center at $x=0.5$. Center: starting from an initial guess for the actuator far from $0.5$, the gradient-based approach of Algorithm \ref{alg:shape} locates the optimal position in the middle. Right: as the actuator moves towards the center in the subsequent iterations of Algorithm \ref{alg:shape}, the value $\cJ_1$ decays until reaching a stationary point.}
	\label{fig:test1shape}
\end{figure}

\paragraph{Test 2} We consider the same setting as in the previous test, but we change the initial condition of the dynamics to be $y_0(x)=100|x-0.7|^4+x(x-1)$, so the setting is asymmetric and the optimal position is different from the center. Results are shown in Figure \ref{fig:test2shape}, where the numerical solution coincides with the result obtained by inspecting all the possible locations.
\begin{figure}[!h]
	\centering
	\includegraphics[width=0.29\textwidth]{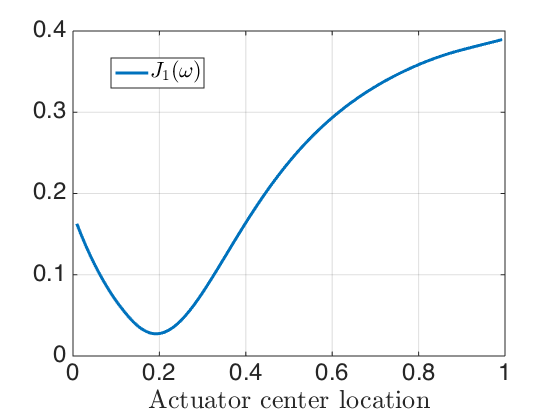}
	\includegraphics[width=0.66\textwidth]{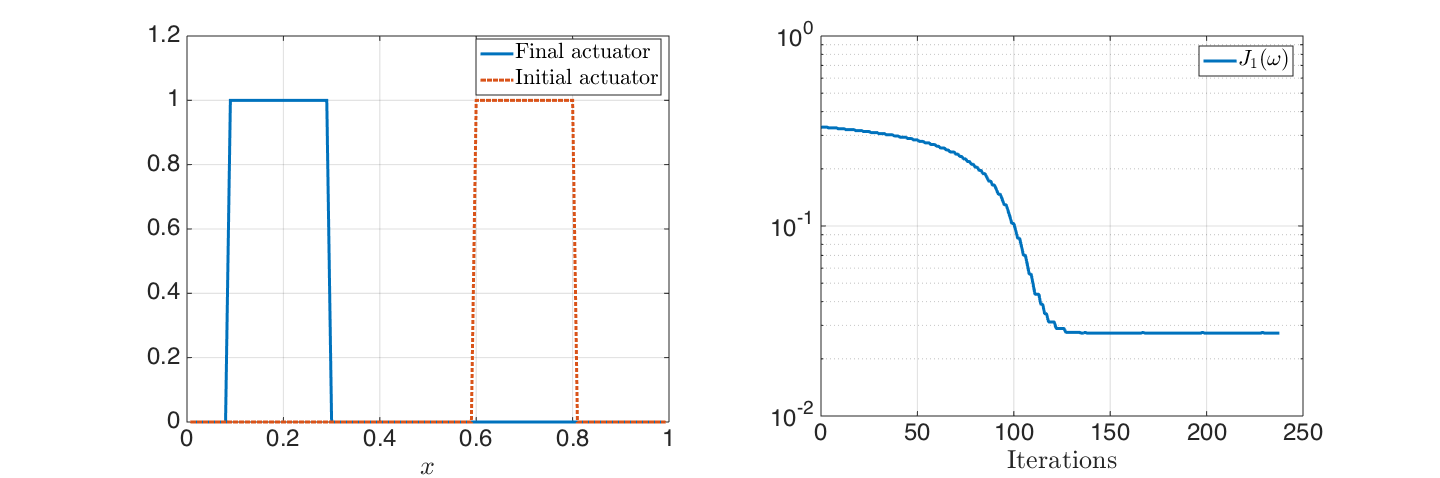}\\
	\caption{Test 2. Left: inspecting different values of $\cJ_1$ by spanning actuators with different centers, the optimal center location is found to be close to $0.2$ . Center: the gradient-based approach steers the initial actuator to the optimal position. Right:  the value $\cJ_1$ decays until reaching a stationary point, which coincides with the minimum for the first plot on the left.}
	\label{fig:test2shape}
\end{figure}

\subsection{Optimal actuator design through topological derivatives} In the following series of experiments we focus on 1D optimal actuator design, i.e. problems \eqref{eq:min_constrained_problem1} and \eqref{eq:min_constrained_problem} without any further parametrisation of the actuator, thus allowing multi-component structures. For this, we consider the approach combining the topological derivative, with a level-set method, as summarized in Algorithm \ref{alg:topo}.
\paragraph{Test 3} For $y_0(x)=\max(\sin(3\pi x),0)^2$, results are presented in Figures \ref{fig:t11d} and \ref{fig:t1it} . As it can be expected from the symmetry of the problem, and from the initial condition, the actuator splits into two equally sized components. We carried out two types of tests, one without and one with a continuation strategy with respect to $\alpha$. Without a continuation strategy, choosing $\alpha=10^3$ we obtain the result depicted in Figure \ref{fig:t11d} (b). With a continuation strategy, as the  penalty increases, the size of the components decreases until approaching the total size constraint.  The behavior of this continuation approach is shown in Table \ref{tabtest1}. When $\alpha$ is increased, the size of the actuator tends to $0.2$, the reference size, while the LQ part of $\cJ_1$, tends to a stationary value. For a final value of $\alpha=10^4$, the overall cost $\cJ_1$ obtained via the continuation approach is approx. 80 times smaller than the value obtained without any initialisation procedure, see Figure \ref{fig:t11d} (b)-(d).
Figure \ref{fig:t1it} illustrates some basic relevant aspects of the level-set approach, such as the update of the shape (left), the computation of the level-set update upon $\beta_n$ and $\psi_h^{n}$ (middle), and the decay of the value $J_1$ (right).

\begin{figure}[!h]
	\centering
	\subfloat[$y_0(x)$]{\label{fig:t1a}\includegraphics[width=0.24\textwidth]{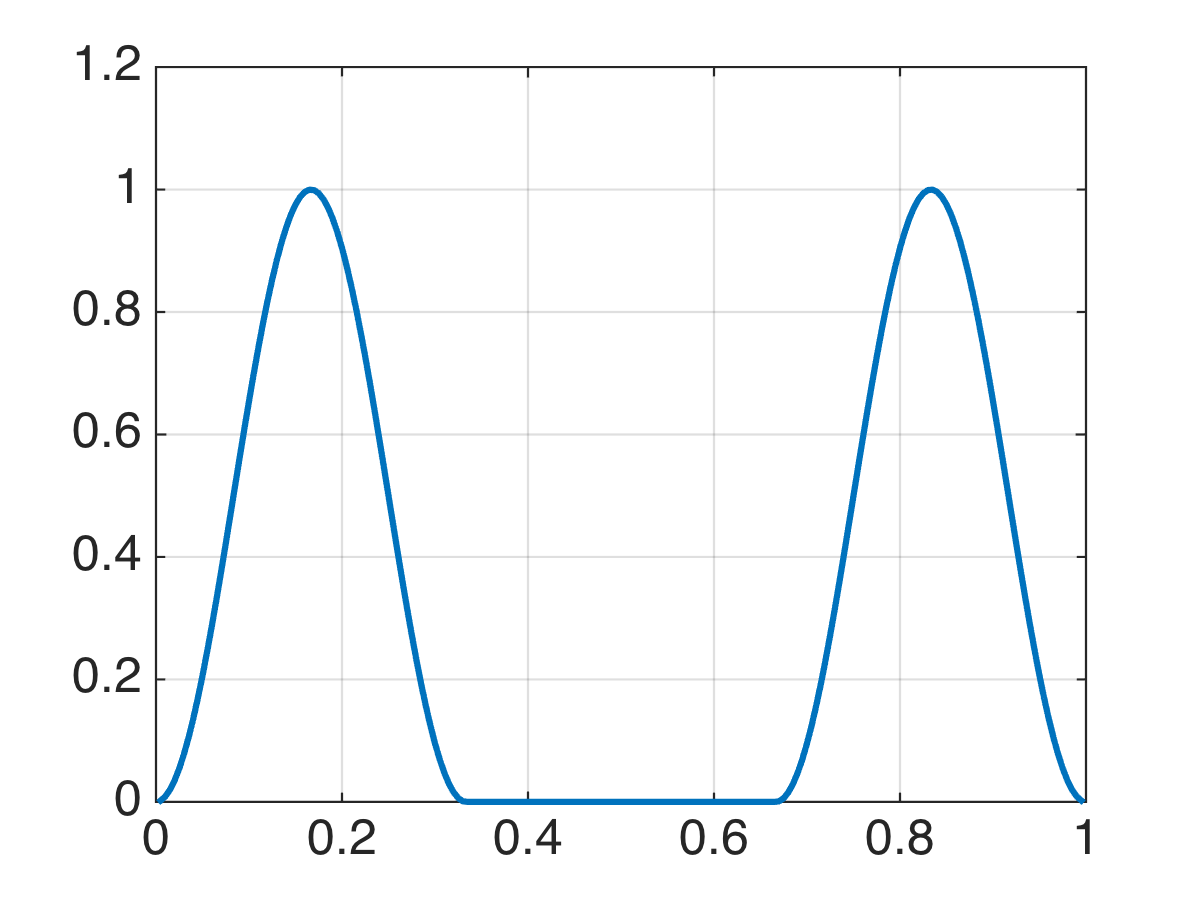}}
	\subfloat[$\alpha=10^3$, no init.]{\label{fig:t1b}\includegraphics[width=0.24\textwidth]{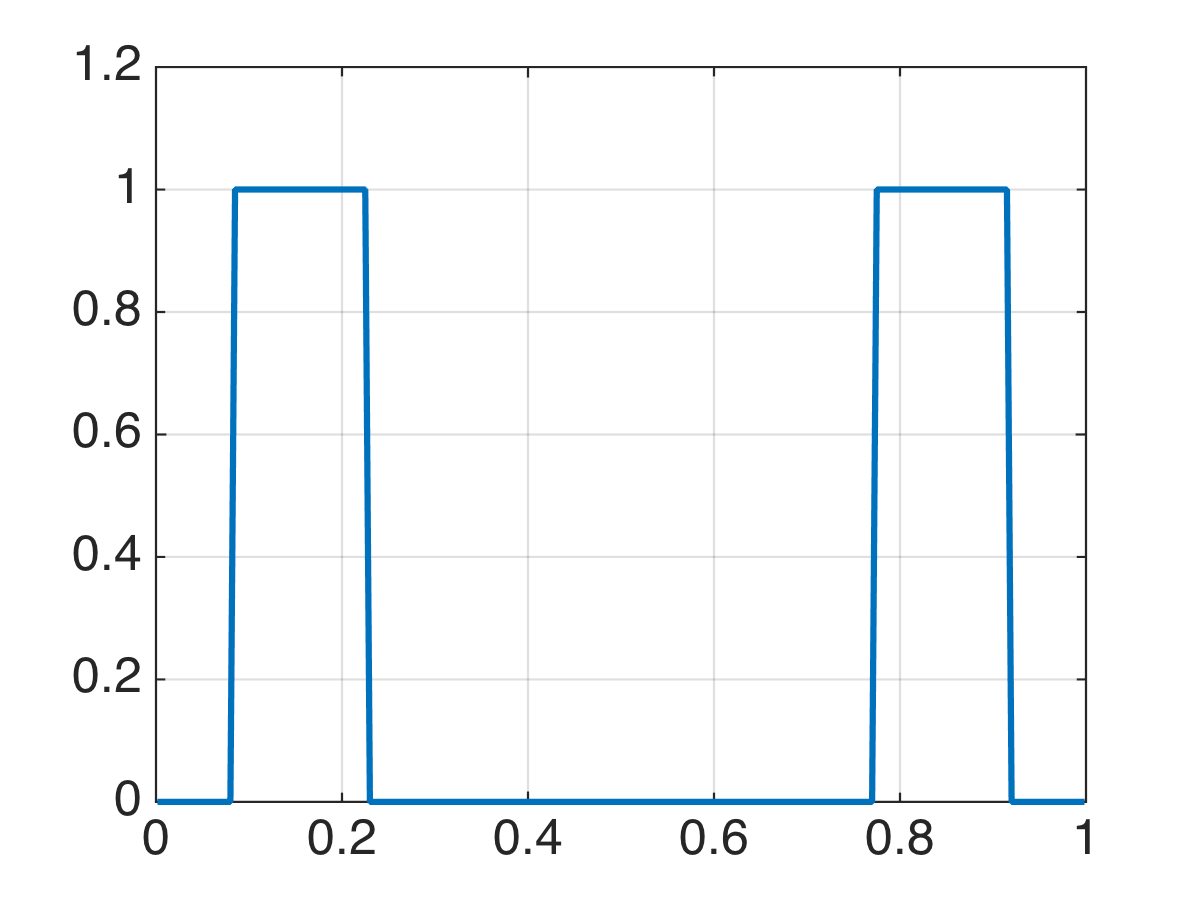}}
	\subfloat[$\alpha=10^{-1}$]{\label{fig:t1c}\includegraphics[width=0.24\textwidth]{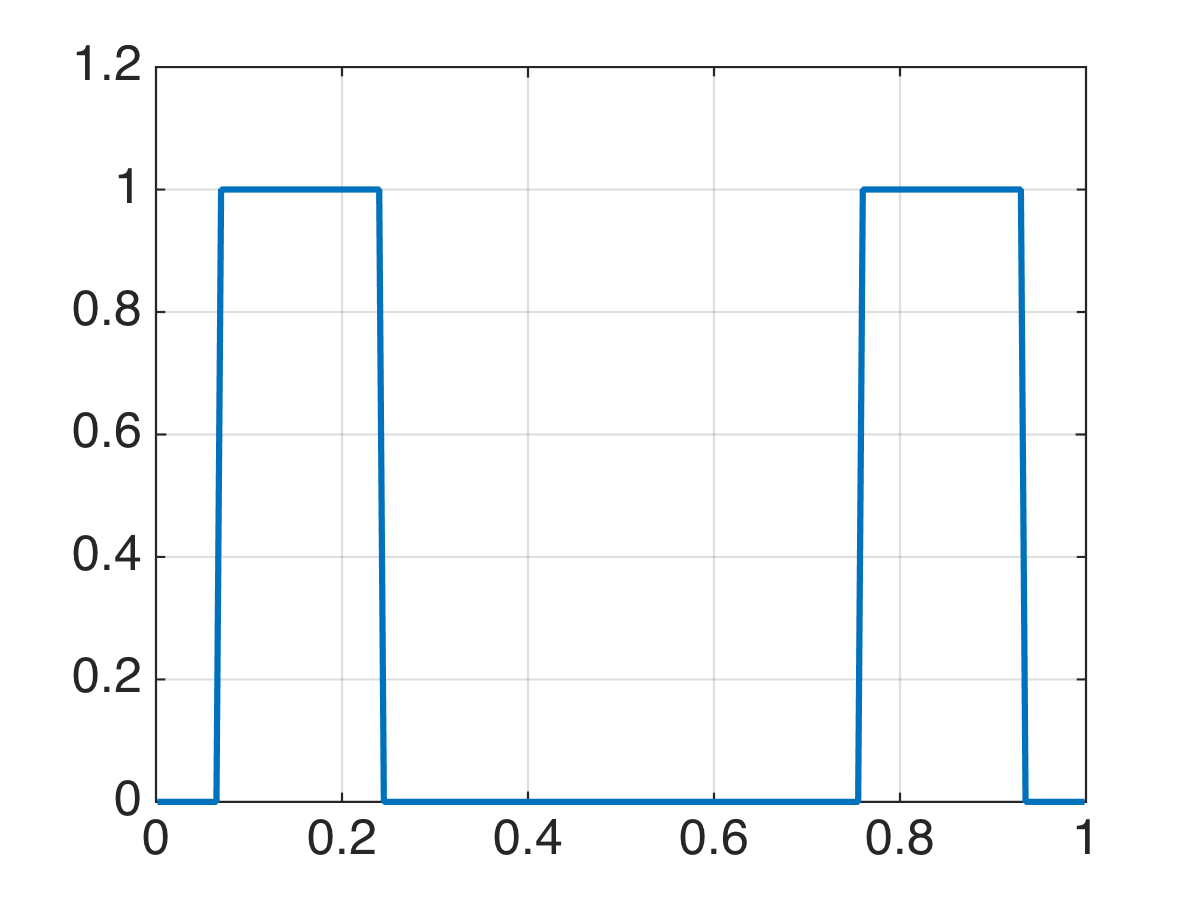}}
	\subfloat[$\alpha=10^3$]{\label{fig:t1d}\includegraphics[width=0.24\textwidth]{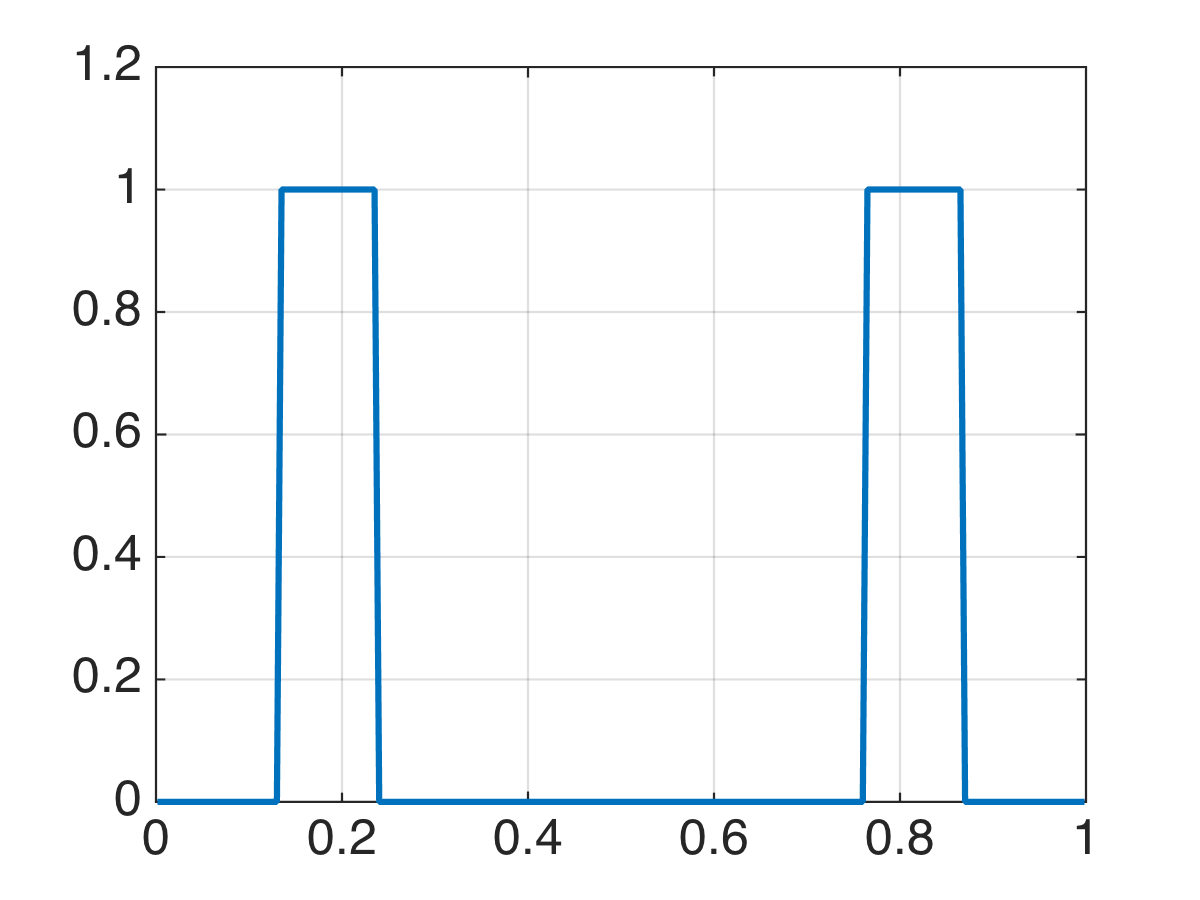}}
	\caption{Test 3. (a) Initial condition $y_0(x)=\max(\sin(3\pi x),0)^2$. (b) Optimal actuator for $\alpha=10^3$, without initialization via increasing penalization. (c) Optimal actuator for $\alpha=10^{-1}$, subsequently used in the quadratic penalty approach. (d) Optimal actuator for $\alpha=10^3$, via increasing penalization.}
	\label{fig:t11d}
\end{figure}
\begin{table}[!h]
	\centering
	\setlength{\tabcolsep}{1mm}
	\begin{tabular}{cp{1.7cm}p{1.9cm}p{2.8cm}p{1.7cm}}
		\hline\\
		$\alpha$& $\cJ_1$&   $\cJ_1^{LQ}$ & $\cJ_1^{\alpha}(size)$&iterations\\
		\hline\\
		0.1   &1.84$\times10^{-2}$ & 1.62$\times10^{-2}$ &  2.30$\times10^{-3}$ (0.35) &225 \\
		1      &2.35$\times10^{-2}$ & 2.26$\times10^{-2}$ & 9.10$\times10^{-4}$ (0.23)  &226\\
		10    &2.56$\times10^{-2}$ & 2.46$\times10^{-2}$ &  1.00$\times10^{-3}$ (0.21) &316\\
		$10^{2}$  &3.46$\times10^{-2}$ & 2.46$\times10^{-2}$ &  1.00$\times10^{-2}$ (0.21) &226\\
		$10^{3}$  &0.12       & 2.46$\times10^{-2}$ &  1.00$\times10^{-1}$ (0.21) &226\\
		$10^{3}$*  &8.18      & 8.00$\times10^{-2}$&  8.10 (0.29) &629\\
		\hline
	\end{tabular}
	\vskip 3mm
	\caption{Test 3. optimisation values for $y_0(x)=\max(\sin(3\pi x),0)^2$. Each row is initialized with the optimal actuator corresponding to the previous one, except for the last row with $\alpha=10^3*$, illustrating that incorrectly initialized solves lead to suboptimal solutions. The reference size for the actuator is $0.2$ .}\label{tabtest1}
\end{table}

\begin{figure}[!h]
	\centering
	\includegraphics[width=\textwidth]{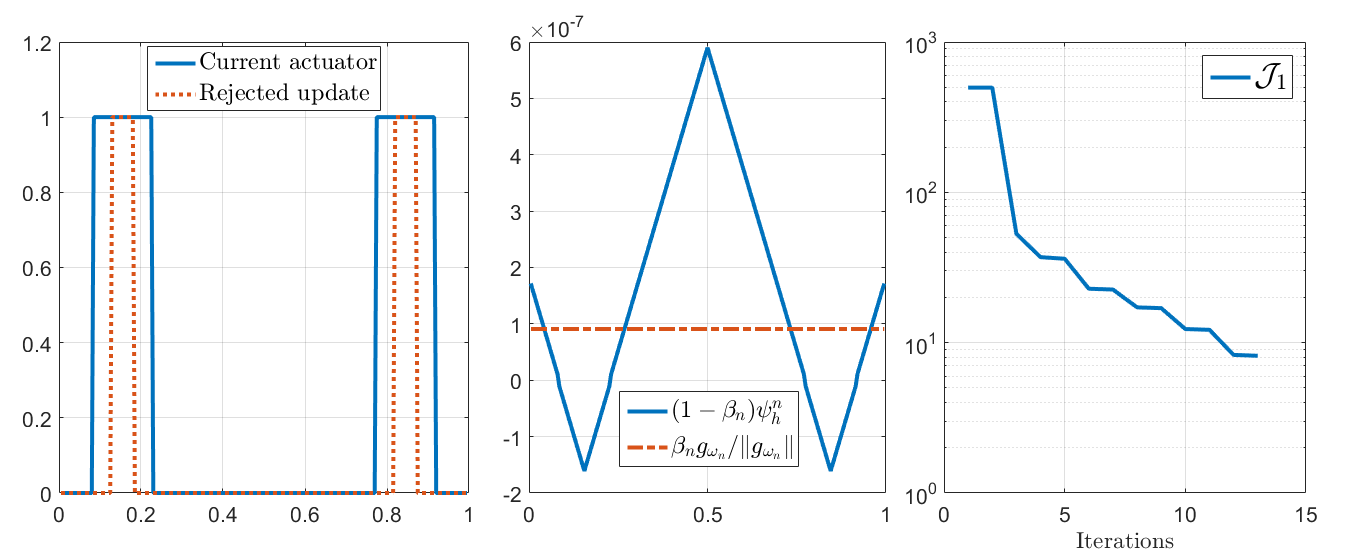}\\
	\caption{Test 3. Level set method implemented in Algorithm \ref{alg:topo}. Left: starting from an initial actuator, the topological derivative of the cost is computed and an updated actuator is obtained. The new shape is evaluated according to its closed-loop performance. If the update is rejected, the parameter $\beta_n$ is reduced. Middle: the level-set approach generates an update of the actuator shape based on the information from $\psi_h^{n}$, $\beta_n$ and $g_{\omega_n}$. Right: This iterative loop generates a decay in the total cost $J_1$, (which accounts for both the closed-loop performance of the actuator and its volume constraint).}
	\label{fig:t1it}
\end{figure}

\paragraph{Test 4} We repeat the setting of Test 3 with a nonsymmetric initial condition $y_0(x)=\sin(3\pi x)^2\chi_{\{x<2/3\}}(x)$. Results are presented in Table \ref{tabtest2} and Figure \ref{fig:t21d}, which illustrate the effectivity of the continuation approach, which generates an optimal actuator with two components of different size, see Figure \ref{t4twoc} and compare with Figure \ref{t4noinit}.
\begin{table}[!h]
	\centering
	\setlength{\tabcolsep}{1mm}
	\begin{tabular}{cp{1.7cm}p{1.9cm}p{2.8cm}p{1.7cm}}
		\hline\\
		$\alpha$& $\cJ_1$&   $\cJ_1^{LQ}$ & $\cJ_1^{\alpha}(size)$&iterations\\
		\hline\\
		0.1   &6.48$\times10^{-2}$ & 6.31$\times10^{-2}$ &  1.7$\times10^{-3}$ (0.33) &229 \\
		1      &8.0$\times10^{-2}$ & 6.31$\times10^{-2}$ & 1.69-2 (0.33)  &226\\
		10    &0.176 & 0.164 &  1.23$\times10^{-2}$ (0.235) &226\\
		$10^{2}$  &0.207 & 0.184 &  2.25$\times10^{-2}$ (0.215) &316\\
		$10^{3}$  &0.234      & 0.209 &  2.50$\times10^{-2}$ (0.195) &316\\
		$10^{4}$  &0.459       & 0.209 &  0.250 (0.195) &316\\
		$10^{4}$*  &9.09       & 9.66$\times10^{-2}$ &  9 (0.23) &629\\
		\hline
	\end{tabular}
	\vskip 3mm
	\caption{Test 4. optimisation values for $y_0(x)=\sin(3\pi x)^2\chi_{x<2/3}(x)$. Each row is initialized with the optimal actuator corresponding to the previous one, except for the last row with $\alpha=10^4*$, illustrating that incorrectly initialized solves lead to suboptimal solutions. The reference size for the actuator is $0.2$ .}\label{tabtest2}
\end{table}

\begin{figure}[!h]
	\centering
	\subfloat[$y_0(x)$]{\includegraphics[width=0.24\textwidth]{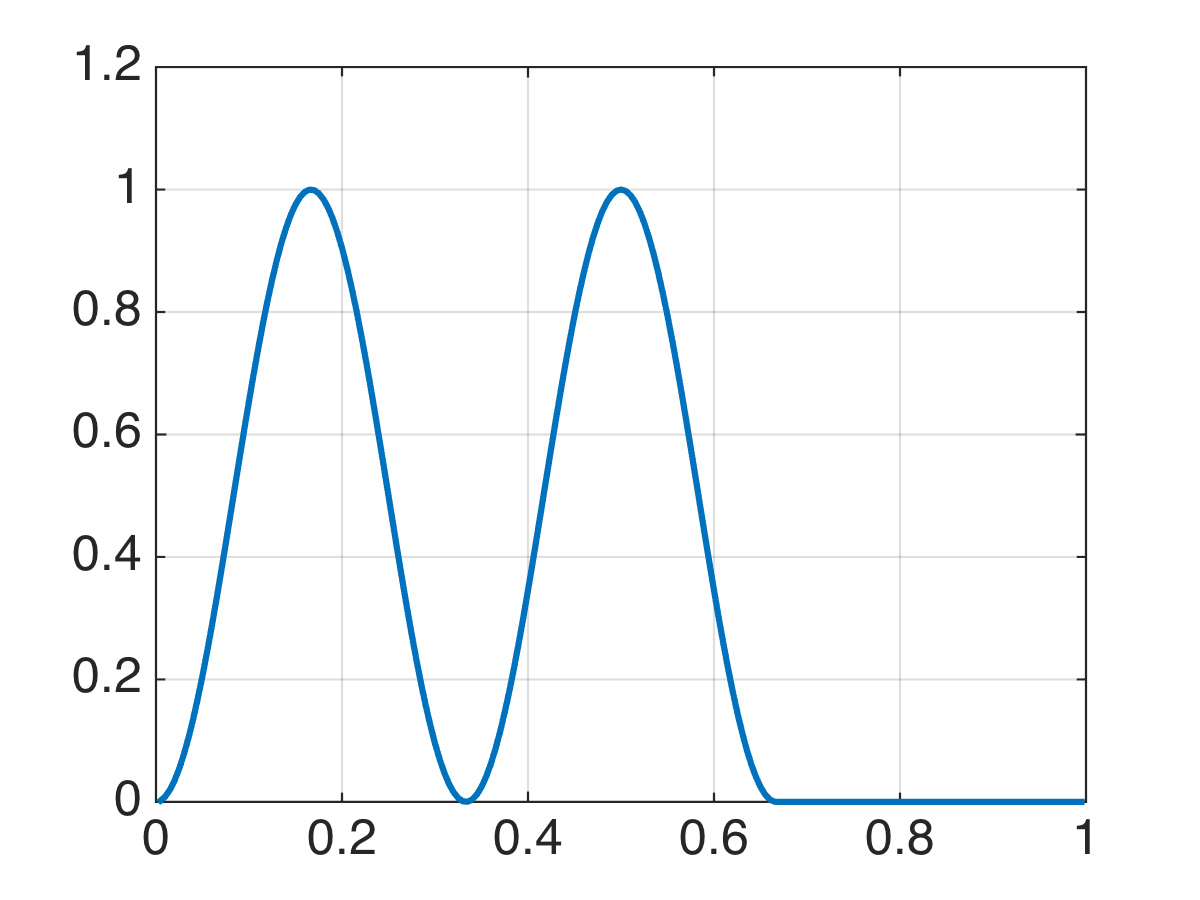}}
	\subfloat[$\alpha=10^4$, no init.]{\label{t4noinit}\includegraphics[width=0.24\textwidth]{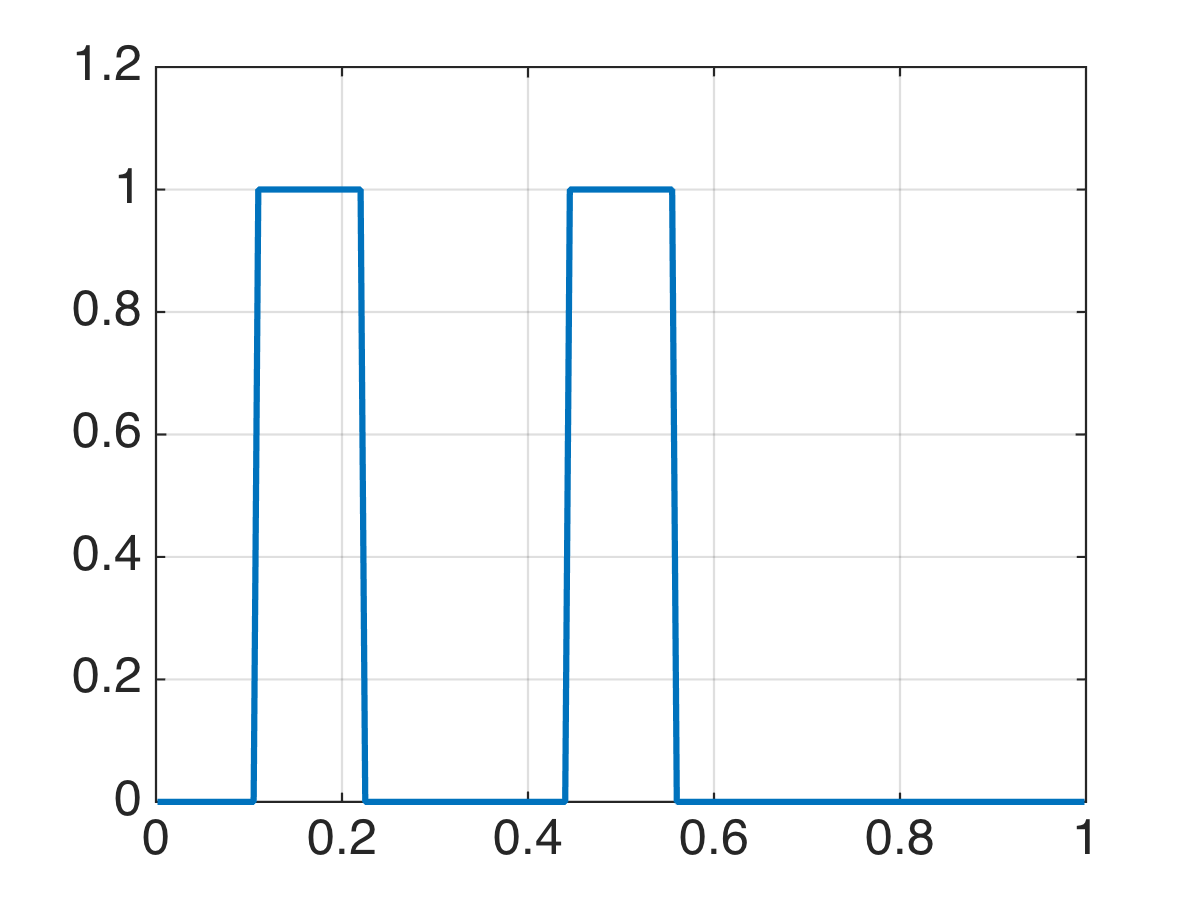}}
	\subfloat[$\alpha=10^{-1}$]{\includegraphics[width=0.24\textwidth]{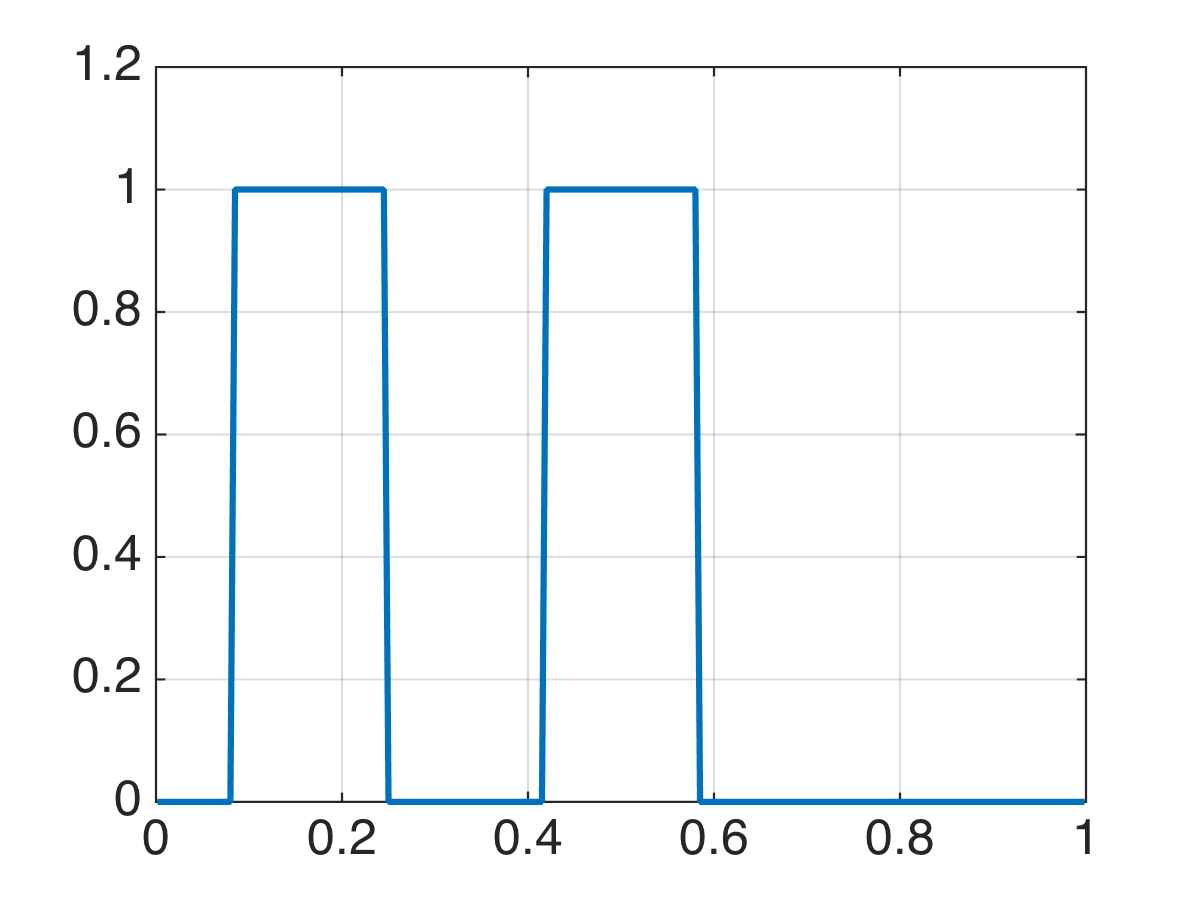}}
	\subfloat[$\alpha=10^4$]{\label{t4twoc}\includegraphics[width=0.24\textwidth]{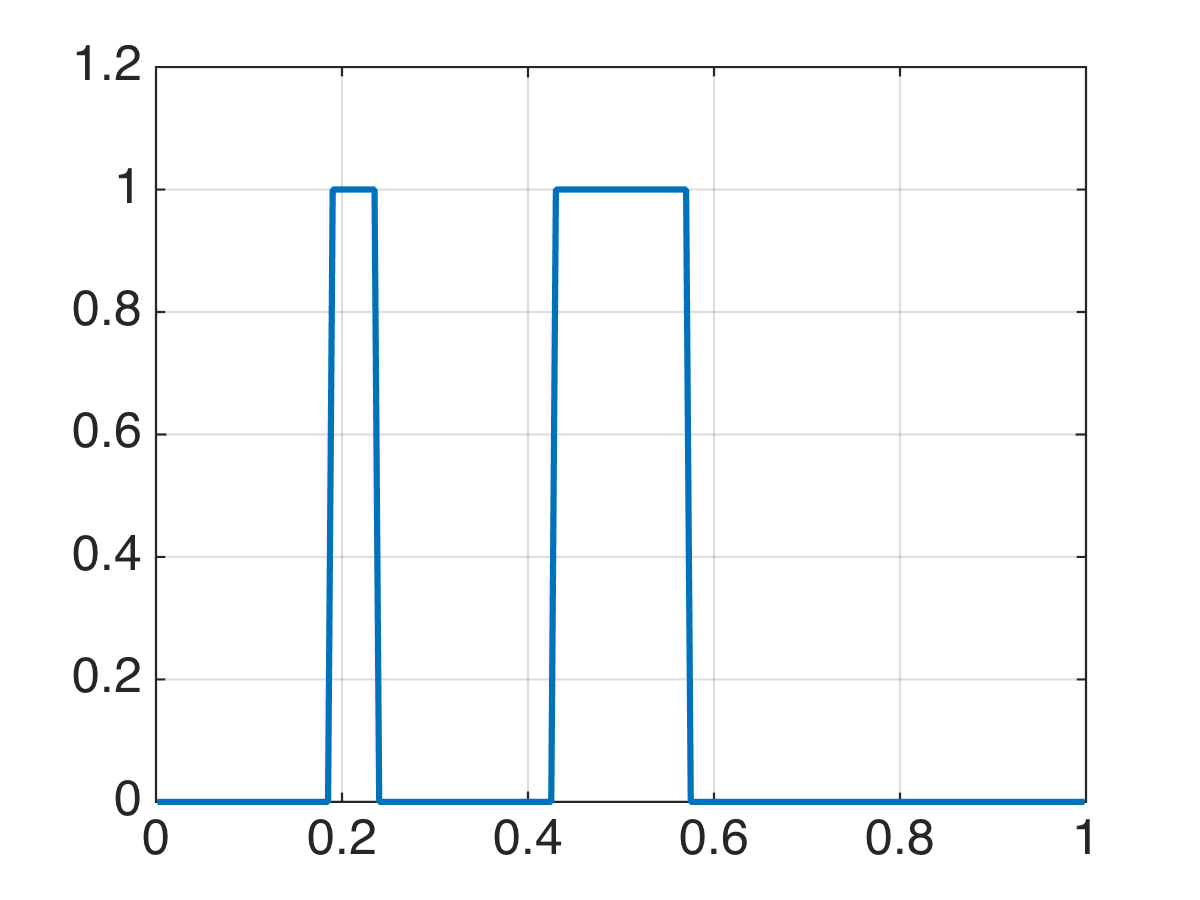}}	
	\caption{Test 4. (a) Initial condition $y_0(x)=\sin(3\pi x)^2\chi_{\{x<2/3\}}(x)$. (b) Optimal actuator for $\alpha=10^4$, without initialization via increasing penalization. (c) Optimal actuator for $\alpha=10^{-1}$, subsequently used in the quadratic penalty approach. (d) Optimal actuator for $\alpha=10^4$, via increasing penalization.}
	\label{fig:t21d}
\end{figure}

\paragraph{Test 5} We now turn our attention to the optimal actuator design for the worst-case scenario among all the initial conditions, i.e. the $\cJ_2$ setting. Results are presented in Figure \ref{fig:t31d} and Table \ref{tabtest3}. The worst-case scenario corresponds to the first eigenmode of the Riccati operator (Figure \ref{eigentest5}), which generates a two-component symmetric actuator (Figure \ref{twoct5}). This is only observed within the continuation approach. For a large value of $\alpha$ without initialisation, we obtain a suboptimal solution with a single component (last row of Table \ref{tabtest3}, Figure \ref{t5noinit}).

\begin{figure}[!h]
	\centering
	\subfloat[$\mathfrak X_2(\omega)$]{\label{eigentest5}\includegraphics[width=0.24\textwidth]{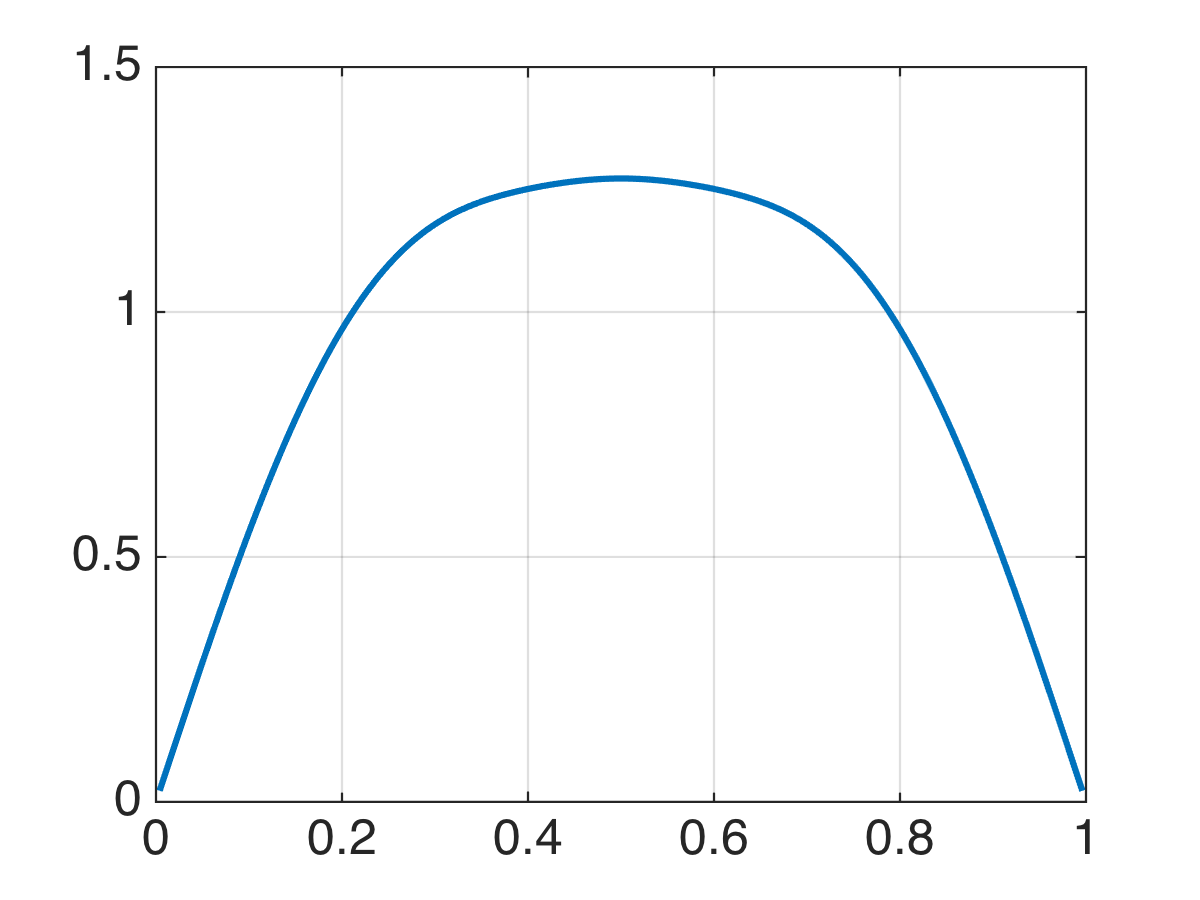}}
	\subfloat[$\alpha=10^3$, no init.]{\label{t5noinit}\includegraphics[width=0.24\textwidth]{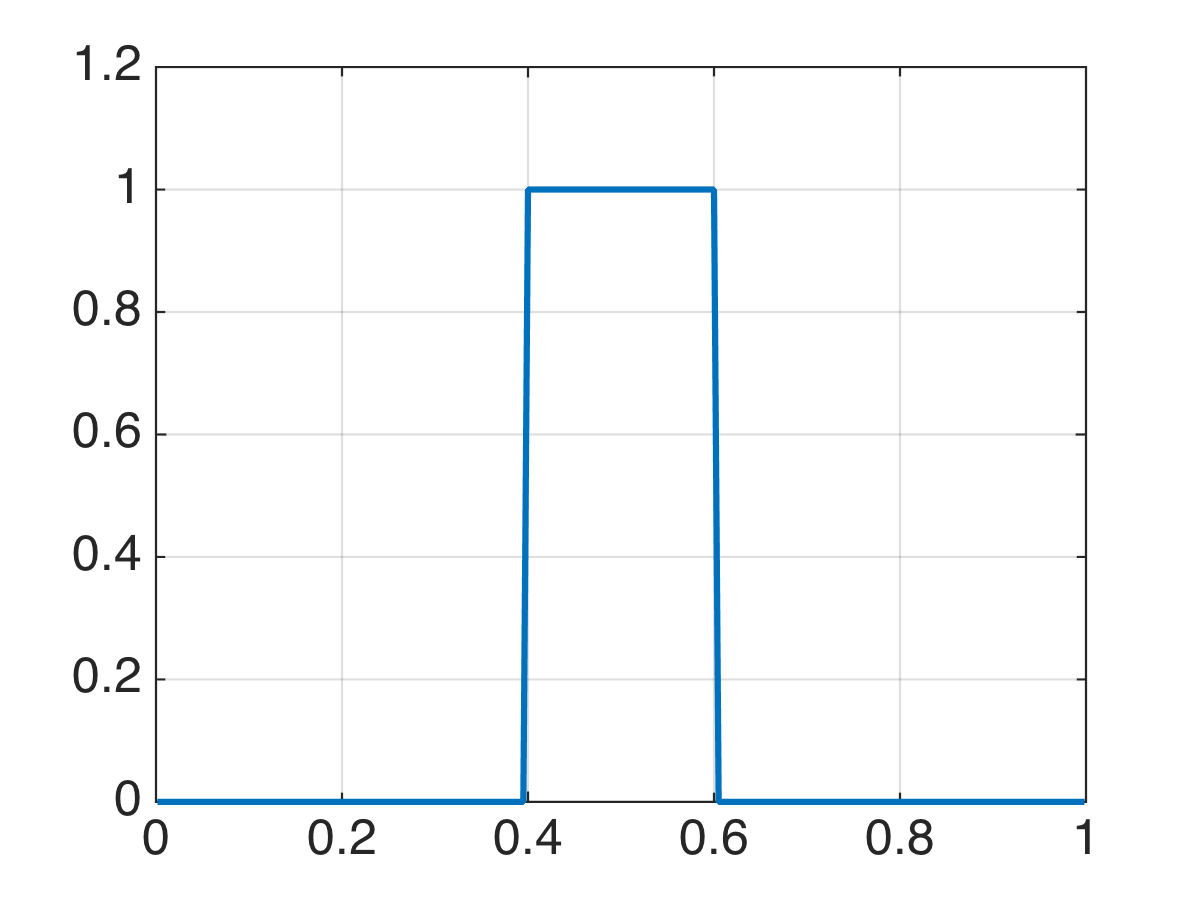}}
	\subfloat[$\alpha=10^{-1}$]{\includegraphics[width=0.24\textwidth]{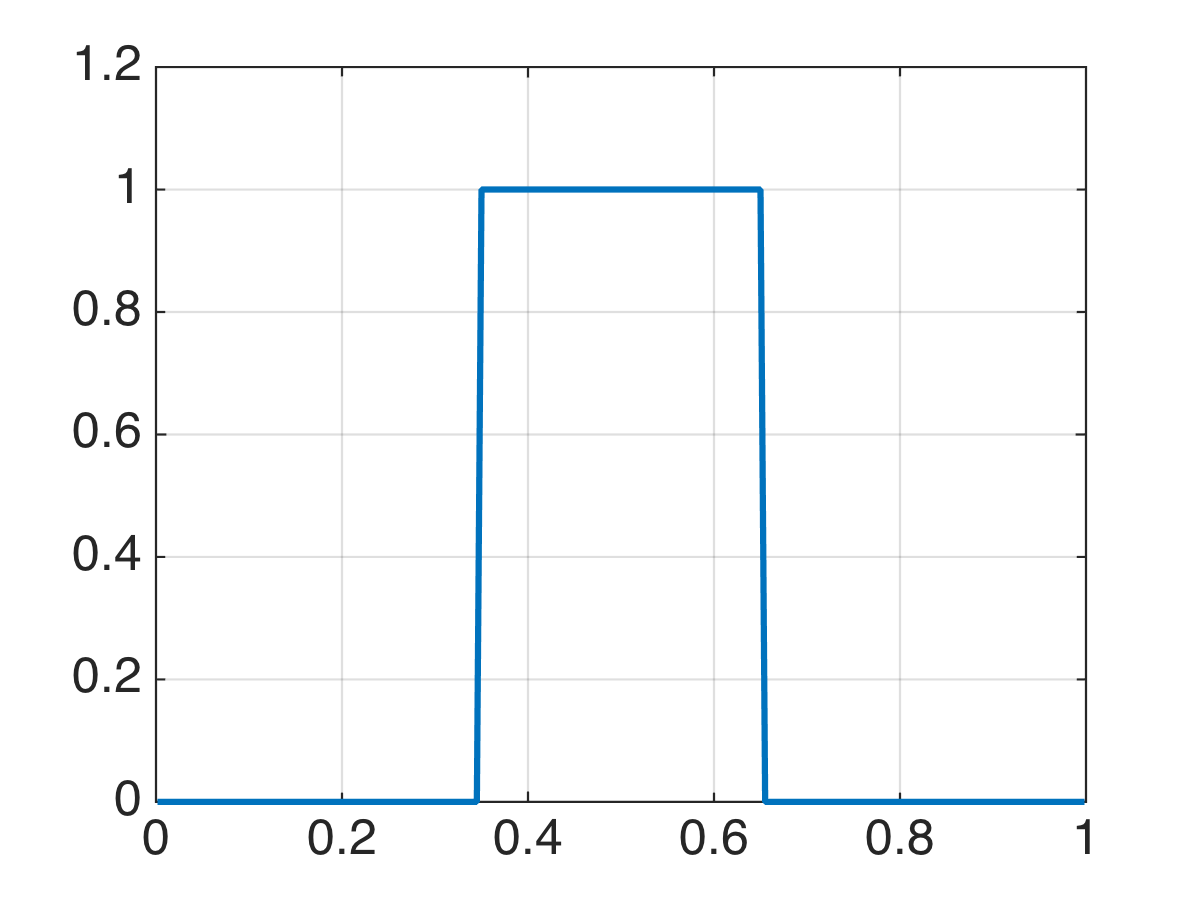}}
	\subfloat[$\alpha=10^{3}$]{\label{twoct5}\includegraphics[width=0.24\textwidth]{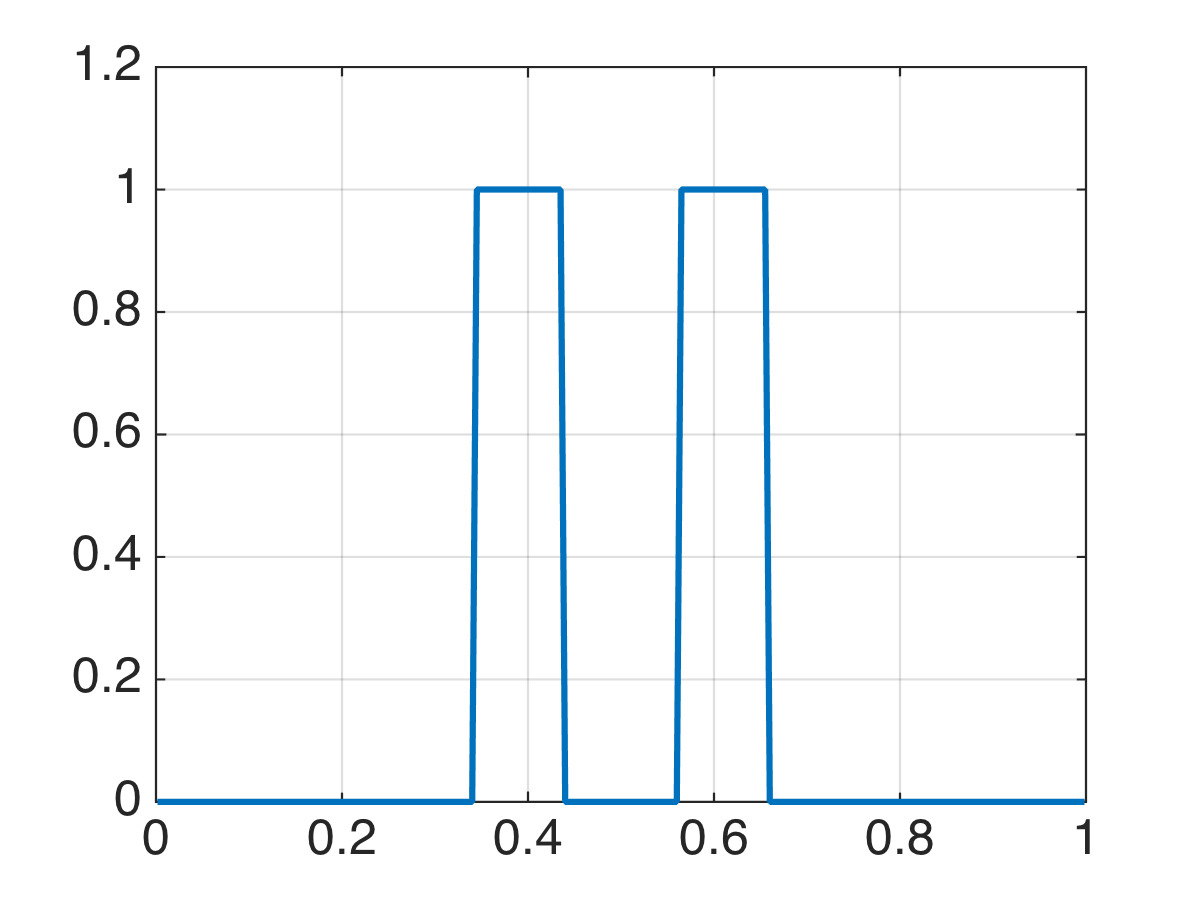}	}
	\caption{Test 5. (a) First eigenmode of the Riccati operator, which corresponds to the set $\mathfrak X_2(\omega)$. (b) Optimal actuator for $\alpha=10^3$, without initialization via increasing penalization. (c) Optimal actuator for $\alpha=10^{-1}$, subsequently used in the quadratic penalty approach. (d) Optimal actuator for $\alpha=10^3$, via increasing penalization.}
	\label{fig:t31d}
\end{figure}

\begin{table}[!h]
	\centering
	\setlength{\tabcolsep}{1mm}
	\begin{tabular}{cp{1.7cm}p{1.9cm}p{2.8cm}p{1.7cm}}
		\hline\\
		$\alpha$& $\cJ_2$&   $\cJ_2^{LQ}$ & $\cJ_2^{\alpha}(size)$&iterations\\
		\hline\\
		0.1   &0.402 & 0.401 &  1.1$\times10^{-3}$ (0.305) &307 \\
		1      &0.369 & 0.364 & 4.0$\times10^{-4}$ (0.22)  &225\\
		10    &0.343 & 0.342 &  1.0$\times10^{-3}$ (0.19) &228\\
		$10^{2}$  &0.352 & 0.342 &  1.0$\times10^{-2}$ (0.19) &226\\
		$10^{3}$  &0.442      & 0.342 &  0.1 (0.19) &226\\
		$10^{3}$*  &0.761       & 0.536 &  0.225 (0.215) &941\\
		\hline
	\end{tabular}
	\vskip 3mm
	\caption{Test 5. optimisation values for $\cJ_2$. Each row is initialized with the optimal actuator corresponding to the previous one, except for the last row with $\alpha=10^{3}$*. The reference size for the actuator is $0.2$ .}\label{tabtest3}
\end{table}

\paragraph{Test 6} As an extension of the capabilities of the proposed approach, we explore the $\cJ_2$ setting with space-dependent diffusion. For this test, the diffusion operator $\sigma\Delta y$ is rewritten as $div(\sigma(x)\nabla y)$, with $\sigma(x)=(1-\max(\sin(9\pi x),0))\chi_{\{x<0.5\}}(x)+10^{-3} $. Iterates of the continuation approach are presented in Table \ref{tabtest4}. Again, the lack of a proper initialization of Algortithm \ref{alg:topo} with a large value of $\alpha$ leads to a poor satisfaction of both the size constraint and the LQ performance, which is solved via the increasing penalty approach. A two-component actuator present in the area of smaller diffusion is observed in Figure \ref{twocompt6}.

\begin{table}[!h]
	\centering
	\setlength{\tabcolsep}{1mm}
	\begin{tabular}{cp{1.7cm}p{1.9cm}p{2.8cm}p{1.7cm}}
		\hline\\
		$\alpha$& $\cJ_2$&   $\cJ_2^{LQ}$ & $\cJ_2^{\alpha}(size)$&iterations\\
		\hline\\
		0.1   &1.792 & 1.743 &  4.97$\times10^{-2}$ (0.908) &194 \\
		1      &2.240 & 1.743 & 0.497 (0.908)  &228\\
		10    &4.734 & 4.462 &  0.272 (0.365) &225\\
		$10^{2}$  &3.134 & 3.071 &  6.25$\times10^{-2}$ (0.175) &538\\
		$10^{3}$  &1.023  & 0.998 &  0.025 (0.195) &226\\
		$10^{4}$  &1.248  & 0.998 &  0.250 (0.195) &226\\
		$10^{4}$*  &28.19   & 3.195 &  25.0 (0.25) &673\\
		\hline
	\end{tabular}
	\vskip 3mm
	\caption{Test 6. $\cJ_2$ values with space-dependent diffusion $\sigma(x)=(1-\max(\sin(9\pi x),0))\chi_{\{x<0.5\}}(x)+10^{-3} $. Each row is initialized with the optimal actuator corresponding to the previous one, except for the last row with $\alpha=10^{4}$*. The reference size for the actuator is $0.2$ .}\label{tabtest4}
\end{table}

\begin{figure}[!h]
	\centering
	\subfloat[$\mathfrak X_2(\omega)$]{\label{eigentest6}\includegraphics[width=0.24\textwidth]{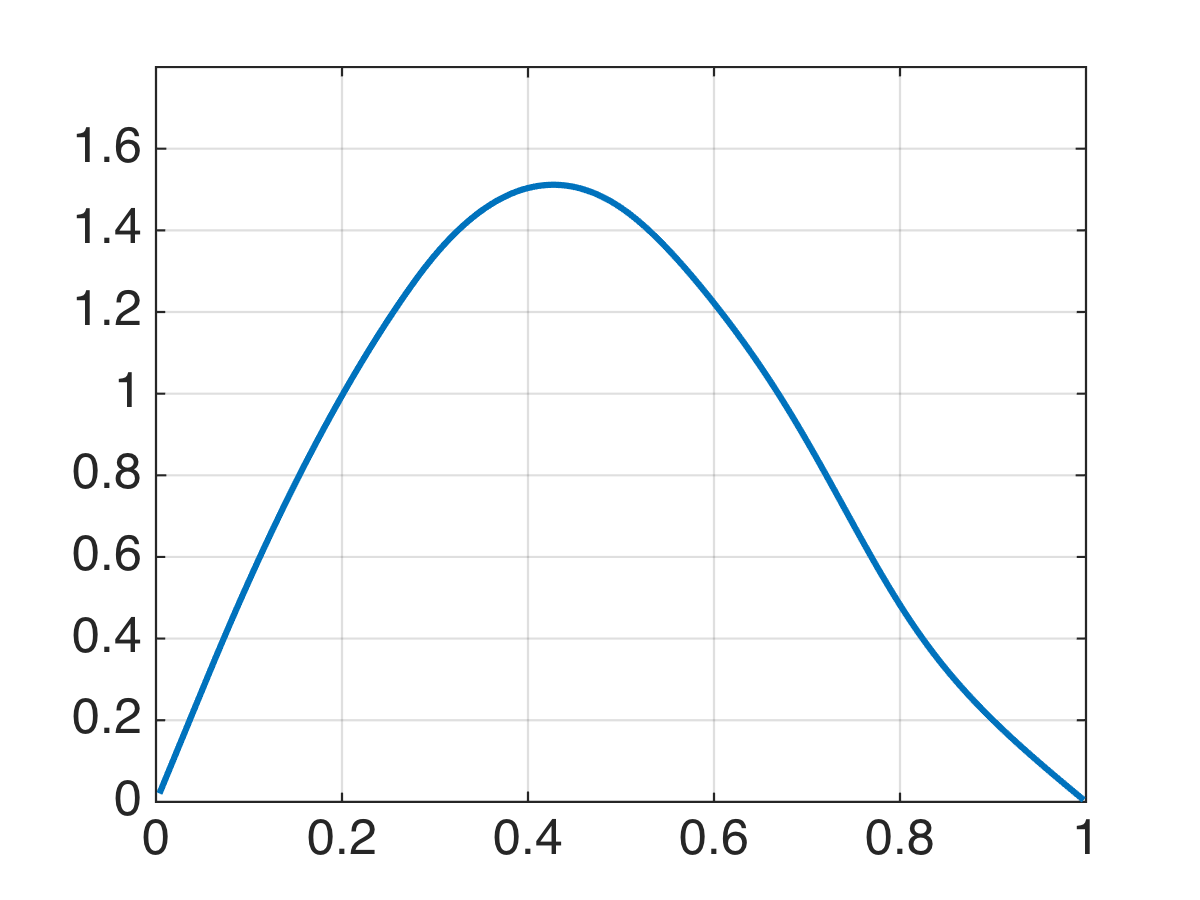}}
	\subfloat[$\sigma(x)$]{\label{difut6}\includegraphics[width=0.24\textwidth]{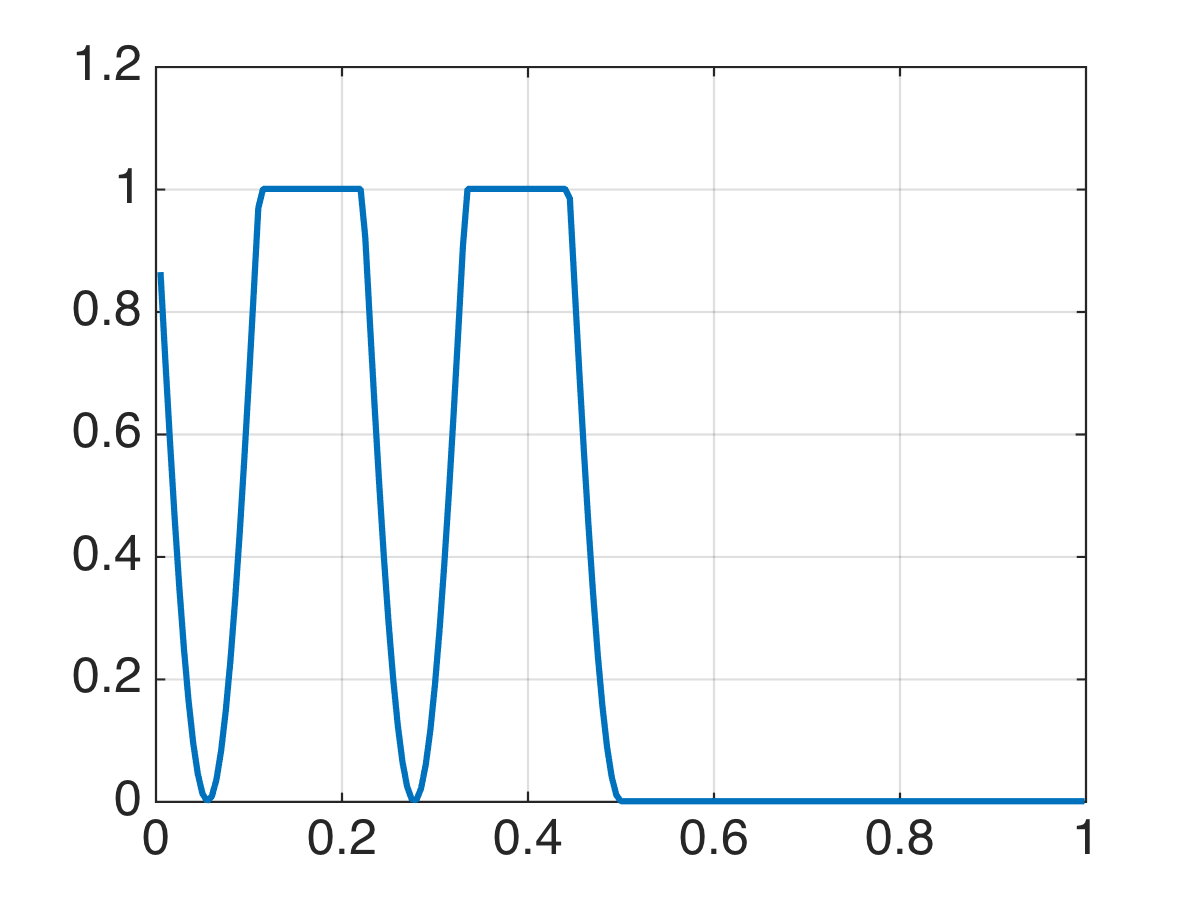}}
	\subfloat[$\alpha=0.1$]{\includegraphics[width=0.24\textwidth]{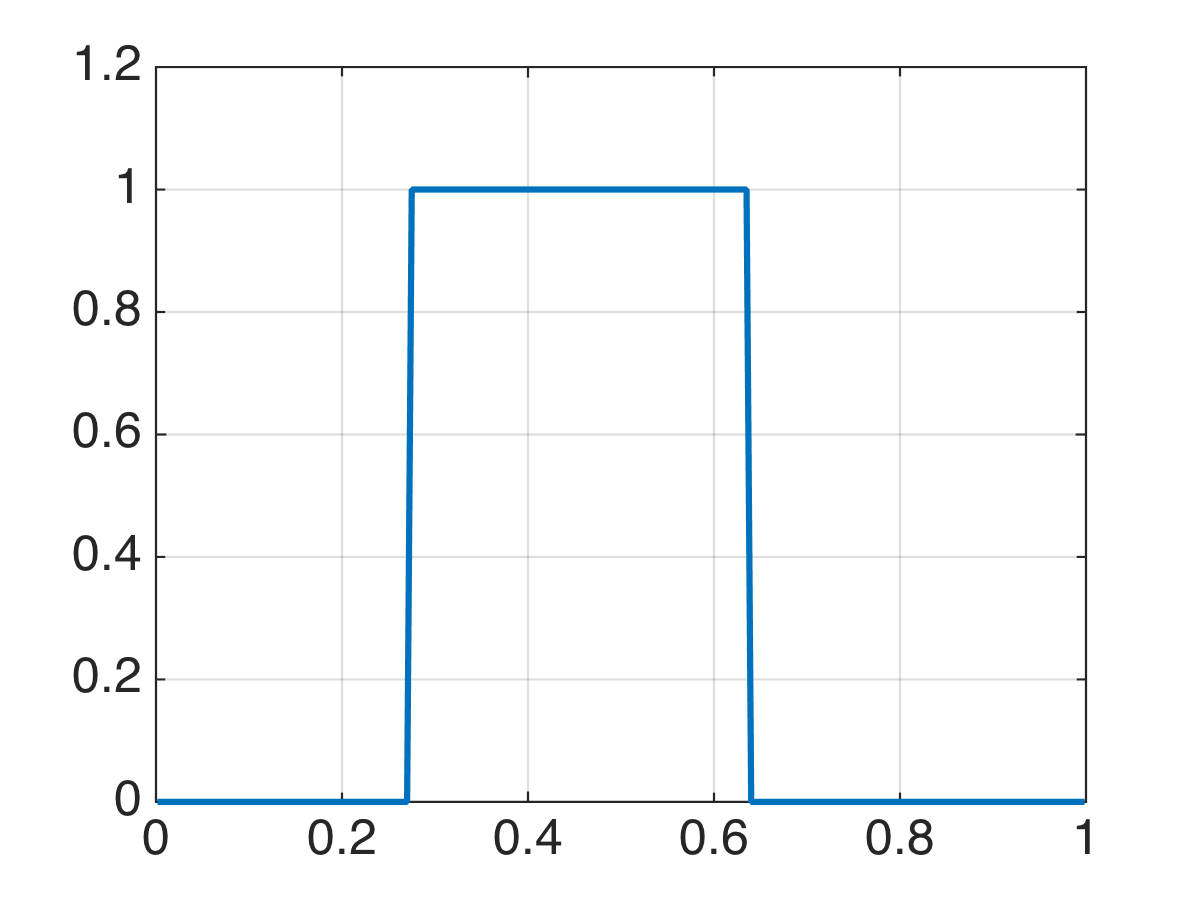}}
	\subfloat[$\alpha=10^4$]{\label{twocompt6}\includegraphics[width=0.24\textwidth]{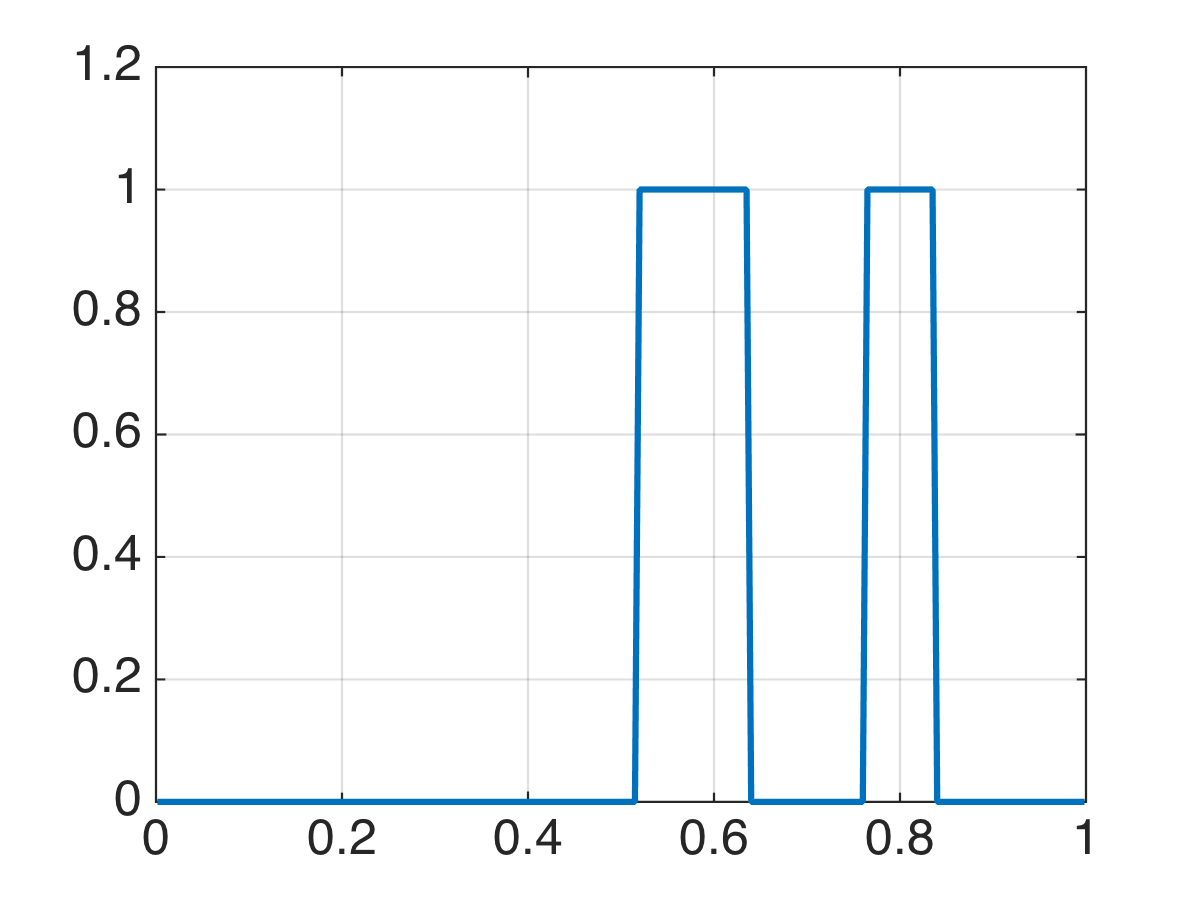}}	
	\caption{Test 6. (a) First eigenmode of the Riccati operator, which corresponds to the set $\mathfrak X_2(\omega)$. (b) space-dependent diffusion coefficient $\sigma(x)=(1-\max(\sin(9\pi x),0))\chi_{\{x<0.5\}}(x)+10^{-3} $. (c) Optimal actuator for $\alpha=10$, subsequently used in the quadratic penalty approach. (d) Optimal actuator for $\alpha=10^4$, via increasing penalization.}
	\label{fig:t41d}
\end{figure}

\subsection{Two-dimensional optimal actuator design} We now turn our attention into assessing the performance of Algorithm \ref{alg:topo} for two-dimensional actuator topology optimisation. While this problem is computationally demanding, the increase of degrees of freedom can be efficiently handled via modal expansions, as explained at the beginning of this Section. We explore both the $\cJ_1$ and $\cJ_2$ settings.
\paragraph{Test 7} This experiment is a direct extension of Test 3. We consider a unilaterally symmetric initial condition $y_0(x_1,x_2)=\max(\sin(4\pi(x_1-1/8)),0)^3\sin(\pi x_2)^3$, inducing a two-component actuator. The desired actuator size is $c=0.04$. The evolution of the actuator design for increasing values of the penalty parameter $\alpha$ is depicted in Figure \ref{fig:testj1}. We also study the closed-loop performance of the optimal shape.  For this purpose
the running cost associated to the optimal actuator is compared against an ad-hoc design, which consists of a cylindrical actuator of desired size placed in the center of the domain, see Figure \ref{fig:perfj1} . The closed-loop dynamics of the optimal actuator generate a stronger exponential decay compared to the uncontrolled dynamics and the ad-hoc shape.

\begin{figure}[!h]
	\centering
	\subfloat[$y_0(x_1,x_2)$]{\label{fig:a}\includegraphics[width=0.5\textwidth]{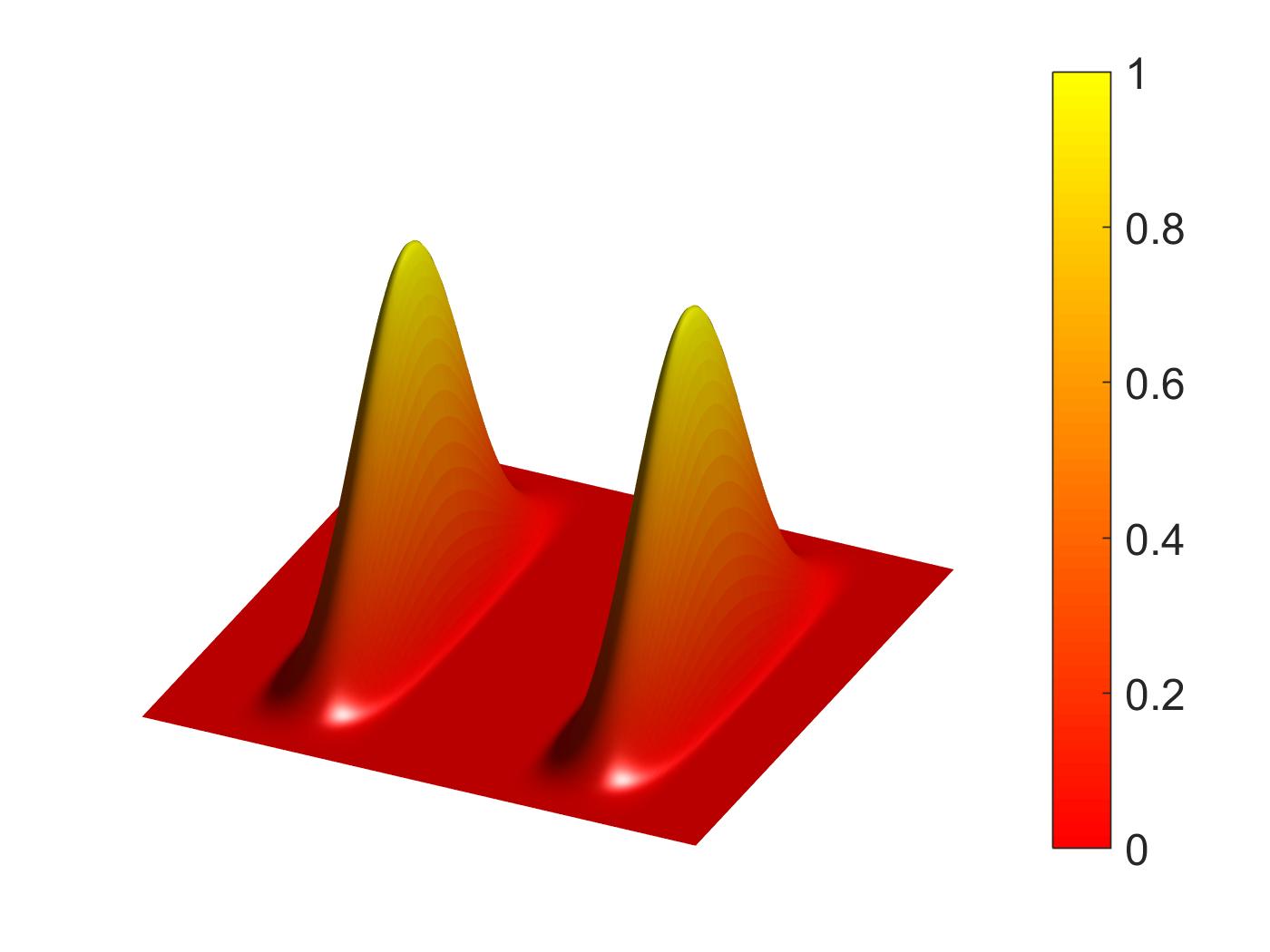}}
	\subfloat[$\psi_h^{n+1}$]{\includegraphics[width=0.5\textwidth]{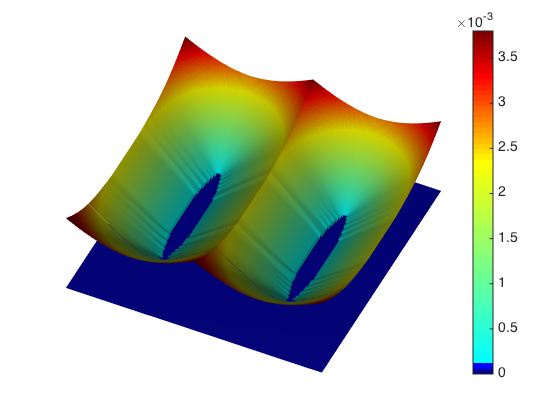}}\\
	\subfloat[$\alpha=0.1$]{\label{fig:b}\includegraphics[width=0.5\textwidth]{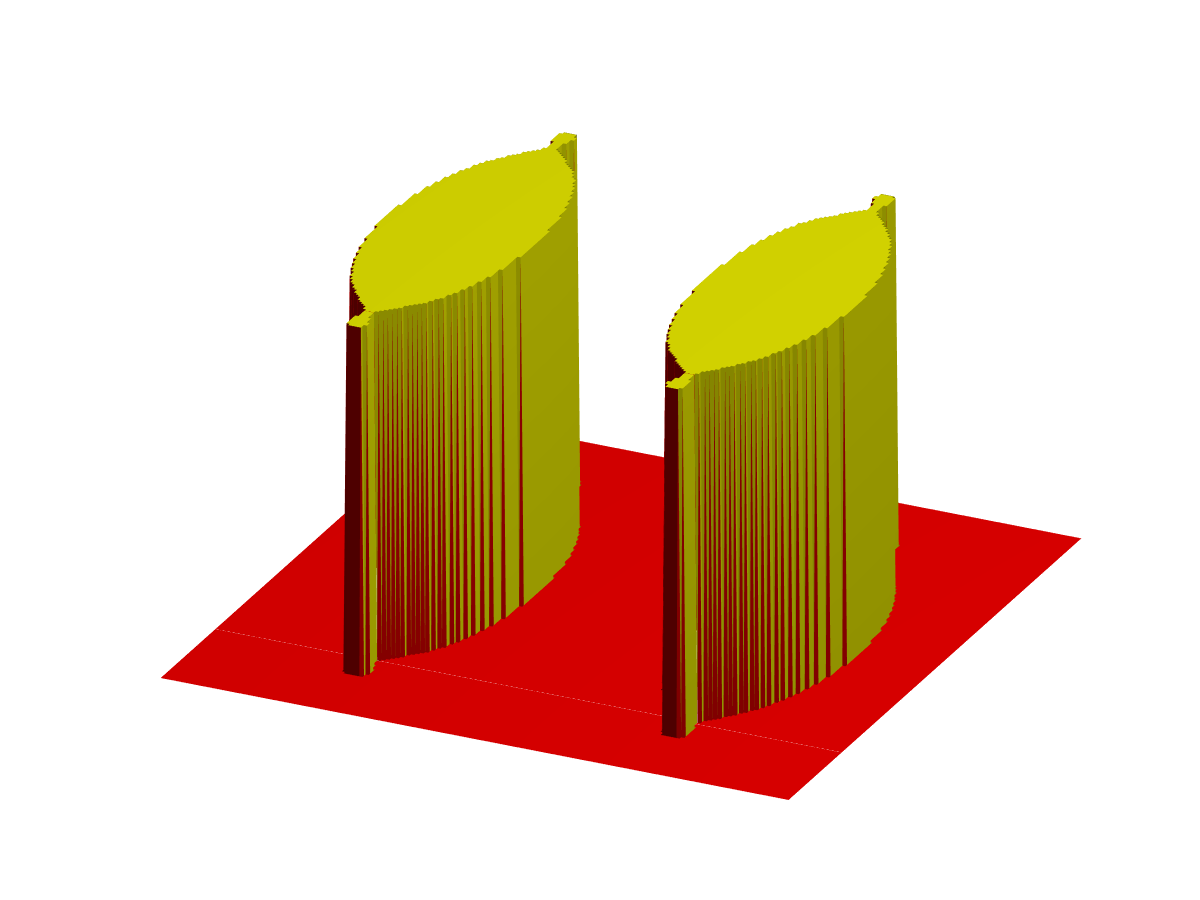}}
	\subfloat[$\alpha=1$]{\label{fig:c}\includegraphics[width=0.5\textwidth]{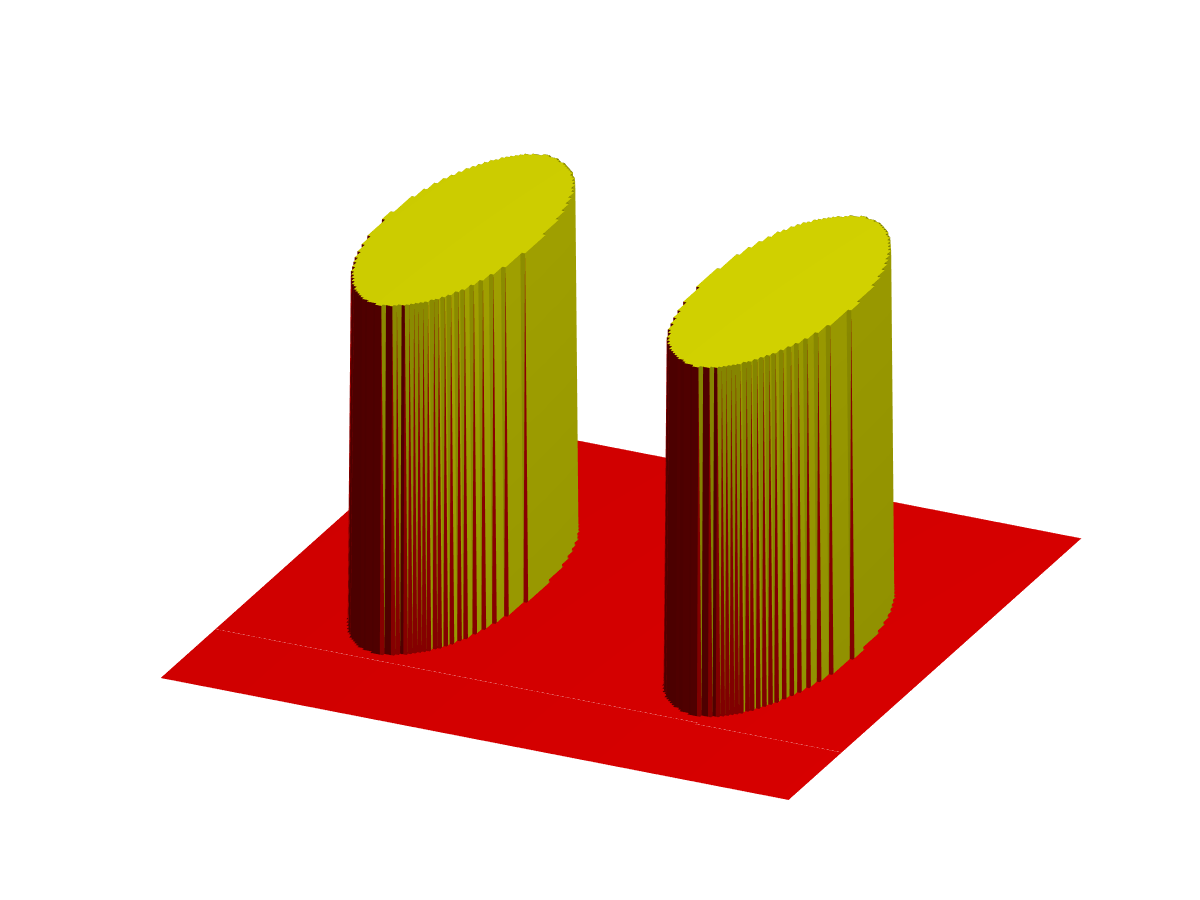}}\\
	\subfloat[$\alpha=1\times10^{2}$]{\label{fig:e}\includegraphics[width=0.5\textwidth]{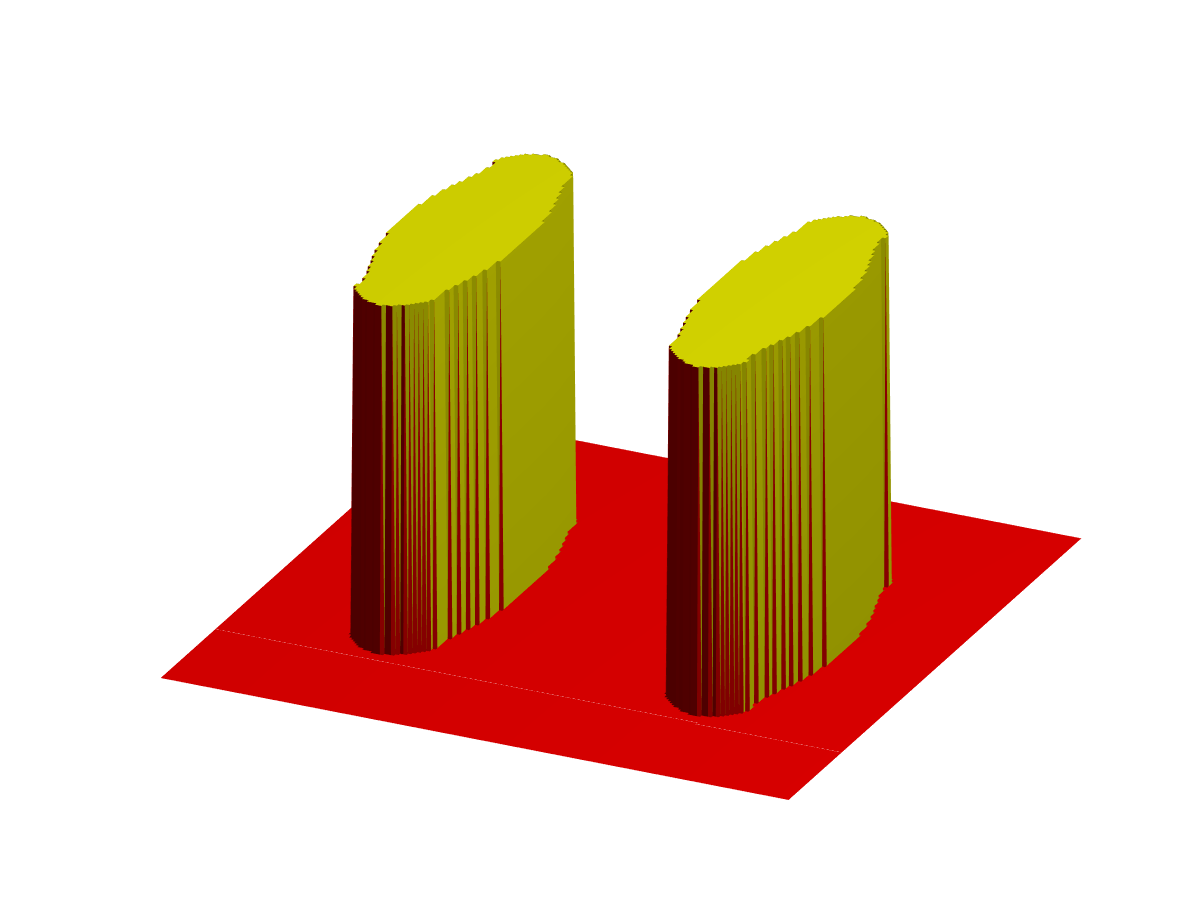}}
	\subfloat[$\alpha=10^{4}$]{\label{fig:g}\includegraphics[width=0.5\textwidth]{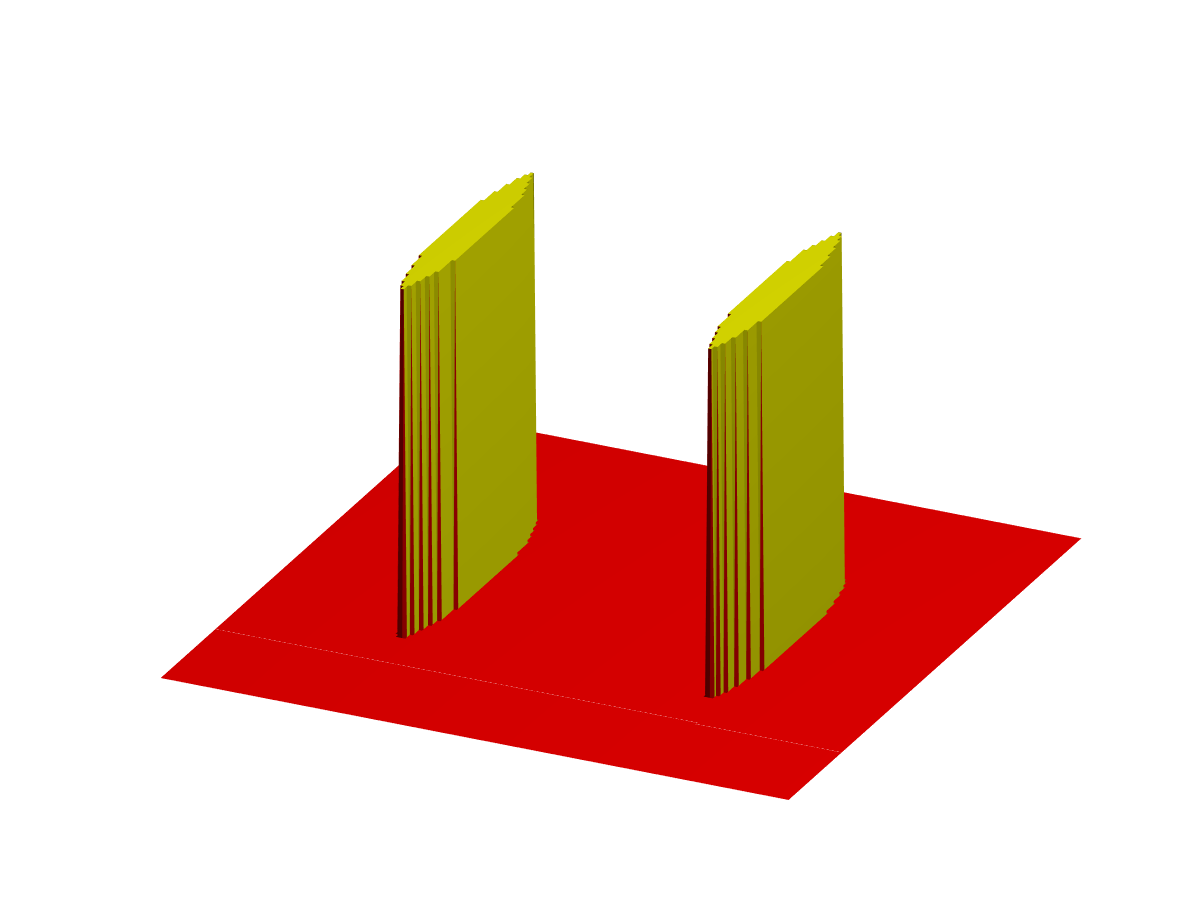}}	
	\caption{Test 7. (a) initial condition $y_0(x_1,x_2)=\max(\sin(4\pi(x_1-1/8)),0)^3\sin(\pi x_2)^3$ for $\cJ_1$ optimisation. (b) within the level-set method, the actuator is updated according to the zero level-set of the function $\psi_h^{n+1}$. (c) to (f) optimal actuators for different volume penalties.}
	\label{fig:testj1}
\end{figure}

\begin{figure}[!h]
	\centering
	\includegraphics[width=0.6\textwidth]{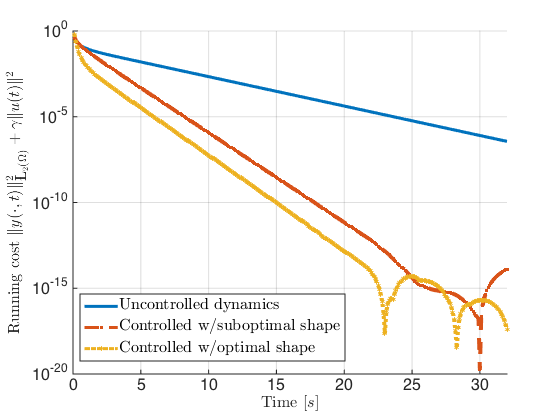}\\
	\caption{Test 7. Closed-loop performance for different shapes. The running cost in $\cJ_1$ is evaluated for uncontrolled dynamics ($u\equiv 0$), an ad-ho cylindrical actuator located in the center of the domain, and the optimal shape (Figure \ref{fig:g}). Closed-loop dynamics of the optimal shape decay faster. }
	\label{fig:perfj1}
\end{figure}

\paragraph{Test 8} In an analogous way as in Test 5, we study the optimal design problem associated to $\cJ_2$. The first eigenmode of the Riccati operator is shown in Figure \ref{fig:aj2}. The increasing penalty approach (Figs. \ref{fig:bj2} to \ref{fig:gj2}) shows a complex structure, with a hollow cylinder and four external components. The performance of the closed-loop optimal solution is analysed in Figure \ref{fig:perfj2}, with a considerably faster decay compared to the uncontrolled solution, and to the ad-hoc design utilised in the previous test.
\begin{figure}[!h]
	\centering
	\subfloat[$\mathfrak X_2(\omega)$]{\label{fig:aj2}\includegraphics[width=0.5\textwidth]{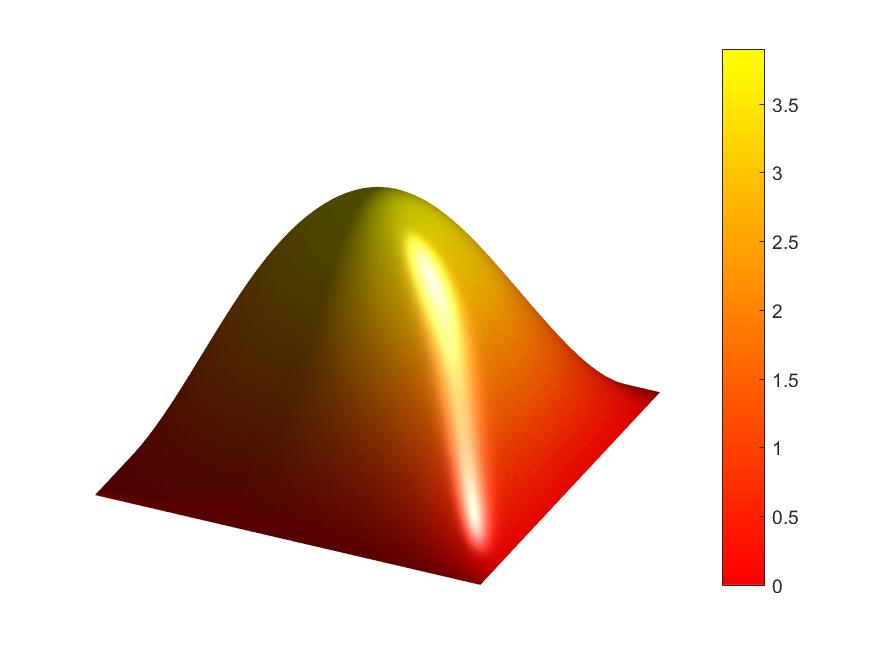}}
	\subfloat[$\psi_h^{n+1}$]{\includegraphics[width=0.5\textwidth]{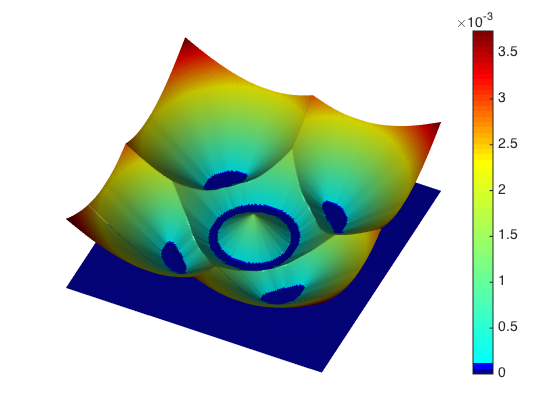}}\\
	\subfloat[$\alpha=0.1$]{\label{fig:bj2}\includegraphics[width=0.5\textwidth]{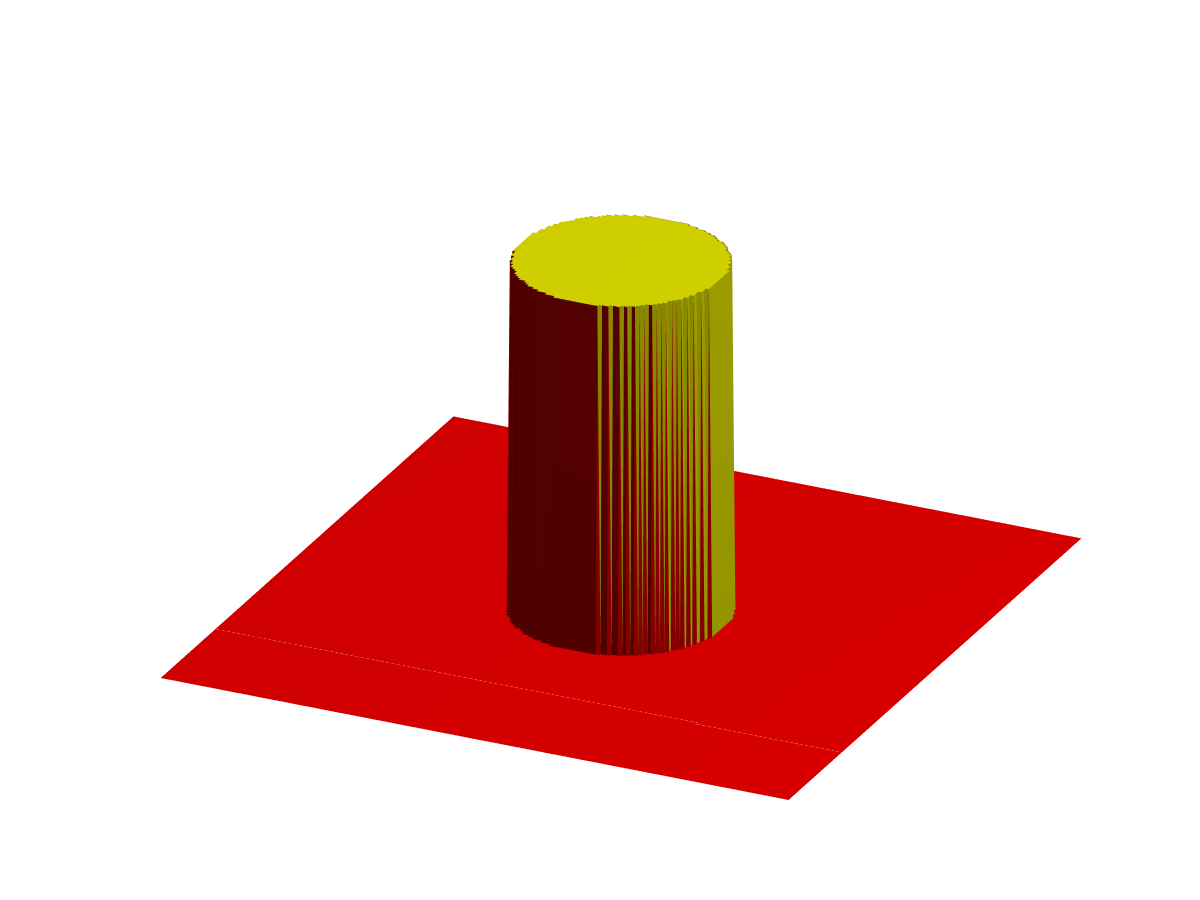}}
	\subfloat[$\alpha=10$]{\label{fig:dj2}\includegraphics[width=0.5\textwidth]{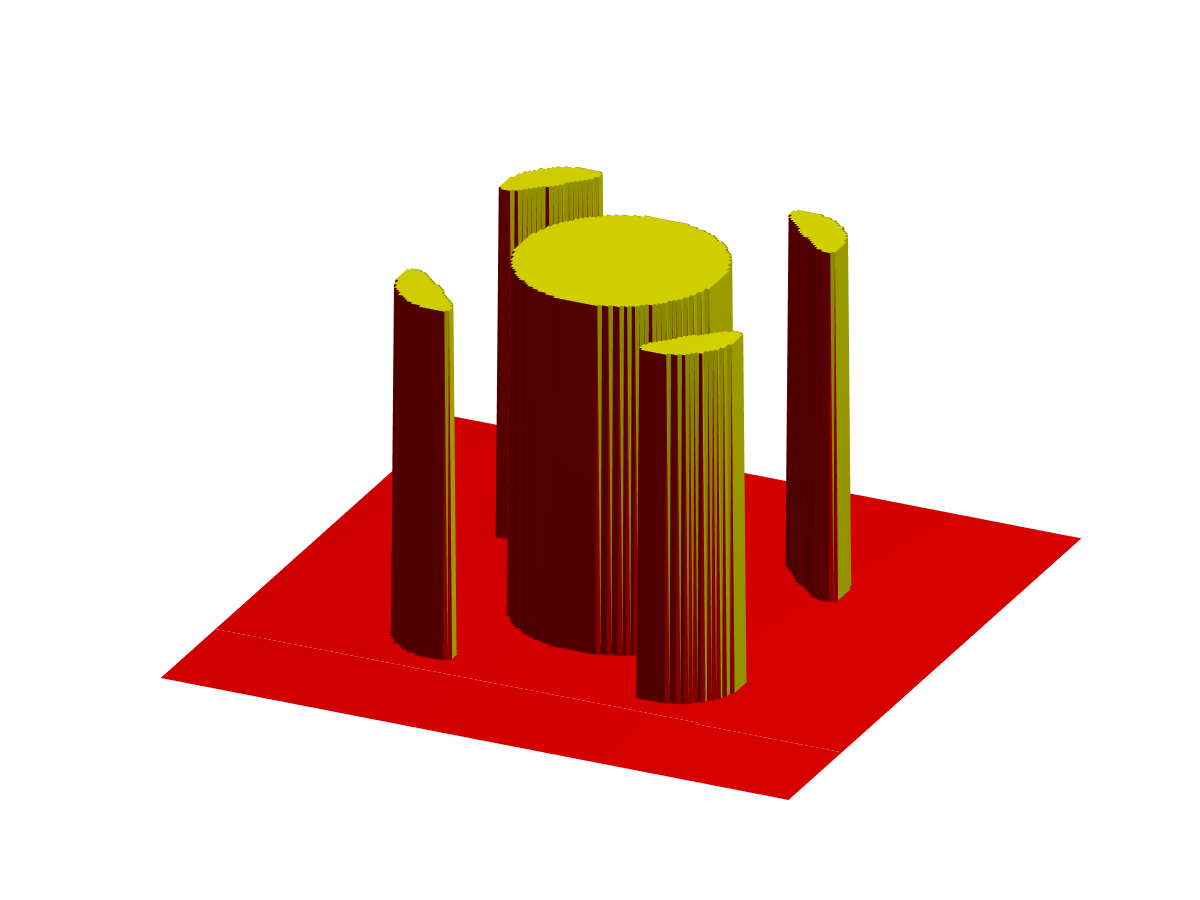}}\\
	\subfloat[$\alpha=10^{2}$]{\label{fig:ej2}\includegraphics[width=0.5\textwidth]{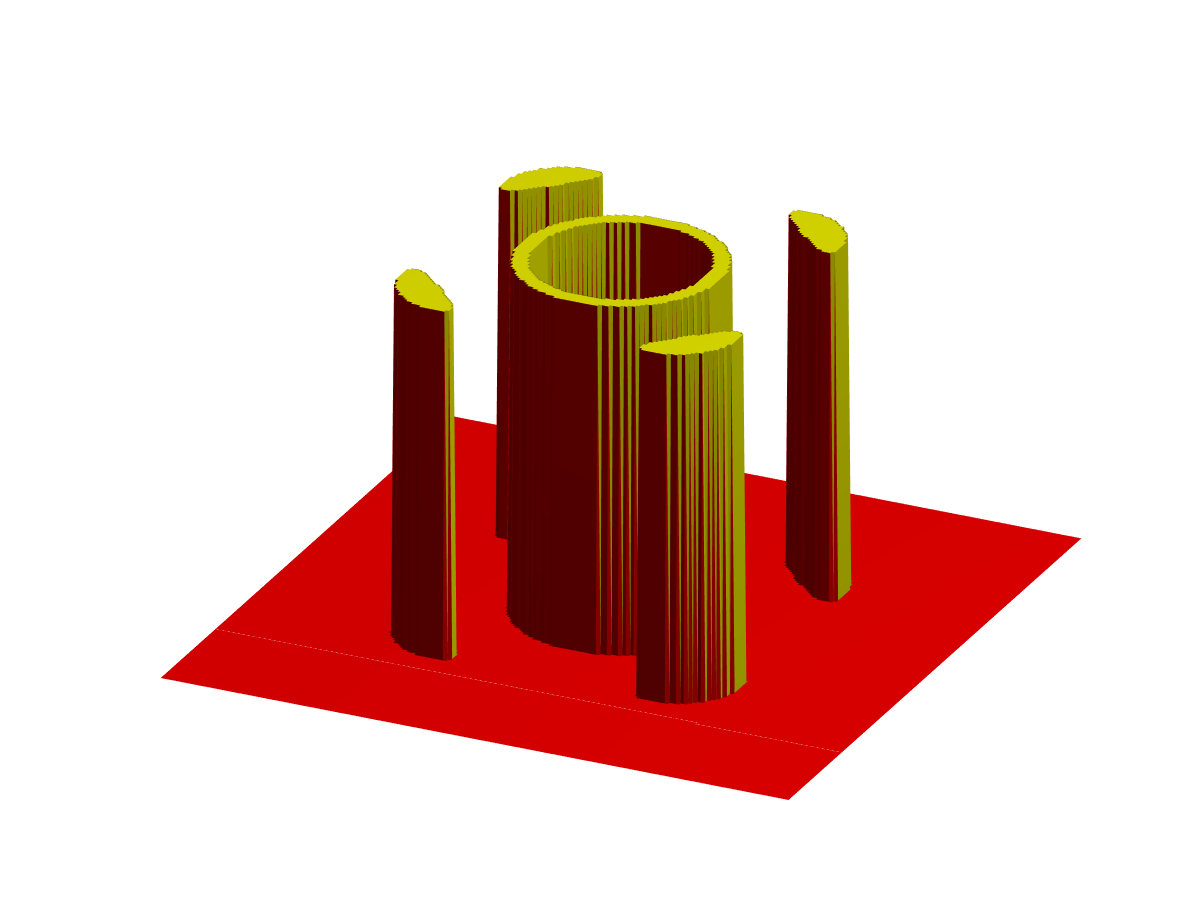}}
	\subfloat[$\alpha=10^{4}$]{\label{fig:gj2}\includegraphics[width=0.5\textwidth]{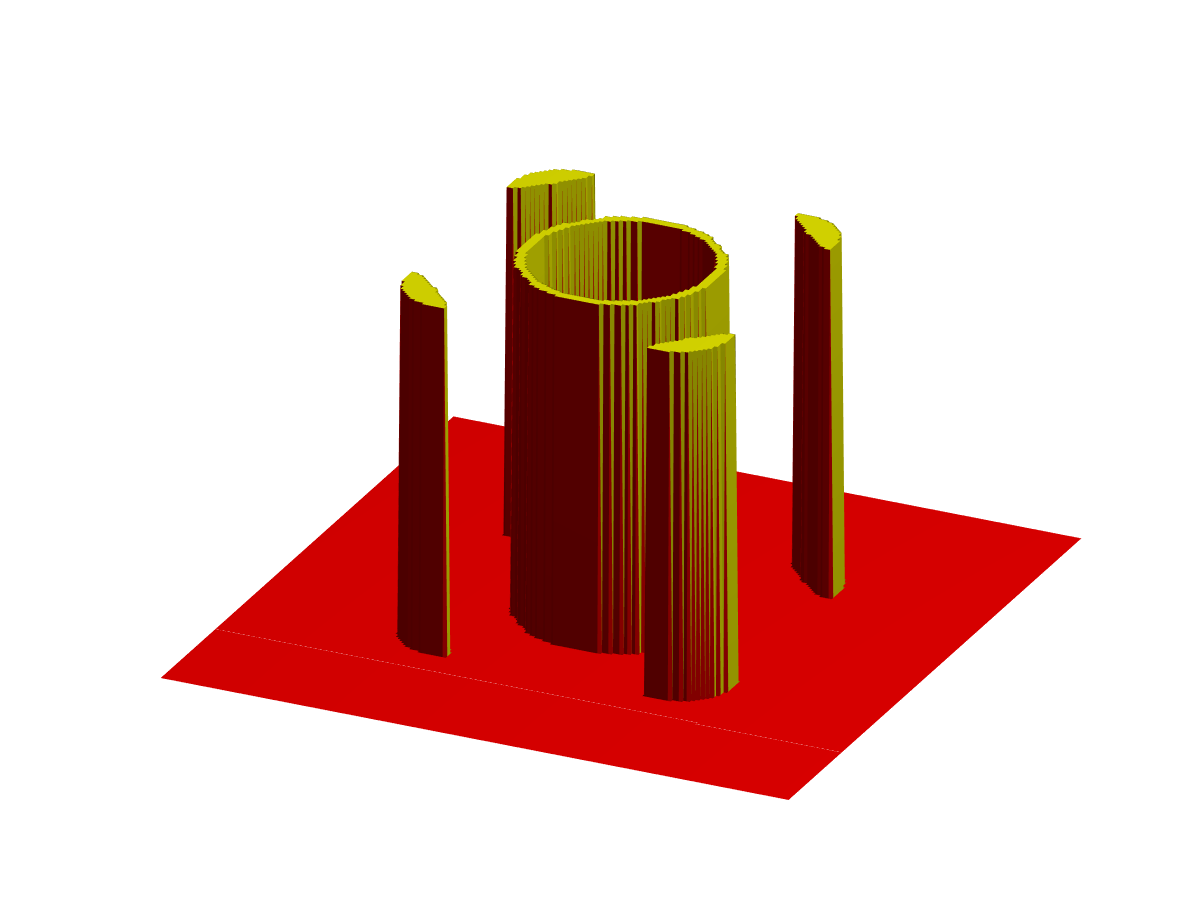}}	
	\caption{Test 8. (a) first eigenmode of the Riccati operator. (b) within the level-set method, the actuator is updated according to the zero level-set of the function $\psi_h^{n+1}$. (c) to (f) optimal actuators for different volume penalties.}
	\label{fig:testj2}
\end{figure}

\begin{figure}[!h]
	\centering
	\includegraphics[width=0.6\textwidth]{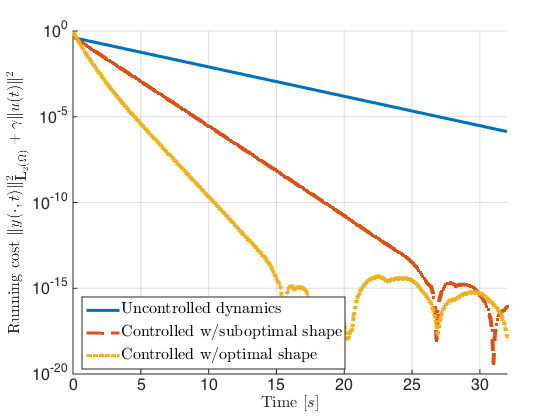}\\
	\caption{Test 8. Closed-loop performance for different shapes. The running cost in $\cJ_2$ is evaluated for uncontrolled dynamics ($u\equiv 0$), a suboptimal cylindrical actuator of size $c$ located in the center of the domain, and the optimal shape  with five components (Figure \ref{fig:gj2}). Closed-loop dynamics of the optimal shape decay faster. }
	\label{fig:perfj2}
\end{figure}

\section*{Concluding remarks} In this work we have developed an analytical and computational framework for optimisation-based actuator design. We derived shape and topological sensitivities formulas which account for the closed-loop performance of a linear-quadratic controller associated to the actuator configuration. We embedded the sensitivities into gradient-based and level-set methods to numerically realise the optimal actuators. Our findings seem to indicate that from a practical point of view, shape sensitivities are a good alternative whenever a certain parametrisation of the actuator is fixed in advance and only optimal position is sought. Topological sensitivities are instead suitable for optimal actuator design in a wider sense, allowing the emergence of nontrivial multi-component structures, which would be difficult to guess or parametrise a priori. This is a relevant fact, as most of the engineering literature associated to computational optimal actuator positioning is based on heuristic methods which strongly rely on experts' knowledge and tuning. Extensions concerning robust control design and semilinear parabolic equation are in our research roadmap.

\section*{Appendix}

\subsection*{Differentiability of maximum functions}

In order to prove Lemma~\ref{lem:maxmin} we recall the following
Danskin-type lemma.

Let $\mathfrak V_1$ be a nonempty set and let $\mathcal G:[0,\tau]\times \mathfrak V_1\rightarrow \R$ be a function, $\tau >0$. Introduce the function $g_1:[0,\tau]\rightarrow \R$,
\ben\label{eq:min_X}
g_1(t) := \sup_{x\in \mathfrak V_1} \mathcal G(t,x),
\een
and  let $\ell:[0,\tau]\rightarrow \R$ be any function such that $\ell(t)>0$ for $t\in (0,\tau]$ and $\ell(0)=0$.
We give sufficient conditions that guarantee that the limit
\ben
\frac{d}{d\ell} g_1(0^+) := \lim_{t\searrow 0} \frac{g_1(t)-g_1(0)}{\ell(t)}
\een
exists. For this purpose we introduce for each $t$ the set of maximisers
\ben\label{eq:set_Xt2}
\mathfrak V_1(t) = \{x^t\in \mathfrak V_1:\; \sup_{x\in \mathfrak V_1} \mathcal G(t,x) = \mathcal G(t,x^t) \}.
\een
The next lemma can be found with slight modifications in \cite[Theorem 2.1, p. 524]{DEZO11}.

\begin{lemma}\label{lem:danskin2}
	Let the following hypotheses be satisfied.
	\begin{itemize}
		\item[(A1)]
		\begin{itemize}
			\item[$(i)$]  For all $t$ in $[0,\tau]$  the set $\mathfrak V_1(t)$  is nonempty,
			\item[$(ii)$] the limit
			\ben
			\partial_\ell \mathcal G(0^+,x) := \lim_{t\searrow 0} \frac{\mathcal G(t,x)-\mathcal G(0,x) }{\ell(t)}
			\een  exists for all $x\in \mathfrak V_1(0)$.
		\end{itemize}
		\item[(A2)] For all real null-sequences $(t_n)$ in $(0,\tau]$ and all sequence $(x_{t_n})$ in $\mathfrak V_1(t_n)$, there exists a subsequence $(t_{n_k})$ of $(t_n)$, $(x_{t_{n_k}})$ in $\mathfrak V_1(t_{n_k})$  and $x_0$ in $\mathfrak V_1(0)$, such that
		\ben
		\lim_{k\rightarrow  \infty} \frac{\mathcal G( t_{n_k} ,x_{t_{n_k}}) -
			\mathcal G(0,x_{t_{n_k}})}{\ell(t_{n_k})} = \partial_\ell \mathcal G(0^+,x_0).
		\een	
	\end{itemize}
	Then $g_1$  is differentiable at $t=0^+$ with derivative
	\ben
	\frac{d}{d\ell}g_1(t)|_{t=0^+} = \max_{x\in  \mathfrak V_1(0) }\partial_\ell \mathcal G(0^+,x).
	\een
\end{lemma}

\subsection*{Proof of Lemma~\ref{lem:maxmin}}
Our strategy is to prove Lemma~\ref{lem:maxmin} by applying \\
 Lemma~\ref{lem:danskin2} to the function $\mathcal G(t,y):= \inf_{x\in \mathfrak V}  G(t,x,y)$ with $\mathfrak V_1 := \mathfrak V$. This will show that $g(t) := \sup_{y\in \mathfrak V} \mathcal G(t,y)$ is right-differentiable at $t=0^+$. By construction Assumption~(A0) of Lemma~\ref{lem:maxmin} is satisfied.

Step 1: For every $t\in [0,\tau]$ and $y\in \mathfrak V$ we have $\mathcal G(t,y) = G(t,x^{t,y},y)$. Hence
\ben\label{eq:difference_A1_ge}
\begin{split}
	\mathcal G(t,y)-\mathcal G(0,y) = & G(t,x^{t,y},y) - G(0,x^{0,y},y) \\
	& = G(t,x^{t,y},y) - G(0,x^{t,y},y) + \underbrace{G(0,x^{t,y},y)- G(0,x^{0,y},y)}_{\ge 0} \\
	& \ge G(t,x^{t,y},y) - G(0,x^{t,y},y)
\end{split}
\een
and similarly
\ben\label{eq:difference_A1_le}
\begin{split}
	\mathcal G(t,y)-\mathcal G(0,y) = & G(t,x^{t,y},y) - G(0,x^{0,y},y) \\
	& = \underbrace{G(t,x^{t,y},y) - G(t,x^{0,y},y)}_{\le 0} + G(t,x^{0,y},y)- G(0,x^{0,y},y) \\
	& \le  G(t,x^{0,y},y)- G(0,x^{0,y},y).
\end{split}
\een
Therefore using Assumption~(A2) of Lemma~\ref{lem:maxmin} we obtain from \eqref{eq:max_min_a1} and \eqref{eq:max_min_a2}
\ben
\liminf_{t\searrow 0} \frac{\mathcal G(t,y)-\mathcal G(0,y)}{\ell(t)} \ge  \partial_\ell G(0^+,x^{0,y},y) \ge \limsup_{t\searrow 0} \frac{\mathcal G(t,y)-\mathcal G(0,y)}{\ell(t)}.
\een
Hence Assumption~(A1) of Lemma~\ref{lem:danskin2} is satisfied.

Step 2: For every $t\in [0,\tau]$ and $y^t\in \mathfrak V(t)$ we have
$\mathcal G(t,y^t) = G(t,x^{t,y^t},y^t)$ and hence
\ben\label{eq:difference_ge}
\begin{split}
	\mathcal G(t,y^t)-\mathcal G(0,y^t)  = &  G(t,x^{t,y^t},y^t) - G(0,x^{0,y^t},y^t) \\
	= & G(t,x^{t,y^t},y^t) - G(0,x^{t,y^t},y^t) + \underbrace{G(0,x^{t,y^t},y^t) - G(0,x^{0,y^t},y^t)}_{\ge 0}\\
	\ge & G(t,x^{t,y^t},y^t) - G(0,x^{t,y^t},y^t)
\end{split}
\een
and similarly
\ben\label{eq:difference_le}
\begin{split}
	\mathcal G(t,y^t)-\mathcal G(0,y^t)  = & \underbrace{G(t,x^{t,y^t},y^t) - G(t,x^{0,y^t},y^t)}_{\le 0} + G(t,x^{0,y^t},y^t) - G(0,x^{0,y^t},y^t)\\
	\le & G(t,x^{0,y^t},y^t) - G(0,x^{0,y^t},y^t).
\end{split}
\een
Thanks to Assumption (A3) of Lemma~\ref{lem:maxmin} For all real null-sequences $(t_n)$ in $(0,\tau]$ and all sequences $(y^{t_n})$,  $y^{t_n}\in \mathfrak V(t_n)$, there exists a subsequence $(t_{n_k})$ of $(t_n)$, $(y^{t_{n_k}})$ of $(y^{t_n})$, and $y^0$ in $\mathfrak V(0)$, such that
\ben\label{eq:limit1}
\lim_{k\rightarrow  \infty} \frac{G( t_{n_k} ,x^{t_{n_k},y^{t_{n_k}}}, y^{t_{n_k}} ) -
	G(0,x^{t_{n_k},y^{t_{n_k}}}, y^{t_{n_k}})}{\ell(t_{n_k})} = \partial_\ell G(0^+,x^{0,y^0},y^0)
\een	
and
\ben\label{eq:limit2}
\lim_{k\rightarrow  \infty} \frac{G( t_{n_k} ,x^{0,y^{t_{n_k}}}, y^{t_{n_k}} ) -
	G(0,x^{0,y^{t_{n_k}}}, y^{t_{n_k}})}{\ell(t_{n_k})} = \partial_\ell G(0^+,x^{0,y^0},y^0).
\een
Hence choosing $t=t_{n_k}$ in \eqref{eq:difference_ge} we obtain
\ben\label{eq:ge}
\begin{split}
	&\liminf_{k\to \infty}\frac{ \mathcal G(t_{n_k},y^{t_{n_k}})-\mathcal G(0,y^{t_{n_k}}) }{\ell(t_{n_k})} \\
	 &  \stackrel{\eqref{eq:difference_ge}}{\ge}   \liminf_{k\rightarrow  \infty} \frac{G( t_{n_k} ,x^{t_{n_k},y^{t_{n_k}}}, y^{t_{n_k}} ) -
		G(0,x^{t_{n_k},y^{t_{n_k}}}, y^{t_{n_k}})}{\ell(t_{n_k})} \\
	& \stackrel{\eqref{eq:limit1}}{=}\partial_\ell G(0^+,x^{0,y^0},y^0)
\end{split}
\een
and similarly choosing $t=t_{n_k}$ in \eqref{eq:difference_le} we get
\ben\label{eq:le}
\begin{split}
	&\limsup_{k\to \infty}\frac{ \mathcal G(t_{n_k},y^{t_{n_k}})-\mathcal G(0,y^{t_{n_k}}) }{\ell(t_{n_k})}  \\
	& \stackrel{\eqref{eq:difference_le}}{\le}   \limsup_{k\rightarrow  \infty} \frac{G( t_{n_k} ,x^{0,y^{t_{n_k}}}, y^{t_{n_k}} ) -
		G(0,x^{0,y^{t_{n_k}}}, y^{t_{n_k}})}{\ell(t_{n_k})} \\
	& \stackrel{\eqref{eq:limit2}}{=}\partial_\ell G(0^+,x^{0,y^0},y^0).
\end{split}
\een
Combining \eqref{eq:ge} and \eqref{eq:le} we conclude that
\ben
\lim_{k\to \infty}\frac{ \mathcal G(t_{n_k},y^{t_{n_k}})-\mathcal G(0,y^{t_{n_k}}) }{\ell(t_{n_k})}  = \partial_\ell G(0^+,x^{0,y^0},y^0),
\een
which is precisely Assumption (A2) of Lemma~\ref{lem:danskin2}.

Step 1 and Step 2 together show that Assumptions (A1) and (A2) of Lemma~\ref{lem:danskin2} are satisfied and this finishes the proof.

\subsection*{Acknowledgement}
The authors gratefully  acknowledge the fact that Prof. D. Wachsmuth pointed out that the original version of Theorem~\ref{thm:diff_G_top} contained a mistake. 

\bibliographystyle{siamplain}
\bibliography{refs}
\end{document}